\documentclass[10pt]{article}

\usepackage{sbc-template} 
\usepackage{graphicx,url}
\usepackage{url}
\usepackage[brazil]{babel} 
\usepackage[utf8]{inputenc} 
\usepackage[T1]{fontenc}
\usepackage[normalem]{ulem}
\usepackage[hidelinks]{hyperref}

\usepackage[square, numbers]{natbib}
\usepackage{amssymb} 
\usepackage{mathalfa} 
\usepackage{algorithm} 
\usepackage{algpseudocode} 
\usepackage[table]{xcolor}
\usepackage{array}
\usepackage{titlesec}
\usepackage{mdframed}
\usepackage{listings}
\usepackage{mathrsfs}
\usepackage{amsmath}
\usepackage{booktabs}
\usepackage{amsthm}
\usepackage{bigints}

\usepackage{fancyhdr}
\newtheorem{theorem}{Theorem} [section]
\newtheorem{proposition}{Proposition} [section]
\newtheorem{lemma}{Lemma} [section]
\newtheorem{corollary}{Corollary} [section]
\newtheorem{remark}{Remark} [section]

\DeclareMathOperator{\sgn}{sgn}

\numberwithin{equation}{section}

\newcolumntype{L}[1]{>{\raggedright\let\newline\\\arraybackslash\hspace{0pt}}m{#1}}
\newcolumntype{C}[1]{>{\centering\let\newline\\\arraybackslash\hspace{0pt}}m{#1}}
\newcolumntype{R}[1]{>{\raggedleft\let\newline\\\arraybackslash\hspace{0pt}}m{#1}}

\usepackage[nolist,nohyperlinks]{acronym}

\title{Well-Posedness of the Schrödinger - Intermediate Long Wave system}

\author{Diego F. Correa-Casta\~neda}
\address{Instituto de Matematica Pura e Aplicada - IMPA\\Rio de Janeiro 22460-320, Brazil
\email{diego.castaneda@impa.br}
    }

\begin{document} 
	
	\maketitle

	\begin{resumo} 
The system coupled by the Schrödinger equation and the Intermediate Long Wave (ILW) equation is a particular model describing the interaction of long and short waves. 
We prove local and global well-posedness of the initial value problem associated with this system in Sobolev spaces of low regularity. Our approach relies on the use of energy estimates, Bourgain spaces, and use of Tao's gauge transformation.   \end{resumo}

\section{Introduction}

We consider the following coupled system:
    \begin{equation}
        \label{Schr-ILW} 
        \begin{cases}
             iu_t + u_{xx}= \alpha uv+\gamma\,|u|^2u  \\ 
             v_t-\nu\,\mathscr{G}_{\delta}v_{xx}+\rho\, vv_{x}=\beta\left(|u|^2\right)_x
        \end{cases}
        \,\,\,\,\,\,(t,x) \in \mathbb{R}^2,
    \end{equation}
where $u(t,x)$ is a complex-valued function, $v(t,x)$ is a real-valued function, and  $\alpha, \beta, \gamma, \nu$ and $\rho$ are real constants. The parameter $\delta>0$ characterizes the depth of the fluid, and $\mathscr{G}_{\delta}$ is the Fourier multiplier defined by the symbol $ -i\big(\coth{(\delta\xi)}-\frac{1}{\delta\xi}\big)$.  \\ \\ This system is a realization of the general framework describing interactions between long waves (L) and short waves (S) under weakly coupled nonlinearities: \begin{equation} \begin{cases} \label{general} i S_t + i c_S S_x + S_{xx} = \alpha S L + \gamma |S|^2 S \\ L_t + c_L L_x + \nu P(D_x) L + \lambda (L^2)_{x} = \beta (|S|^2)_x. \end{cases} \end{equation} Such systems arise in various physical contexts, ranging from capillary-gravity waves to plasma physics (see \cite{benilov}, \cite{djordjevic1977}, \cite{satsuma1979}). Specifically, \eqref{Schr-ILW} represents a coupling between a nonlinear Schrödinger-type equation and the Intermediate Long Wave (ILW) equation:
\begin{equation}
    \label{ILW} v_t-\,\mathscr{G}_{\delta}v_{xx}+\, vv_{x}=0. 
\end{equation}  Formally, when $\delta\rightarrow\infty$, the operator $\mathscr{G}_{\delta}$ tends to the Hilbert transform $\mathscr{H}$, which is defined as the Fourier multiplier by $-i\sgn(\xi)$, obtaining the Benjamin-Ono equation (BO) (see \cite{abdelouhab1989nonlocal}, \cite{chapouto}): 
\begin{equation}
    \label{BO} v_t-\,\mathscr{H}v_{xx}+\, vv_{x}=0.
\end{equation}\\ The limit case $\delta\rightarrow0$ can be considered as follows: define the scale transformation 
\begin{equation*}
    v_\delta(t,x):=3\delta^{-1}v(3\delta^{-1}t,x),
\end{equation*}
which transforms \eqref{ILW} into the following scaled ILW equation:
\begin{equation}
    \label{s-ILW} \partial_tv_\delta-3\delta^{-1}\mathscr{G}_\delta\partial_x^2v_\delta+v_\delta\partial_x v_\delta = 0.
\end{equation}
Then, under suitable assumptions, when $\delta\rightarrow0$, equation \eqref{s-ILW} is known to converge to the Korteweg-de Vries equation (KdV) (see \cite{abdelouhab1989nonlocal}, \cite{chapouto2025shallowwaterconvergenceintermediatelong}):
\begin{equation}
    \label{KdV} v_t+v_{xxx}+vv_x=0.
\end{equation} We also remark that while equations \eqref{ILW}, \eqref{BO}, and \eqref{KdV} are completely integrable systems, this does not appear to be the case for the system \eqref{Schr-ILW}. Nonetheless, the system \eqref{Schr-ILW} possesses several conserved quantities. For smooth solutions, the following functionals remain invariant on time:
\begin{equation}
    \label{C1} \mathcal{E}_1(t):=\int_{\mathbb{R}}v(t,x)\,dx = \mathcal{E}_1(0),
\end{equation}
\begin{equation}
    \label{C2} \mathcal{E}_2(t):=\int_{\mathbb{R}}\left|u(t,x)\right|^2\,dx = \mathcal{E}_2(0),
\end{equation}
\begin{equation}
    \label{C3} \mathcal{E}_3(t):=\int_{\mathbb{R}}v^2(t,x)\,dx-\frac{2\beta}{\alpha}\text{Im}\int_{\mathbb{R}}u(t,x)\overline{u_x(t,x)}\,dx \,= \mathcal{E}_3(0),
\end{equation}
and 
\begin{equation}
    \begin{split}
        \mathcal{E}_4(t):= & \int_{\mathbb{R}}|u(t,x)|^2\,v(t,x)\,dx \, + \, \frac{\nu}{2\beta}\int_{\mathbb{R}}v(t,x)\,\mathcal{G}_\delta v_x(t,x) \,dx \, - \frac{\rho}{6\beta} \int_{\mathbb{R}} v^3(t,x)\,dx\, \\
         & + \, \frac{1}{\alpha}\int_{\mathbb{R}} |u_x(t,x)|^2\,dx\,+\, \frac{\gamma}{\alpha}\int_{\mathbb{R}}|u(t,x)|^4\,dx= \mathcal{E}_4(0).
    \end{split}
    \label{C4}  
\end{equation} Beyond their physical significance, these conservation laws ensure that local solutions in $H^1(\mathbb{R})\times H^{\frac{1}{2}}(\mathbb{R})$ can be extended globally. \\ \\ Our principal purpose in this work is to study the Cauchy problem associated to \eqref{Schr-ILW} in low regularity Sobolev spaces. In this regards, we know plenty of well-posedness results for the equations involved in this system individually but not for the system itself. \\ \\ We aim to establish the local well-posedness of the Cauchy problem associated with \eqref{Schr-ILW} for initial data in $H^{s+\frac{1}{2}}(\mathbb{R})\times H^s(\mathbb{R})$ for all $s\geq0$, and provide the global theory for the case $s=1/2$. Our main result in this manuscript is the following: \\
\label{Th2}\begin{theorem} Let $s \geq 0$ be given. In system \eqref{Schr-ILW}, take $\alpha=1$, $\beta=v$, $\rho=\nu^2$, and assume that $|\nu|\not=1$. For all $(u_0,v_0) \in H^{s+\frac{1}{2}}(\mathbb{R})\times H^s(\mathbb{R})$ and all $T > 0$, there exist $b>\frac{1}{2}$, $0<\theta<1$ and a solution 
    \begin{equation}
        (u,v) \in X^{s+\frac{1}{2},b}_\text{S}(T)\times\big(C_TH^S_x\cap L^4_TW^{s,4}_x\cap X^{s-\theta,\theta}_\text{BO}(T)\big)
    \end{equation}
    of the system \eqref{Schr-ILW}, such that 
    \begin{equation}
        u(0)= u_0, \, \, \, v(0)=v_0, \, \, \text{and}
    \end{equation}
    \begin{equation}
        w = \partial_x P_{hi}(e^{-\frac{\nu}{2}i F[u]}) \in Y^{s,\frac{1}{2}}(T), 
    \end{equation}
    where $F[u]$ is some primitive of $u$ defined in \eqref{primitive}. 
This solution is unique in the class $v \in L^\infty(]0, T[; L^2(\mathbb{R})) \cap L^4(]0, T[ \times \mathbb{R})$ and $w \in X^{0, \frac{1}{2}}_\text{BO}(T)$. Moreover, the flow-map data solution $u_0 \mapsto u$ is continuous from $H^s(\mathbb{R})$ into $C([0, T]; H^s(\mathbb{R}))$.\\
\end{theorem} 

\begin{remark}
    Recently, Dai and Forlano in \cite{dai2026localglobalwellposednessextended} considered the so called extended Schrödinger-Benjamin Ono system:
    \begin{equation}
        \label{Schr-BO ext} 
        \begin{cases}
             iu_t + u_{xx}= \alpha uv+\gamma\,|u|^2u  \\ 
             v_t-\nu\,\mathscr{H}v_{xx}+\rho\, vv_{x}=\beta\left(|u|^2\right)_x
        \end{cases}
        \,\,\,\,\,\,(t,x) \in \mathbb{R}^2.
    \end{equation}
    They have obtained similar results as the one depicted above. Our analysis here is independent of the results in \cite{dai2026localglobalwellposednessextended}. \\ 
\end{remark} In \cite{linares2024well} Linares, Mendez, and Pilod studied the Cauchy problem associated to the system \eqref{Schr-BO ext} and established a local well-posedness result for smooth data. On the other hand, Chapouto, Li, Oh, and Pilod in \cite{chapouto} proved local and global well-posedness results for the ILW equation \eqref{ILW} in $H^s(M)$, $M=\mathbb{R}$ or $M=\mathbb{T}$, for $s\geq0$, using the approach introduced by Molinet and Pilod in \cite{molinet-pilod} to obtain similar results for the BO equation. These two facts inspired us to treat the Cauchy problem for system \eqref{Schr-ILW}. For a more complete account related to the ILW equation see \cite{abdelouhab1989nonlocal}, \cite{kleinsaut2022nonlinear}, and \cite{linares2024decayasymptoticpropertiessolutions}.  \\ \\Before describing the method employed in this work, it is instructive to review several related models of interest that arise as particular cases of \eqref{general}:
\begin{equation}
\begin{cases}
\label{Shcr-BO}
i u_t + u_{xx} = \alpha uv + \gamma|u|^2u  \\
v_t - \nu \mathscr{H}v_{xx} = \beta (|u|^2)_x
\end{cases}
\,\,\,\,\,\,(t,x) \in \mathbb{R}^2,
\end{equation}
and, 
\begin{equation}
\begin{cases}
\label{Shcr-quaseILW}
i u_t + u_{xx} = \alpha uv + \gamma|u|^2u  \\
v_t - \nu \mathscr{G}_\delta v_{xx} = \beta (|u|^2)_x
\end{cases}
\,\,\,\,\,\,(t,x) \in \mathbb{R}^2. 
\end{equation}  \\ The system \eqref{Shcr-BO} is known in the literature as the Schrödinger-Benjamin Ono system. Related to this system, we just mention that, using Bourgain spaces, Bekiranov, Ogawa and Ponce proved local well-posedness of the associated Cauchy problem with initial data in $H^{s}(\mathbb{R}) \times H^{s-\frac{1}{2}}(\mathbb{R})$ for all $s\geq0$, when $|\nu|\not=1$ (see \cite{MR1648479}). The case $|\nu|=1$ was covered by Pecher with a result of well-posedness in $H^{s}(\mathbb{R}) \times H^{s-\frac{1}{2}}(\mathbb{R})$ for all $s>0$ (see \cite{pecher2007rough}). We also comment that, in \cite{MR3449329},  Domingues obtained sharp local well-posedness results in $H^s(\mathbb{R})\times H^k(\mathbb{R})$ in the non-resonant case $|\nu|\not=1$ under some restrictions on the parameters $(s,k)$  which can be decoupled from $\left(s,s-\frac{1}{2}\right)$. We refer to \cite{linares2024well} for a  detailed review of the literature concerning the system \eqref{Shcr-BO}.   \\ \\Note that these studies can be replicated for the system \eqref{Shcr-quaseILW}. Indeed, following \cite{chapouto}, for each equation of \eqref{Shcr-quaseILW}  one can apply the Galilean transformation
\begin{equation}
    \label{galileana} \tilde{u}(t,x):=u\left(t,x+\nu\delta^{-1}t\right), \, \, \text{and,}
\, \, \tilde{v}(t,x):=v\left(t,x+\nu\delta^{-1}t\right), \, \, \, \, \, (t,x)\in\mathbb{R}^2.\end{equation}
which yields the following system (omitting the tildes):
\begin{equation}
    \begin{cases}
\label{Shcr-quaseILW-T}
i \left(u_t-\nu\delta^{-1}u_x\right) + u_{xx} = \alpha uv + \gamma|u|^2u  \\
v_t - \nu\mathscr{H}v_{xx}  = \nu\mathscr{Q}_{\delta}v_x+\beta (|u|^2)_x
\end{cases}
\,\,\,\,\,\,(t,x) \in \mathbb{R}^2,
\end{equation}
where $\mathscr{Q}_\delta:= \left(\tau_\delta-\mathscr{H}\right)\partial_x$ and $\tau_\delta$ is the Fourier multiplier by $-i\coth{\left(\delta\xi\right)}$. \\ \\ By considering the Duhamel formulation for \eqref{Shcr-quaseILW-T}, establishing crucial estimates in Bourgain spaces for each integral equation, and utilizing a contraction mapping argument, we obtain the local well-posedness of the associated Cauchy problem. Specifically, local well-posedness is established for initial data in $H^{s+\frac{1}{2}}(\mathbb{R}) \times H^s(\mathbb{R})$ for all $s\geq -\frac{1}{2}$ in the non-resonant case ($|\nu|\not=1$). In the resonant case ($|\nu|=1$), local well-posedness holds in $H^{s+\frac{1}{2}}(\mathbb{R}) \times H^s(\mathbb{R})$ for all $s>-\frac{1}{2}$. We remark that the natural Bourgain spaces associated with the first and second equations in \eqref{Shcr-quaseILW-T} are $X^{s+\frac{1}{2},b}_\text{S}$ and $X^{s,b}_\text{BO}$, respectively, with $b>\frac{1}{2}$ (see Subsection \ref{function spaces} below for the definition of the functional spaces).   \\ \\ We now return to the system \eqref{Schr-ILW}. Applying the Galilean transformation for this system, we get  
\begin{equation}
        \label{Schr-ILW-T} 
        \begin{cases}
             i\left(u_t-\nu\delta^{-1}u_x\right) + u_{xx}= \alpha uv+\gamma\,|u|^2u \\ 
             v_t-\nu\,\mathscr{H}v_{xx}+\rho\, vv_{x}=\nu \mathscr{Q}_{\delta}v_x+\beta\left(|u|^2\right)_x
        \end{cases}
        \,\,\,\,\,\,(t,x) \in \mathbb{R}^2.
    \end{equation} \\ 
In contrast to system \eqref{Shcr-quaseILW-T}, from the work of Molinet, Saut and Tzvetkov \cite{MR1885293} on the equation \eqref{BO}, it can be inferred that it is not plausible to directly apply a fixed-point argument to the system \eqref{Schr-ILW-T}. To overcome this obstacle, we use a compactness argument. \\ \\ First, we work with smooth initial data. In fact, following \cite{linares2024well}, we use energy estimates to obtain smooth solutions for the system \eqref{Schr-ILW}. Then, to consider initial data in low regularity Sobolev spaces we follow the approach used by Molinet and Pilod \cite{molinet-pilod} since the main difficulty is introduced by the second equation in \eqref{Schr-ILW-T}. We define a suitable gauge transformation, as the one considered by Tao in \cite{tao2004global} to treat the second equation in \eqref{Schr-ILW-T}. Specifically, let $(u,v)$ be a smooth solution of the system \eqref{Schr-ILW-T} with initial data $(u_0,v_0)\in H^\infty (\mathbb{R})\times H^\infty(\mathbb{R)}$. We also work with the parameters $\alpha=1$ and $\beta=\nu$. Consider a function $F:=F(t,x)$ such that $F_x=v$ and  
    \begin{equation}
        \label{F}F_t - \nu\mathscr{H}F_{xx}+\frac{\rho}{2}F_x^2=\nu \mathscr{Q}_\delta F_x + \nu |u|^2. 
    \end{equation}
See Section \ref{Gauge def} for an explicit construction of $F$. Define 
\begin{equation}
    w:= -\frac{\nu}{2}iP_{\text{+hi}}\left(e^{-\frac{\nu}{2}iF}v\right).
\end{equation} Setting $\rho=\nu^2$ it can be shown that $w$ satisfies the following equation: 
\begin{equation}
    \label{g} w_t-\nu\mathscr{H}w_{xx} = \mathscr{N}_\delta(u,v,w).
\end{equation}
where 
\begin{multline}
   \label{Ndelta} \mathscr{N}_\delta(u,v,w) := \nu^2 \partial_xP_{\text{+hi}}\left(\partial_x^{-1}w \, P_-\partial_xv\right)+\nu^2 \partial_xP_{\text{+hi}}\left(P_{\text{lo}}e^{-\frac{\nu}{2}iF} \, P_-\partial_xv\right)\\
    -\frac{\nu^2}{2}i \partial_xP_{\text{+hi}}\left(e^{-\frac{\nu}{2}iF} \, \mathscr{Q}_\delta v\right) -\frac{\nu^2}{2}i \partial_xP_{\text{+hi}}\left(e^{-\frac{\nu}{2}iF} \, |u|^2\right).
\end{multline}   The system \eqref{Schr-ILW-T}, coupled with equation \eqref{g}, forms a $3 \times 3$ system. While for the coupled system \eqref{Shcr-quaseILW-T} it was sufficient to work within $X^{s+\frac{1}{2},b}_\text{S}$ and $X^{s,b}_{\text{BO}}$, the inclusion of the gauge-transformed variable $w$ in \eqref{Schr-ILW-T} needs a more delicate choice of functional spaces. Adopting the framework proposed in \cite{molinet-pilod} and \cite{chapouto}, for the second equation of \eqref{Schr-ILW-T} we employ the intersection space $L^\infty_T H^s_x \cap L^4_T W^{s,4}_x \cap X^{s-\theta,\theta}_{\text{BO}}(T)$, with $0\leq\theta\leq1$, and for equation \eqref{g}, we work with $Y^{s,\frac{1}{2}}(T)$. We also remark that the $Y^{s,\frac{1}{2}}$ -  norm of the first three terms of $\mathscr{N}_\delta(u,v,w)$ were bounded in \cite{molinet-pilod} and \cite{chapouto}.\\ \\  For the first equation of system \eqref{Schr-ILW-T}, following the method employed for \eqref{Shcr-quaseILW-T}, we work with $X^{s+\frac{1}{2},b}_\text{S}(T)$, the time-localized version of $X^{s+\frac{1}{2},b}_\text{S}$. In fact, for the Duhamel formulation of this equation, we apply estimates in Bourgain spaces, obtaining 
\begin{equation} \label{yacasito_final} \|u\|_{X^{s+\frac{1}{2},b}_\text{S}(T)} \lesssim \|u(0)\|_{H^{s+\frac{1}{2}}} + \|uv\|_{X^{s+\frac{1}{2},b-1}_\text{S}} + \||u|^2u\|_{X^{s+\frac{1}{2},b-1}_\text{S}(T)}. \end{equation}
While it is tempting to apply a bilinear estimate of the type
\begin{equation}
    \label{be} \left\|uv\right\|_{X^{s+\frac{1}{2},b}_\text{S}(T)}\lesssim \|u\|_{X^{s+\frac{1}{2},b}_\text{S}(T)}\|v\|_{X^{s,b}_\text{BO}(T)}
\end{equation}
to the second term on the right-hand side of \eqref{yacasito_final}, a difficulty arises: $X^{s,b}_{\text{BO}}$ is not the space where $v$ lies. Recalling that $v$ is instead placed in a subset of $X^{s-\theta,\theta}_{\text{BO}}$, the primary technical challenge of this work is therefore to obtain an estimate of type 
\begin{equation*}
    \left\|uv\right\|_{X^{s+\frac{1}{2},a}_\text{S}} \lesssim \|u\|_{X^{s+\frac{1}{2},b}_\text{S}}\left\|v\right\|_{X^{s-\theta,\theta}_\text{BO}},
\end{equation*} for any $0\leq\theta\leq1$ depending on $s\geq0$, $b>\frac{1}{2}$  and some $a\leq0$. Specifically, in Section \ref{uv}, we prove the required inequality \eqref{estimativa bilineal} for the non-resonant case $|\nu| \neq 1$. We also remark that this method does not work in the resonant case $|\nu| = 1$.  \\ \\ Combining these estimates with a bootstrap argument and the conservation laws \eqref{C2}, \eqref{C3}, and \eqref{C4}, we obtain the desired result. A final important remark regards the smallness assumptions to implement the Molinet-Pilot argument. In Remark \ref{scalingch} we explain how to overcome this difficulty. \\  \\ This paper is organized as follows: in Section 2, we introduce the notations, define the function spaces, and recall some classical linear estimates. In Section 3 we introduce the gauge transformation and we work its a priori estimates. Section 5 is devoted to prove the key nonlinear estimates, which are used in Section 6 to prove the main part of Theorem \ref{Th2}. In the appendix, following the approach in \cite{linares2024well}, we establish local well-posedness for the Cauchy problem associated to \eqref{Schr-ILW} for smooth initial data. \\ \\ \textbf{Acknowledgments} This paper is part of my Ph.D. thesis at IMPA under the guidance of my advisor Felipe Linares. I want to take the opportunity to express my sincere gratitude to him. The author was supported by CNPq-Brazil. \\ 

\section{Notation, function spaces and preliminary estimates}
\subsection{ Notation} 
For any positive numbers $a$ and $b$, the notation $a \lesssim b$ means that there exists a positive constant $c$ such that $a \le cb$. We also write $a \sim b$ when $a \lesssim b$ and $b \lesssim a$. Moreover, if $\alpha \in \mathbb{R}$, $\alpha_+$ and $\alpha_-$ will denote a number slightly greater and lesser than $\alpha$, respectively.\\ \\ For $u = u(t,x) \in \mathscr{S}'(\mathbb{R}^2)$, $\mathscr{F}u = \hat{u}$ will denote its space-time Fourier transform whereas $\mathscr{F}_x u = (u)^{\wedge_x}$ and $\mathscr{F}_t u = (u)^{\wedge_t}$ will denote its Fourier transform in space and time, respectively. For $s \in \mathbb{R}$, we define the Bessel and Riesz potentials of order $s$, $J_x^s$ and $D_x^s$, by
\begin{equation}
\widehat{J_x^s u}^x(t,\xi) = \left(1+\xi^2\right)^\frac{s}{2}\widehat{u}^x(t,\xi) \quad \text{and} \quad \widehat{D_x^s u}^x(t,\xi) = |\xi|^s \widehat{u}^x(t,\xi),
\end{equation} respectively. Throughout the paper, we fix a cutoff function $\eta$ such that $\eta \in C_0^\infty(\mathbb{R})$, $0 \le \eta \le 1$, $\text{supp}(\eta) \subset [-2, 2]$, and $\eta|_{[-1, 1]} = 1$. We define
\begin{equation}
\phi(\xi) := \eta(\xi) - \eta(2\xi) \quad \text{and} \quad \phi_{2^l}(\xi) := \phi(2^{-l}\xi).
\end{equation} Summations over capitalized variables such as $N$ are presumed to be dyadic with $N \ge 1$; i.e., these variables range over numbers of the form $2^n, n \in \mathbb{Z}_+$. Then we have $\sum_N \phi_N(\xi) = 1 - \eta(2\xi)$ for all $\xi \neq 0$ and $\text{supp}(\phi_N) \subset \{ \frac{1}{2}N \le |\xi| \le 2N \}$. \\ Let us define the Littlewood-Paley multipliers by
\begin{equation}
\widehat{P_N u}^x(t,\xi) = \phi_N(\xi)\widehat{u}^x(t,\xi) \quad \text{and} \quad P_{\ge N} := \sum_{K \ge N} P_K. 
\end{equation} We also define the operators $P_{\text{hi}}$, $P_{\text{HI}}$, $P_\text{lo}$, and $P_\text{LO}$ by
\begin{equation}
P_\text{hi} = \sum_{N \ge 2} P_N, \quad P_\text{HI} = \sum_{N \ge 8} P_N, \quad P_\text{lo} = I - P_\text{hi}, \quad \text{and} \quad P_\text{LO} = I - P_\text{HI} \text{.}
\end{equation} Let $P_+$ and $P_-$ denote the projections on the positive and the negative Fourier frequencies, respectively. Then
\begin{equation}
\widehat{P_{\pm} u}^x(t,\xi) := \chi_{\mathbb{R}_{\pm}}(\xi) \widehat{u}^x(t,\xi),
\end{equation} and we also define $P_{\pm \text{hi}}=P_{\pm}P_\text{hi}$, $P_{\pm \text{HI}}=P_{\pm}P_\text{HI}$, $P_{\pm \text{lo}}=P_{\pm}P_\text{lo}$ and $P_{\pm \text{LO}}=P_{\pm}P_\text{LO}.$ Observe that $P_\text{hi}$, $P_\text{HI}$, $P_\text{lo}$, and $P_\text{LO}$ are bounded operators on $L^{p}(\mathbb{R})$ for $1\le p\le\infty$ while $P_{\pm}$ is only bounded on $L^{p}(\mathbb{R})$ for $1<p<\infty$. We also note that
\begin{equation}
\mathscr{H}=-iP_{+}+iP_{-} \text{.} \label{2-1}
\end{equation}
Finally, we denote by $U(\cdot)$ and $V(\cdot)$ the free groups associated with the linearized Schrödinger and Benjamin-Ono equations, respectively, which is to say,
\begin{equation}
    \label{grupo1} \widehat{\text{U}(t)u}^x(t,\xi):= e^{-it\left(\xi^2-\nu\delta^{-1}\xi\right)}\widehat{u}^x(t,\xi),
\end{equation}
and,
\begin{equation}
    \label{grupo2} \widehat{\text{V}(t)u}^x(t,\xi):= e^{-it\left(\nu\xi|\xi|\right)}\widehat{u}^x(t,\xi).
\end{equation}

\subsection{\label{function spaces}Function spaces} 
For $1\le p\le\infty$, $L^{p}(\mathbb{R})$ is the usual Lebesgue space with the norm $\|\cdot\|_{L^{P}}$ and for $s\in\mathbb{R},$ the real-valued Sobolev spaces $H^{s}(\mathbb{R})$ and $W^{s,p}(\mathbb{R})$ denote the spaces of all real-valued functions with the usual norms
\begin{equation}
\|f\|_{H^{s}}=\left\|J^{s}f\right\|_{L^{2}} \quad \text{and} \quad \|f\|_{W^{s,p}}=\|J_{x}^{s}f\|_{L^{p}} \text{.} \label{2-3}
\end{equation}
Let us see how the symbol $\mathscr{Q}_\delta$ acts in Sobolev spaces. \\ 

\begin{lemma}
    Let $s\geq0$. Then, we have
    \begin{equation}
      \label{qdelta}  \left\|\mathscr{Q}_\delta f\right\|_{H^s}\lesssim \delta^{-1}\big(1+\delta^{-s}\big)\|f\|_{L^2},
    \end{equation}
    where the implicit constant is independent of $\delta$.  
\end{lemma}
\begin{proof} [Proof]
    See the proof of Lemma 2.3 in \cite{chapouto}. 
\end{proof} For $1<p<\infty$, we define the space $\tilde{L}^{p}$ as 
\begin{equation}
\|f\|_{\tilde{L}^{p}}=\|P_{lo}f\|_{L^{p}}+\bigg(\sum_N\|P_{N}f\|_{L^{p}}^{2}\bigg)^{1/2} \text{.}
\end{equation}
Observe that when $p\ge2$, the Littlewood-Paley theorem on the square function and Minkowski's inequality imply that the injection $\tilde{L}^{p}\hookrightarrow L^{p}$ is continuous. Moreover, if $u=u(t,x)$ is a real-valued function defined for $x\in\mathbb{R}$ and in the time interval $[0, T]$ with $T>0,$ $B$ is one of the spaces defined above, and $1\le p\le\infty,$ we will define the mixed space-time spaces $L_{T}^{p}B_{x}$ and $L_{t}^{p}B_{x}$ by the norms 
\begin{equation}
\|u\|_{L_{T}^{p}B_{x}}=\Big(\int_{0}^{T}\|u(\cdot,t)\|_{B}^{p}dt\Big)^{\frac{1}{p}} \quad \text{and} \quad \|u\|_{L_{t}^{p}B_{x}}=\Big(\int_{\mathbb{R}}\|u(\cdot,t)\|_{B}^{p}dt\Big)^{\frac{1}{p}} \text{,}
\end{equation}
respectively.

For $s, b\in\mathbb{R}$ we introduce the Bourgain spaces $X^{s,b}$ and $Z^{s,b}$ related to the Benjamin-Ono equation as the completion of the Schwartz space $\mathscr{S}(\mathbb{R}^{2})$ under the norms 
\begin{equation*}
    \|u\|_{X^{s,b}_\text{S}}:= \big\|\langle\tau+\xi^2-\nu\delta^{-1}\xi\rangle^b\langle\xi\rangle^s\widehat{u}(\tau,\xi)\big\|_{L^2_{\tau,\xi}},
\end{equation*}
\begin{equation*}
    \|u\|_{X^{s,b}_{\text{BO}}}:= \big\|\langle\tau-\nu\xi|\xi|\rangle^b\langle\xi\rangle^s\widehat{u}(\tau,\xi)\big\|_{L^2_{\tau,\xi}},
\end{equation*}
\begin{equation*}
    \|u\|_{Z^{s,b}}:= \big\|\langle\tau-\nu\xi|\xi|\rangle^b\langle\xi\rangle^s\widehat{u}(\tau,\xi)\big\|_{L^2_\xi\,L^1_\tau},
\end{equation*}
\begin{equation*}
    \|u\|_{\tilde{Z}^{s,b}}:= \bigg(\sum_{N\in2^\mathbb{Z}}\|P_Nu\|_{Z^{s,b}}^2\bigg)^{\frac{1}{2}},
\end{equation*}
and,
\begin{equation*}
    \|u\|_{Y^{s,b}}:= \|u\|_{X^{s,b}_{\text{BO}}}+\|u\|_{\tilde{Z}^{s,b-\frac{1}{2}}},
\end{equation*}
where $\langle x\rangle := 1 + |x|$ for all $x\in\mathbb{R}$. \\ \\ We will also use the localized (in time) version of these spaces. Let $T>0$ be a positive time and $\|\cdot\|_{B}$ be one of the norms defined above. If $u:\mathbb{R}\times[0,T]\rightarrow\mathbb{C}$, then 
\begin{equation}
\|u\|_{B_{T}}:=\inf\left\{\|\tilde{u}\|_{B} \mid \tilde{u}:\mathbb{R}\times\mathbb{R}\rightarrow\mathbb{C}, \tilde{u}|_{\mathbb{R}\times[0,T]}=u\right\} \text{.} \label{2-5}
\end{equation}

\subsection{Linear estimates}

In this subsection, we recall some linear estimates in Bourgain's spaces that will be needed later. The first ones are well known (see \cite{ginibre}, for example). \\

\begin{lemma}
Let $s \in \mathbb{R}$. Then
\begin{equation}
\left\|\eta(t)U(t)f\right\|_{Y^{s,\frac{1}{2}}} \lesssim \|f\|_{H^s}. \label{2-5}
\end{equation}
\end{lemma} 

\begin{lemma}
Let $s \in \mathbb{R}$ and $-\frac{1}{2}<b'\leq0\leq b\leq b'+1$ and $0<T\leq1$. Then,
\begin{equation}
\Big\|\eta(t) \int_{0}^{t} U(t-t')g(t') dt'\Big\|_{X^{s,b}_\text{BO}} \lesssim T^{1-b+b'}\|g\|_{X^{s,b'}_\text{BO}},\label{2-6}
\end{equation}
\begin{equation}
\Big\|\eta(t) \int_{0}^{t} U(t-t')g(t') dt'\Big\|_{X^{s,b}_\text{S}} \lesssim T^{1-b+b'}\|g\|_{X^{s,b'}_\text{S}}, \label{2-6}
\end{equation}
and
\begin{equation}
\Big\|\eta(t) \int_{0}^{t} U(t-t')g(t') dt'\Big\|_{Y^{s,\frac{1}{2}}} \le \|g\|_{X_\text{BO}^{s,-\frac{1}{2}}} + \|g\|_{\tilde{Z}^{s,-1}}. \label{2-7}
\end{equation}
\end{lemma} 

\begin{lemma}
    For any $T>0$, $s\in\mathbb{R}$ and for all $-\frac{1}{2}<b'\leq b< \frac{1}{2}$,
    \begin{equation}
        \|u\|_{X^{s,b'}_\text{BO}(T)}\lesssim T^{b-b'}\|u\|_{X^{0,\frac{3}{8}}_\text{BO}(T)}. \label{2-8}
    \end{equation}
\end{lemma}
The following Bourgain-Strichartz estimates will also be useful: \\

\begin{lemma}
    We have
    \begin{equation}
        \|u\|_{L^4_{x,t}}\lesssim T^{b-b'}\|u\|_{\tilde{L}^4_{x,t}}\lesssim \|u\|_{X_\text{BO}^{0,\frac{3}{8}}}, \label{2-9}
    \end{equation}
    and for any $T>0$ and $\frac{3}{8}\leq b \leq \frac{1}{2}$,
    \begin{equation}
       \label{BouSt} \|u\|_{L^4_{x,T}}\lesssim T^{b-\frac{3}{8}}\|u\|_{X^{0,\frac{3}{8}}_\text{BO}(T)}
    \end{equation}
\end{lemma}
\begin{proof} [Proof]
    See the proof of Lemma 2.4 in \cite{molinet-pilod}.
\end{proof}



\subsection{Fractional Leibniz rules}
\begin{lemma}
    \label{LiLeibniz} Let $1<p<\infty$ and $1<p_1,\, p_2,\, p_3,\, p_4 \leq \infty$ such that $\frac{1}{p_1}+\frac{1}{p_2}=\frac{1}{p_3}+\frac{1}{p_4}=\frac{1}{p}$. Then, for all $f,g\in\mathscr{S}(\mathbb{R})$, 
    \begin{equation*}
        \left\|\left[D^s_x;f\right]g\right\|_{L^p}\lesssim \left\|D^{s-1}_x\partial_x f\right\|_{L^{p_1}}\|g\|_{L^{p_2}} \, \, \, \, \text{if} \, \, \, \, 0<s\leq 1,
    \end{equation*}
    and 
    \begin{equation*}
        \left\|\left[D^s_x;f\right]g\right\|_{L^p}\lesssim \left\|D^{s}_xf\right\|_{L^{p_1}}\|g\|_{L^{p_2}}+ \left\|\partial_xf\right\|_{L^{p_3}}\|D^{s-1}_xg\|_{L^{p_2}}\, \, \, \, \text{if} \, \, \, \, s>1.
    \end{equation*}
\end{lemma}
\begin{proof} [Proof]
    See Theorem 5.1 and Corollary 5.3 in \cite{Li}.
\end{proof}
\begin{lemma} Let $\alpha\geq0$ and $1<q<\infty$. Then 
    \begin{equation}
        \label{FLR'} \left\|D^\alpha_xP_+\left(fP_-\partial_xg\right)\right\|_{L^q}\lesssim \left\|D^{\alpha_1}_xf\right\|_{L^{q_1}}\left\|D^{\alpha_2}_xg\right\|_{L^{q_2}}
    \end{equation}
    with $1<q_1,q_2<\infty$, $\frac{1}{q}=\frac{1}{q_1}+\frac{1}{q_2}$, and  $\alpha_1\geq \alpha,\alpha_2\geq0$, and $\alpha_1+\alpha_2=1+\alpha$. 
\end{lemma}
\begin{proof}[Proof]
    See Lemma 3.2 in \cite{Molinet2007}.
\end{proof}
\begin{lemma}  Let $2\leq q <\infty$ and $0\leq \alpha \leq \frac{1}{q}$. Consider $F_1$ and $F_2$ two real valued functions such that $f_j:=\partial_xF_j\in L^2(\mathbb{R})$ for $j\in\{1,2\}$. Then
    \begin{equation}
        \label{fracexp} \left\|J^\alpha_x\left(e^{\frac{\nu}{2}iF_1}g\right)\right\|_{L^q}\lesssim \left(1+\|f_1\|_{L^2}\right)\left\|J^\alpha_xg\right\|_{L^q},
    \end{equation}
    and
    \begin{equation}
        \label{diffracexp} \big\|J^\alpha_x((e^{\frac{\nu}{2}iF_1}-e^{\frac{\nu}{2}F_2})g)\big\|_{L^q}\lesssim \left(\|f_1-f_2\|_{L^2}+\big\|e^{\frac{\nu}{2}iF_1}-e^{\frac{\nu}{2}iF_2}\big\|_{L^\infty}\|f_1\|_{L^2}\right)\left\|J^\alpha_x g\right\|_{L^q}.
    \end{equation}
\end{lemma}
\begin{proof} [Proof]
    See Lemma 2.7 in \cite{molinet-pilod}. 
\end{proof}

\section{A priori estimates}
\label{Gauge def}\subsection{The gauge transformation}
In the upcoming sections, we present the necessary tools for the proof of Theorem \ref{Th2}. Following \cite{molinet-pilod}, we first construct a spatial primitive $F$ of $v$ (i.e. $F_x=v$) that satisfies the equation \eqref{F}. 
Consider $\psi\in C^{\infty}_0(\mathbb{R})$ such that $\int_\mathbb{R}\psi(y)\,dy=1$ and define 
\begin{equation}
    \label{primitive} F(x,t):=\int_\mathbb{R}\psi(y)\int_y^x v(t,z)\,dz\,dy\,+ G(t),
\end{equation}
where $G=G(t)$ is a function to be found. Note that $F_x=v$ and 
\begin{align*}
    &F_t(t,x)\\
            = &\int_\mathbb{R}\psi(y)\int_y^x\Big(\nu \mathscr{H}v_{zz}-\frac{\nu^2}{2}\left(v^2\right)_z+\nu\mathscr{Q}_\delta v_z+\nu\left(|u|^2\right)_z\Big)(t,z)\,dz\,dy+G'(t) \\
            = &\nu \mathscr{H}v_{x}-\frac{\nu^2}{2}v^2+\nu\mathscr{Q}_\delta v + \nu |u|^2 - \int_{\mathbb{R}}\Big(-\nu\mathscr{H}\psi'(y)\,v(t,y)-\psi(y)\frac{\nu^2}{2}v(t,y)^2\\&+\nu\psi(y)\mathscr{Q}_\delta\, v(t,y)+\nu\psi(y)|u(t,y)|^2\Big)\,dy+ G'(t).
\end{align*}
Therefore, we choose G as 
\begin{multline}
    \label{G-function} G(t)= \int_0^t\int_{\mathbb{R}}\big(\nu\mathscr{H}\psi'(y)\,v(t',y)-\psi(y)\frac{\nu^2}{2}v(t',y)^2\\+\nu\psi(y)\mathscr{Q}_\delta\, v(t',y)+\nu\psi(y)|u(t',y)|^2\big)dydt'
\end{multline}
to ensure that \eqref{F} is satisfied. Observe that this construction makes sense for $u,v\in L^2_{\text{loc}}(\mathbb{R}^2)$. Next, we introduce the new unknown 
\begin{equation}
  \label{gauge}  W:= P_{+\text{hi}}\big(e^{-\frac{\nu}{2} i F}\big) \, \, \, \, \text{and} \, \,\, \, w:= W_x = -\frac{\nu}{2}iP_{+\text{hi}}\big(e^{-\frac{\nu}{2} i F}v\big). 
\end{equation}
Setting $\rho=\nu^2$, from \eqref{F} it follows that 
\begin{align*}
    &W_{t}-\nu\mathscr{H}W_{xx}\\
    =&-\frac{\nu}{2}iP_{+\text{hi}}\left(e^{-\frac{\nu}{2}iF}F_t\right)+\nu i \left(P_+-P_{-}\right)P_+P_{\text{hi}}\partial_x^2\left(e^{-\frac{\nu}{2}iF}\right)\\
    =& -\frac{\nu}{2}iP_{+\text{hi}}\Big(e^{-\frac{\nu}{2}i}\big(F_t+\frac{\nu^2}{2}F_x^2+\nu i F_{xx}\big)\Big)\\
    =& -\frac{\nu}{2}iP_{+\text{hi}}\left(e^{-\frac{\nu}{2}i}\left(\nu\mathscr{H}F_{xx}+\nu\mathscr{Q}_\delta F_x + \nu |u|^2+\nu i F_{xx}\right)\right)\\
    =&\,\,\,\nu^2P_{+\text{hi}}\left(e^{-\frac{\nu}{2}i}P_{-}v_x\right)-\frac{\nu^2}{2}iP_{+\text{hi}}\left(e^{-\frac{\nu}{2}i}\mathscr{Q}_\delta v\right)-\frac{\nu^2}{2}iP_{+\text{hi}}\left(e^{-\frac{\nu}{2}i}|u|^2\right).
\end{align*}
Regarding the first term, note that 
\begin{align*}
    P_{+\text{hi}}\left(e^{-\frac{\nu}{2}i}P_{-}v_x\right)&= P_{+\text{hi}}\left(\left(P_{+\text{hi}}+P_{-hi}+P_{\text{lo}}\right)e^{-\frac{\nu}{2}i}P_{-}v_x\right)\\&=P_{+\text{hi}}\left(WP_{-}v_x\right) + P_{+\text{hi}}\left(P_{\text{lo}}e^{-\frac{\nu}{2}i}P_{-}v_x\right).
\end{align*}
From differentiating it follows the equation \eqref{g}. We recall that we now consider system \eqref{Schr-ILW-T} joined with equation \eqref{g} as a system $3\times3$. We start recovering $v$ in terms of $w$ in the following way: 
 \begin{align*}
     v &= e^{\frac{\nu}{2}iF}e^{-\frac{\nu}{2}iF}v \\
       &= e^{\frac{\nu}{2}iF}\left(\left(P_++P_-\right)P_{\text{hi}}+P_{lo}\right)\left(e^{-\frac{\nu}{2}iF}v\right)\\
       &= \frac{2}{\nu}ie^{\frac{\nu}{2}iF}w+e^{\frac{\nu}{2}iF}P_{\text{lo}}\left(e^{-\frac{\nu}{2}iF}v\right)+\frac{2}{\nu}ie^{\frac{\nu}{2}iF}\partial_xP_{-\text{hi}}\left(e^{-\frac{\nu}{2}iF}\right),
 \end{align*}
 where in the last line we assume that $\nu\not=0$. Moreover, from frequency localization, 
 \begin{multline}
     \label{PHI} P_{+\text{HI}}\,v = \frac{2}{\nu}iP_{+\text{HI}}\left(e^{\frac{\nu}{2}iF}w\right)+P_{+\text{HI}}\left(P_{+\text{hi}}e^{\frac{\nu}{2}iF}P_{\text{lo}}\left(e^{-\frac{\nu}{2}iF}v\right)\right)\\+\frac{2}{\nu}iP_{+\text{HI}}\left(P_{+\text{HI}}e^{\frac{\nu}{2}iF}\partial_xP_{-\text{hi}}\left(e^{-\frac{\nu}{2}iF}\right)\right).
 \end{multline}
 This fact is so useful to derive the following estimates: \\
 
 \begin{proposition} Let $0\leq s \leq 1$, $0<T\leq1$, $0\leq\theta\leq1$, $\nu\not=0$, and $(u,v)$ be a solution of the system \eqref{Schr-ILW-T} in the interval $[0,T]$, with initial data $(u(0),v(0))=(u_0,v_0)\in H^{s+\frac{1}{2}}(\mathbb{R})\times H^s(\mathbb{R})$. Then
               \begin{multline}
         \label{apriori1} \|v\|_{X^{s-\theta,\theta}_{\text{BO}}(T)} \lesssim T^{\frac{1}{2}}\delta^{-1}\left(1+\delta^{-s}\right)\|v\|_{L^\infty_TL^2_x}+\|v\|_{L^\infty_T H^s_x}+\|v\|_{L^4_{T,x}}\|v\|_{L^4_T W^{s,4}_x} \\+ T^\frac{1}{2}\|u\|_{X^{\frac{1}{2},b}_\text{S}(T)}\|u\|_{X^{s+\frac{1}{2},b}_\text{S}(T)}. 
     \end{multline}
Moreover, if $0\leq s \leq \frac{1}{4}$, we have 
     \begin{multline}
         \label{apriori2} \|J^s_xv\|_{L^q_T L^p_x}\lesssim \|v_0\|_{L^2_x}+T\delta^{-1}\|v\|_{L^\infty_TL^2_x}\\+\left(1+\|v\|_{L^\infty_TL^2_x}\right)\left(\|w\|_{Y^{s,\frac{1}{2}}(T)}+T^{1+\frac{1}{p}}\|v\|_{L^\infty_TL^2_x}^2\right)+T\|u\|_{X^{\frac{1}{2},b}_\text{S}(T)}^2.  
     \end{multline}
for $(p,q)\in\left\{(\infty,2),(4,4)\right\}$. 
 \end{proposition}
 \begin{proof}[Proof]
 We start by deriving the estimate \eqref{apriori1}. Following \cite{molinet-pilod}, define 
     \begin{equation*}
        \Upsilon (t)=\begin{cases}
            v(0),  \hspace{50pt} -2\leq t\leq 0 \\
            \text{V}(-t)v(t), \hspace{31pt}   0\leq t\leq T\\
            \text{V}(-T)v(T), \hspace{23pt}  T\leq t\leq 2
        \end{cases}.
     \end{equation*}
     Note that $\partial_t\Upsilon=0$ in $[-2,2]\setminus[0,T]$, and for all $r\in\mathbb{R},$
     \begin{equation}
         \label{faim} \left\|\partial_t\Upsilon\right\|_{L^2_{[-2,2]}H^r_x} = \left\|\partial_t v\right\|_{L^2_{T}H^r_x} \, \, \text{and} \, \,\left\|\Upsilon\right\|_{L^2_{[-2,2]}H^r_x} \lesssim  \|v_0\|_{H^r}+T^{\frac{1}{2}}\|v\|_{L^\infty_T H^r_x}. 
     \end{equation}
     Define $\tilde{v}(t):= \eta(t)\text{V}(t)\Upsilon(t)$. Note that $\tilde{v}$ is an extension of $v$. Moreover, 
     \begin{align*}
         \|\tilde{v}\|_{X^{s-\theta,\theta}_\text{BO}}&:= \left\|\langle\tau\rangle^\theta\langle\xi\rangle^{-\theta}\left(J^s_x\text{V}(-t)\tilde{v}\right)^\wedge(\tau,\xi)\right\|_{L^2_{\tau,\xi}}\\
         &\sim \Big(\int_{\{|\tau|\leq|\xi|\}}\bigg(\frac{1+|\tau|}{1+|\xi|}\bigg)^{2\theta}\big|\big(J^s_x\big(\eta(t)\Upsilon\big)\big)^\wedge(\tau,\xi)\big|^2\,d\tau d\xi\Big)^{1/2}\\&\textcolor{white}{+}+\Big(\int_{\{|\tau|\geq|\xi|\}}\bigg(\frac{1+|\tau|}{1+|\xi|}\bigg)^{2\theta}\left|\left(J^s_x\left(\eta(t)\Upsilon\right)\right)^\wedge(\tau,\xi)\right|^2\,d\tau d\xi\Big)^{1/2}:= I + II,
     \end{align*}
     \begin{equation*}
         I \leq \Big(\int_{\{|\tau|\leq|\xi|\}}\left|\left(J^s_x\left(\eta(t)\Upsilon\right)\right)^\wedge(\tau,\xi)\right|^2\,d\tau d\xi\Big)^{\frac{1}{2}} \leq \|\eta(t)\Upsilon\|_{L^2_{[-2,2]}H^s_x}\lesssim \|v\|_{L^\infty_T H^s_x},
     \end{equation*}
     where the last inequality is due to \eqref{faim}. On the other hand, 
     \begin{align*}
         II &\leq \Big(\int_{\{|\tau|\geq|\xi|\}}\bigg(\frac{1+|\tau|}{1+|\xi|}\bigg)^2\left|\left(J^s_x\left(\eta(t)\Upsilon\right)\right)^\wedge(\tau,\xi)\right|^2\,d\tau d\xi\Big)^{\frac{1}{2}}\\
         &\lesssim \left\|\eta(t)\Upsilon\right\|_{L^2_t H^{s-1}_x}+\left\|\partial_t\left(\eta(t)\Upsilon\right)\right\|_{L^2_t H^{s-1}_x} \lesssim T^\frac{1}{2}\|v\|_{L^\infty_T H^{s-1}_x} + \left\|\partial_t v\right\|_{L^2_T H^{s-1}_x},
     \end{align*}
     where the last inequality is due to \eqref{faim}. Now, to bound the second term, at first we use the identity 
     \begin{equation*}
         \partial_t v = \partial_t\left(\text{V}(-t)v\right)= \text{V}(-t)\left[\nu\mathscr{H}\partial^2_xv-\partial_tv\right],
     \end{equation*}
     then, we use the second equation of the system \eqref{Schr-ILW-T}, and after that, the fractional Leibniz rule (see \cite{kenig1993wellposedness}) and the estimate \eqref{qdelta}, obtaining 
     \begin{align*}
         \left\|\partial_tv\right\|_{L^2_T H^{s-1}_x} &\lesssim \|v\|_{L^4_{T,x}}\|v\|_{L^4_T W^{s,4}_x}+ \delta^{-1}\left(1+\delta^{-s}\right)\|v\|_{L^2_{T,x}}+\|u\|_{L^4_{T,x}}\|u\|_{L^4_T W^{s,4}_x}\\
         &\leq \|v\|_{L^4_{T,x}}\|v\|_{L^4_T W^{s,4}_x}+ T^{\frac{1}{2}}\delta^{-1}\left(1+\delta^{-s}\right)\|v\|_{L^\infty_{T} L^2_x}+\|u\|_{L^4_{T,x}}\|u\|_{L_T^4 W^{s,4}_x}.
     \end{align*}
     Finally, we consider the third term on the last inequality. Indeed, using the Sobolev and Bourgain embeddings, we get
     \begin{align*}
         \|u\|_{L^4_{T,X}} \|u\|_{L^4_{T}W^{s,4}_x}&=\bigg(\int_0^T \|u(t)\|^4_{L^4_x}\,dt\bigg)^{1/4}\bigg(\int_0^T \|J^s_xu(t)\|^4_{L^4_x}\,dt\bigg)^{1/4}\\
         &\lesssim \bigg(\int_0^T \big\|J^{\frac{1}{2}-\frac{1}{4}}_xu(t)\big\|^4_{L^2_x}\,dt\bigg)^{1/4}\bigg(\int_0^T \big\|J^{s+\frac{1}{2}-\frac{1}{4}}_xu(t)\big\|^4_{L^2_x}\,dt\bigg)^{1/4}\\
         &\leq T^{1/2}\|u\|_{L^\infty_T H^{1/2}_x}\|u\|_{L^\infty_T H^{s+1/2}_x} \lesssim T^{1/2}\|u\|_{X^{1/2,b}_\text{S}(T)}\|u\|_{X^{s+1/2,b}_\text{S}(T)}. 
     \end{align*}
Then, it follows the estimate \eqref{apriori1}. Now we turn to the proof of the inequality \eqref{apriori2}. We follow the proof of  Proposition 3.2 in \cite{molinet-pilod} and  Lemma 3.4 in \cite{chapouto}. Indeed, since $v$ is real-valued, it holds that $P_-v=\overline{P_+v}$ so that 
\begin{equation}
    \label{lowhight} \left\|J^s_x v\right\|_{L^p_TL^q_x}\lesssim \left\|P_\text{LO}v\right\|_{L^p_TL^q_x}+\left\|D^s_xP_{+\text{HI}}\,v\right\|_{L^p_TL^q_x}.
\end{equation}
We consider the second term on the right hand side of \eqref{lowhight}. By identity \eqref{PHI},
\begin{equation}
    \label{altisima}
    \begin{split}
\big\|D^s_xP_{+\text{HI}}\,v\big\|_{L^p_TL^q_x}&\lesssim \big\|D^s_xP_{+\text{HI}}\big(e^{\frac{\nu}{2}iF}w\big)\big\|_{L^p_TL^q_x}+\big\|D^s_xP_{+\text{HI}}\big(P_{+\text{hi}}e^{\frac{\nu}{2}iF}P_{\text{lo}}\big(e^{-\frac{\nu}{2}iF}v\big)\big)\big\|_{L^p_TL^q_x}\\&+\big\|D^s_xP_{+\text{HI}}\big(P_{+\text{HI}}\,e^{\frac{\nu}{2}iF}\partial_xP_{-\text{hi}}\big(e^{-\frac{\nu}{2}iF}\big)\big)\big\|_{L^p_TL^q_x}\\
&:= \tilde{I}+\tilde{II}+\tilde{III}.
    \end{split}
\end{equation} Estimates \eqref{fracexp} and \eqref{BouSt} yield 
\begin{equation}
    \label{11}\tilde{I}\lesssim \big(1+\|v\|_{L^\infty_TL^2_x}\big)\big\|J^s_xw\big\|_{L^p_TL^q_x}\lesssim \big(1+\|v\|_{L^\infty_TL^2_x}\big)\|w\|_{Y^{s,\frac{1}{2}}(T)}.
\end{equation} On the other hand, the fractional Leibniz rule, Bernstein's inequality in time and the Sobolev embedding imply that 
\begin{equation}
   \begin{split}
        \label{22} 
        &\tilde{II}^p\\ \lesssim &\int_{0}^{T}\big\|D^s_xP_{+\text{hi}}e^{\frac{\nu}{2}iF}\big\|_{L^q_x}^p\big\|P_{\text{lo}}\big(e^{-\frac{\nu}{2}iF}v\big)\big\|_{L^\infty_x}^p+\big\|P_{+\text{hi}}e^{\frac{\nu}{2}iF}\big\|_{L^\infty_x}^p\big\|D^s_xP_\text{lo}\big(e^{-\frac{\nu}{2}iF}v\big)\big\|_{L^q_x}^p\,dt\\
        \lesssim & \int_{0}^{T}\Big\|D^{s+\frac{1}{2}-\frac{1}{q}}_xP_{+\text{hi}}e^{\frac{\nu}{2}iF}\Big\|_{L^2_x}^p\big\|P_{\text{lo}}\big(e^{-\frac{\nu}{2}iF}v\big)\big\|_{L^2_x}^p+\big\|\partial_xP_{+\text{hi}}e^{\frac{\nu}{2}iF}\big\|_{L^2_x}^p\big\|D^s_xP_\text{lo}\big(e^{-\frac{\nu}{2}iF}v\big)\big\|_{L^2_x}^p\,dt\\
        \lesssim &\,\, T\,\|\partial_x P_{+\text{hi}}e^{\frac{\nu}{2}iF}\|_{L^\infty_TL^2_x}^p\big\|P_\text{lo}\big(e^{-\frac{\nu}{2}iF}v\big)\big\|_{L^\infty_TL^2_x}^p\lesssim T \|v\|_{L^\infty_T L^2_x}^{2p}. 
    \end{split}
\end{equation} Finally, estimate \eqref{FLR'} with $\alpha_1=\alpha_2=(1+s)/2$ and $q_1=q_2=2q$, Bernstein's inequality and the Sobolev embedding lead to 
\begin{equation}
    \label{33}
    \begin{split}
        \tilde{III} &\lesssim \Big\|D^\frac{1+s}{2}_xP_{+\text{HI}}e^{\frac{\nu}{2}iF}\Big\|_{L^{2p}_TL^{2q}_x} \Big\|D^\frac{1+s}{2}_xP_{-\text{hi}}e^{-\frac{\nu}{2}iF}\Big\|_{L^{2p}_TL^{2q}_x}\\
        &\lesssim T^\frac{1}{p}\Big\|D^{1+\frac{s}{2}-\frac{1}{2q}}_xP_{+\text{HI}}e^{\frac{\nu}{2}iF}\Big\|_{L^\infty_TL^2_x} \Big\|D^\frac{s+1}{2}_xP_{-\text{hi}}e^{-\frac{\nu}{2}iF}\Big\|_{L^\infty_TL^2_x}\\
        &\lesssim T^{1/p}\big\|\partial_xP_{+\text{HI}}e^{\frac{\nu}{2}iF}\big\|_{L^\infty_TL^2_x}\big\|\partial_xP_{-\text{hi}}e^{-\frac{\nu}{2}iF}\big\|_{L^\infty_TL^2_x}\lesssim T^{1/p}\|v\|^2_{L^\infty_TL^2_x}.
    \end{split}
\end{equation} Applying \eqref{11}, \eqref{22} and \eqref{33} in \eqref{altisima}, we get 
\begin{equation}
    \label{DHI} \left\|D^s_xP_{+\text{HI}}\,v\right\|_{L^p_TL^q_x} \lesssim \left(1+\|v\|_{L^\infty_TL^2_x}\right)\big(\|w\|_{Y^{s,\frac{1}{2}}(T)}+T^{1/p}\|v\|^2_{L^\infty_TL^2_x}\big). 
\end{equation}
For the low-frequency contribution, at first we use Bernstein's inequality to obtain
\begin{equation}
    \label{PLO} \left\|P_\text{LO}v\right\|_{L^p_TL^q_x} \lesssim T^{1/p}\left\|P_\text{LO}v\right\|_{L^\infty_TL^2_x}.
\end{equation}
Next, we consider the integral equation satisfied by $P_\text{LO}\,v$,
\begin{multline}
    \label{IPLO} P_\text{LO}v(t)= \text{V}(t)P_\text{LO}v_0-\frac{\nu^2}{2}\int_0^t\text{V}(t-t')\partial_xP_\text{LO}\left(v(t')^2\right)\,dt'\\+\nu\int_0^t\text{V}(t-t')\partial_x\mathscr{Q}_\delta P_\text{LO}v(t')\,dt'+\nu\int_0^t\text{V}(t-t')\partial_x P_\text{LO}\left(|u(t')|^2\right)\,dt'.
\end{multline}
Applying inequality \eqref{PLO} to \eqref{IPLO}, it is enough to estimate the $L^\infty_TL^2_x$ norm for each term on the right hand side of \eqref{IPLO}. Indeed, using the Bernstein´s inequality, the Young´s inequality for convolution and \eqref{qdelta}, the $L^\infty_TL^2_x$ norm of the three first terms of the right hand side of \eqref{IPLO} is bounded by
\begin{equation}
    \label{Con1} \|v_0\|_{L^2_x}+T^{1+\frac{1}{p}}\|v\|_{L^\infty_TL^2_x}^2. 
\end{equation} Now, consider the $L^\infty_TL^2_x$ norm of the fourth term. Indeed, for all $0\leq t\leq T$,  
\begin{align*}
    \Big\|\int_0^t\text{V}(t-t')\partial_xP_\text{LO}\left(|u|^2\right)\,dt'\Big\|_{L^2_x}\lesssim \int_0^t\left\||u|^2\right\|_{L^2_x}\,dt'=\int_0^t\left\|u\right\|_{L^4_x}^2\,dt'\lesssim \int_0^t\|u\|^2_{H^\frac{1}{2}_x}\,dt', 
\end{align*}
where the last inequality follows from a classic Sobolev embedding. Hence,
\begin{equation}
    \label{Con3} \Big\|\int_0^t\text{V}(t-t')\partial_xP_\text{LO}\left(|u|^2\right)\,dt'\Big\|_{L^\infty_TL^2_x} \lesssim \int_0^T\|u\|^2_{H^\frac{1}{2}_x}\,dt'\lesssim T\|u\|_{X^{\frac{1}{2},b}_\text{S}(T)}^2, 
\end{equation}
where the last inequality follows from a classic Bourgain embedding. Combining \eqref{Con1}, and \eqref{Con3}, we get
\begin{equation}
    \label{PLOob}  \left\|P_\text{LO}v\right\|_{L^p_TL^q_x} \lesssim \|v_0\|_{L^2_x}+T^{1+\frac{1}{p}}\|v\|_{L^\infty_TL^2_x}^2+T\delta^{-1}\|v\|_{L^\infty_TL^2_x}+T\|u\|_{X^{\frac{1}{2},b}_\text{S}(T)}^2.
\end{equation}
Finally, from \eqref{PLOob}, \eqref{DHI} and \eqref{lowhight}, it follows the estimate \eqref{apriori2}.
 \end{proof}

\section{Nonlinear estimates}
\subsection{The nonlinear terms in w}
We now consider the equation \eqref{g}. Applying the linear estimates \eqref{2-5} and \eqref{2-7} on the Duhamel formulation of \eqref{g}, we get 
\begin{equation}
   \label{doblew} \|w\|_{Y^{s,\frac{1}{2}}(T)}\lesssim\|w(0)\|_{H^s} + \left\|\chi_{[0,T]}(t)\mathscr{N}_\delta(u,v,w)\right\|_{Y^{s,-\frac{1}{2}}}.
\end{equation} We aim to estimate each term of the right hand side of \eqref{doblew}. At first, By fractional Leibniz rule,
\begin{equation}
    \label{w_0}\|w(0)\|_{H^s_x}\lesssim \left(1+\|v_0\|_{L^2}\right)\|v_0\|_{H^s}
\end{equation}
for all $0\leq s \leq \frac{1}{2}$. On the other hand, in \cite{molinet-pilod} the first three terms of $\mathscr{N}_\delta(u,v,w)$ were considered. \\

\begin{proposition}
Let 
    \begin{multline*}
        \tilde{N}_\delta(v,w):= \nu^2 \partial_x\text{P}_{\text{+hi}}\left(\partial_x^{-1}w \, \text{P}_-\partial_xv\right)+\nu^2 \partial_x\text{P}_{\text{+hi}}\left(\text{P}_{\text{lo}}e^{-\frac{\nu}{2}iF} \, \text{P}_-\partial_xv\right)\\
    -\frac{\nu^2}{2}i \partial_x\text{P}_{\text{+hi}}\left(e^{-\frac{\nu}{2}iF} \, \mathscr{Q}_\delta v\right).
    \end{multline*}
    Then, for all $0\leq s \leq \frac{1}{2}$ there exists $\kappa>0$ such that 
    \begin{multline}
        \label{bilcha} \big\|\chi_{[0,T]}(t)\tilde{\mathscr{N}}_\delta(v,w)\big\|_{Y^{s,-\frac{1}{2}}}\lesssim \|v\|_{L^4_{T,x}}^2 \\+ \big(T^\kappa \|v\|_{L^\infty_T L^2_x}+\|v\|_{L^4_{T,x}}+\|v\|_{X^{-1,1}_\text{BO}(T)}\big)\|w\|_{X^{s,\frac{1}{2}}_\text{BO}(T)}.
    \end{multline}
\end{proposition} 
\begin{proof} [Proof]
    See the proof of Proposition 3.5 in \cite{chapouto}.
\end{proof} Now we aim to work with the fourth term of $\mathscr{N}_\delta(u,v,w)$. We start by recalling an useful estimate. \\
\begin{lemma}
\label{dif}Let $0\leq\ell_1 \leq \ell_2$ such that $\ell_1+\ell_2 > 1$. Then, for all $a,b\in\mathbb{R}$, 
\begin{equation}
    \int_{\mathbb{R}}\frac{dx}{\langle x-a\rangle^{\ell_1}\langle x-b\rangle^{\ell_2}}\lesssim\frac{1}{\langle a-b\rangle^{\ell}},
\end{equation}
with constant independent from $a$ and $b$, where $\ell=\ell_1$ if $\ell_2>1$, $\ell=\ell_1-\varepsilon$ if $\ell_2=1$, where $\varepsilon>0$ is any small number, and $\ell=\ell_1+\ell_2-1$ if $\ell_2<1$. 
\end{lemma}
\begin{proof} [Proof]
    See Lemma 4.2 of \cite{ginibre}.
\end{proof}
\begin{lemma}
    For all $s\geq0$ and $b>\frac{1}{2}$,
    \begin{equation}
        \label{duplafg} \left\|\partial_x(f\overline{g})\right\|_{L^2_{t}H^s_x}\lesssim \|f\|_{X^{s+\frac{1}{2},b}_\text{S}}\|g\|_{X^{\frac{1}{2},b}_\text{S}}+\|f\|_{X^{\frac{1}{2},b}_\text{S}}\|g\|_{X^{s+\frac{1}{2}}_\text{S}}.
    \end{equation}
\end{lemma}
\begin{proof}[Proof]
   First, consider the case $s=0$. Define 
    \begin{equation*}
                F(\tau,\xi):= \left\langle\tau+\xi^2-\nu\delta^{-1}\xi\right\rangle^b\langle\xi\rangle^\frac{1}{2}\widehat{f}(\tau,\xi),
    \end{equation*}
    and,
    \begin{equation*}
        G(\tau,\xi):= \left\langle\tau-\xi^2-\nu\delta^{-1}\xi\right\rangle^b\langle\xi\rangle^\frac{1}{2}\widehat{g}(\tau,\xi).
    \end{equation*}
    Then, using Cauchy-Schwarz inequality and Fubini theorem, we get, 
    \begin{equation*}
        \|\partial_x(fg)\|_{L^2_{t,x}}
        =\bigg\|i\xi\bigintssss_{\mathbb{R}^2}\frac{F(\tau_1,\xi_1)}{\left\langle\sigma_1\right\rangle^b\langle\xi_1\rangle^\frac{1}{2}}\frac{G(\tau_2,\xi_2)}{\left\langle\sigma_2\right\rangle^b\langle\xi_2\rangle^\frac{1}{2}}\,d\tau_1d\xi_1\bigg\|_{L^2_{\tau,\xi}}
        \leq S\, \|f\|_{X^{\frac{1}{2},b}_\text{S}}\|g\|_{X^{\frac{1}{2},b}_\text{S}},
    \end{equation*}
    where  $\tau_2:=\tau-\tau_1$, $\xi_2:=\xi-\xi_1$, $\sigma_1:=\tau_1+\xi_1^2-\nu\delta^{-1}\xi_1$, $\sigma_2:=\tau_2-\xi_2^2-\nu\delta^{-1}\xi_2$, and 
    \begin{equation*}
        S:= \Bigg\|\xi\bigg(\int_{\mathbb{R}^2}\frac{d\tau_2\,d\tau_1}{\langle\xi_1\rangle\left\langle\sigma_1\right\rangle^{2b}\langle\xi_2\rangle\left\langle\sigma_2\right\rangle^{2b}}\bigg)^{1/2}\Bigg\|_{L^\infty_{\tau,\xi}}.
    \end{equation*}
    We aim to show that $S\lesssim 1$. Indeed, by Lemma \ref{dif}, we obtain 
        \begin{equation*}
        S \lesssim \Bigg\||\xi|\bigg(\int_\mathbb{R}\frac{d\xi_1}{\left\langle2\xi\xi_1+\tau-\xi^2-\nu\delta^{-1}\xi\right\rangle^{2b}}\bigg)^{1/2}\Bigg\|_{L^\infty_{\tau,\xi}}= \Bigg\||\xi|^\frac{1}{2}\bigg(\int_\mathbb{R}\frac{1}{\langle\mu\rangle^{2b}}\frac{d\mu}{2|\xi|}\bigg)^{1/2}\Bigg\|_{L^\infty_{\tau,\xi}} \lesssim 1. 
    \end{equation*}    
It proves \eqref{normacuadradou}, when $s=0$. For the general case $s\geq0$, note that 
\begin{align*}
    &\left\|\partial_x\left(f\overline{g}\right)\right\|_{L^2_tH^s_x}\\\lesssim &\bigg\||\xi|\int_{\xi=\xi_1+\xi_2}\big|\langle\xi_1\rangle^{s}\widehat{f}(\xi_1)\big||\widehat{g}(\xi_2)|\,d\xi_1\bigg\|_{L^2_{\tau,\xi}}+ \bigg\||\xi|\int_{\xi=\xi_1+\xi_2}\big|\widehat{f}(\xi_1)\big|\big|\langle\xi_2\rangle^{s}\widehat{g}(\xi_2)\big|\,d\xi_1\bigg\|_{L^2_{\tau,\xi}}\\
    = &\big\|\partial_x\big(\big|\widehat{J^s_xf}\big|^\vee|\widehat{g}|^\vee\big)\big\|_{L^2_{t,x}}+ \big\|\partial_x(\big|\widehat{f}\big|^\vee \big|\widehat{J^s_xg}\big|^\vee)\big\|_{L^2_{t,x}}\\
    \lesssim & \big\||\widehat{J^s_xf}\big|^\vee\big\|_{X^{{1/2},b}_\text{S}}\big\||\widehat{g}|^\vee\big\|_{X^{{1/2},b}_\text{S}}+ \big\||\widehat{f}|^\vee\big\|_{X^{{{1/2}},b}_\text{S}}\big\|\big|\widehat{J^s_xg}\big|^\vee\big\|_{X^{{1/2},b}_\text{S}}\\
    =& \|f\|_{X^{s+1/2,b}_\text{S}}\|g\|_{X^{{1/2},b}_\text{S}}+\|f\|_{X^{{1/2},b}_\text{S}}\|g\|_{X^{s+{1/2}}_\text{S}}.
\end{align*}
\end{proof}
\begin{proposition} Let $0\leq s \leq \frac{1}{2}$. For all $b>\frac{1}{2}$,
        \begin{multline}
        \label{normacuadradou} \left\|\chi_{[0,T]}(t)\partial_x P_{+\text{hi}}\left(e^{-\frac{\nu}{2}iF}|u|^2\right)\right\|_{Y^{s,-\frac{1}{2}}} \\\lesssim \|u\|_{X^{\frac{1}{2},b}_{\text{S}}(T)}\big( \|v\|_{L^4_T W^{s,4}_x}+\big(\|v\|_{L^4_{T,x}}+1\big)\|u\|_{X^{s+\,{1/2},b}_{\text{S}}(T)}\big).
    \end{multline}
\end{proposition}
\begin{proof}[Proof] 
    There exists a $0<\epsilon<\frac{1}{2}$ such that 
    \begin{equation}
         \label{epi}\left\|\chi_{[0,T]}(t)\partial_x P_{+\text{hi}}\left(e^{-\frac{\nu}{2}iF}|u|^2\right)\right\|_{Y^{s,-\frac{1}{2}}} \lesssim \left\|\partial_x P_{+\text{hi}}\left(e^{-\frac{\nu}{2}iF}|u|^2\right)\right\|_{X^{s,-\frac{1}{2}+\epsilon}_{\text{BO}}(T)}. 
    \end{equation}
    Differentiating and using inequality \eqref{fracexp}, 
    \begin{equation}
         \label{ul}\begin{split}
            &\left\|\partial_x P_{+\text{hi}}\left(e^{-\frac{\nu}{2}iF}|u|^2\right)\right\|_{X^{s,-\frac{1}{2}+\epsilon}_{\text{BO}}(T)}\\
        \lesssim  & \left\|P_{+\text{hi}}\left(e^{-\frac{\nu}{2}iF}\,v|u|^2\right)\right\|_{L^2_T H^s_x} + \left\| P_{+\text{hi}}\left(e^{-\frac{\nu}{2}iF}\partial_x\left(|u|^2\right)\right)\right\|_{L^2_T H^s_x}\\
        \lesssim  & \big(1+\|v\|_{L^2_{T,x}}\big)\big(\big\|v|u|^2\big\|_{L^2_T H^s_x} + \big\| \partial_x\big(|u|^2\big)\big\|_{L^2_T H^s_x}\big).
        \end{split}
    \end{equation}
On the other hand, from the estimate \eqref{normacuadradou} we get 
    \begin{equation}
      \label{ul2}  \left\| \partial_x\left(|u|^2\right)\right\|_{L^2_T H^s_x} \lesssim \left\|u\right\|_{X^{\frac{1}{2},b}_\text{S}(T)}\|u\|_{X^{s+\frac{1}{2},b}_\text{S}(T)}. 
    \end{equation}
Assuming at first that $s>0$, we use the fractional Leibniz rule, Cauchy-Schwarz, and Sobolev and Bourgain embedding to get, 
\begin{equation*}
\label{termocade1}
    \begin{split}
                \|v|u|^2\|_{L^2_{T}H^s_x} = &\bigg(\int_0^T \left\|v|u|^2\right\|_{H^s_x}^2\bigg)^{1/2}\\
        \lesssim & \bigg(\int_0^T\left\|J^s_x v\right\|_{L^4_x}^2\,\left\|u\right\|_{L^8_x}^4\,dt\Big)^{1/2} + \bigg(\int_0^T\left\|v\right\|_{ L^4_x}^2\,\left\|u\right\|_{L^8_x}^2\left\|J^s_x u\right\|_{L^8_x}^2\,dt\bigg)^{1/2}\\
        \leq &\|J^s_x v\|_{L^4_{T,x}}\bigg(\int_0^T\|u\|_{H^{1/2}_x}^8\,dt\bigg)^{1/4}+ \|v\|_{L^4_{T,x}}\bigg(\int_0^T\|u\|_{H^{1/2}_x}^4\|u\|_{H^{s+{1/2}}_x}^4\bigg)^{1/4}\\
        \leq & T^{1/4}\|J^s_x v\|_{L^4_{T,x}}\|u\|_{L^\infty_T H^{1/2}_x}^2+ T^{1/4}\|v\|_{L^4_{T,x}}\|u\|_{L^\infty_T H^{1/2}_x}\|u\|_{L^\infty_T H^{s+{1/2}}_x}\\
        \lesssim & T^{1/4}\|J^s_x v\|_{L^4_{T,x}}\left\|u\right\|_{X^{{1/2},b}_\text{S}(T)}^2+ T^{1/4}\|v\|_{L^4_{T,x}}\|u\|_{X^{{1/2},b}_\text{S}(T)}\|u\|_{X^{s+{1/2},b}_\text{S}(T)}. 
    \end{split}
\end{equation*}
The case $s=0$ can be worked in a similar way.  Combining \eqref{ul} and \eqref{ul2}, we obtain \eqref{epi}. 
\end{proof} Applying estimates \eqref{w_0}, \eqref{bilcha} and \eqref{normacuadradou} in \eqref{doblew}, we get the following result. \\
\begin{proposition}
    For all $0\leq s \leq \frac{1}{2}$ and $b>\frac{1}{2}$, there is any $\kappa>0$ such that 
    \begin{multline}
        \label{3gauge} \|w\|_{Y^{s,\frac{1}{2}}(T)}\lesssim (1+\left\|v_0\right\|_{L^2})\left\|v_0\right\|_{H^s}+\|v\|^2_{L^4_{T,x}} \\
        +\big(T^\kappa\|v\|_{L^\infty_T L^2_x}+\|v\|_{L^4_{T,x}}+\|v\|_{X^{-1,1}_\text{BO}(T)}\big)\|w\|_{X^{s,\frac{1}{2}}_\text{BO}(T)}\\+\|u\|_{X^{\frac{1}{2},b}_{\text{S}}(T)}\big( \|v\|_{L^4_T W^{s,4}_x}+\big(\|v\|_{L^4_{T,x}}+1\big)\|u\|_{X^{s+\frac{1}{2},b}_{\text{S}}(T)}\big).
    \end{multline}
\end{proposition}

\subsection{Estimate of the cubic term}
We now turn to the first equation of \eqref{Schr-ILW-T}. We need some previous results.  \\
\begin{lemma}
    The group $\{U(t)\}_{t\in\mathbb{R}}$ satisfies the Strichartz estimate 
    \begin{equation}
       \label{Strichartz 2} \left\|U(t)f\right\|_{L^q_t L^p_x}\lesssim \|f\|_{L^2_x},
    \end{equation}
    where $2\leq p \leq \infty$ and $\frac{2}{q}=\frac{1}{2}-\frac{1}{p}$. 
\end{lemma}
\begin{proof} [Proof]
See the proof of Lemma 4.1 in \cite{linares2015introduction}. See also \cite{kenig1991oscillatory}. 
\end{proof}

\begin{proposition} For all $s\geq0$ and $b>\frac{1}{2}$, we get 
    \begin{multline}
    \label{trilineal}\left\|fgh\right\|_{L^2_{t}H^s_x}\lesssim \|f\|_{X^{s,b}_\text{S}}\,\|g\|_{X^{0,b}_\text{S}}\,\|h\|_{X^{0,b}_\text{S}}\\+\|f\|_{X^{0,b}_\text{S}}\,\|g\|_{X^{s,b}_\text{S}}\,\|h\|_{X^{0,b}_\text{S}}+\|f\|_{X^{0,b}_\text{S}}\,\|g\|_{X^{0,b}_\text{S}}\,\|h\|_{X^{s,b}_\text{S}}.
    \end{multline}
\end{proposition}
\begin{proof}[Proof]
    At first we prove \eqref{trilineal} when $s=0$. Indeed, by a generalized Hölder inequality,
    \begin{equation}
        \label{L6} \left\|fgh\right\|_{L^2_{t,x}}\lesssim \|f\|_{L_{t,x}^6}\,\|g\|_{L_{t,x}^6}\,\|h\|_{L_{t,x}^6}. 
    \end{equation}
    Now, let us define 
    \begin{align*}
        &F(\tau,\xi)= \left\langle \tau +\xi^2-\nu\delta^{-1}\xi\right\rangle^b\widehat{f}(\tau,\xi), \,\,\,\,G(\tau,\xi)= \left\langle \tau +\xi^2-\nu\delta^{-1}\xi\right\rangle^b\widehat{g}(\tau,\xi), \, \, \, \text{and,}\\
        &H(\tau,\xi)= \left\langle \tau +\xi^2-\nu\delta^{-1}\xi\right\rangle^b\widehat{h}(\tau,\xi).
    \end{align*}
    By the Fourier inversion formula, we get 
    \begin{align*}
        \|f\|_{L_{t,x}^6} &= \Big\|\bigintssss_{\mathbb{R}}\bigintssss_{\mathbb{R}}\frac{F(\tau,\xi)}{\left\langle \tau +\xi^2-\nu\delta^{-1}\xi\right\rangle^b}e^{i\tau t}e^{i\xi x}\,d\tau\,d\xi\Big\|_{L_{t,x}^6}\\
        &= \Big\|\bigintssss_{\mathbb{R}}\bigintssss_{\mathbb{R}}\frac{F(\sigma-\xi^2+\nu\delta^{-1},\xi)}{\left\langle \sigma\right\rangle^b}e^{i\left(\sigma-\xi^2+\nu\delta^{-1}\right) t}e^{i\xi x}\,d\sigma\,d\xi\Big\|_{L_{t,x}^6}\\
        &= \Big\|\bigintssss_{\mathbb{R}}\frac{e^{i\sigma t}}{\left\langle \sigma\right\rangle^b}\bigintssss_{\mathbb{R}}e^{-i\left(\xi^2-\nu\delta^{-1}\right) t}F(\sigma-\xi^2+\nu\delta^{-1},\xi)\,e^{i\xi x}\,d\xi\,d\sigma\Big\|_{L_{t,x}^6}.
    \end{align*}
    Let $\widehat{\tilde{F}_\sigma}(\xi):= F(\sigma-\xi^2+\nu\delta^{-1},\xi)$. Using the Fourier inversion formula, then the Minkowski integral inequality, after that the Strichartz estimate \eqref{Strichartz 2} with $p=q=6$, and finally the Cauchy-Schwarz inequality, we obtain 
    \begin{align*}
        \|f\|_{L_{t,x}^6}&= \Big\|\bigintssss_{\mathbb{R}}\frac{e^{i\sigma t}}{\langle\sigma\rangle^b}\,U(t)\tilde{F}_\sigma\,(x)\,d\sigma\Big\|_{L^6_{t,x}}\\
        &\leq \bigintssss_{\mathbb{R}}\frac{\|U(t)\tilde{F}_\sigma\|_{L_{t,x}^6}}{\langle\sigma\rangle^b}\,d\sigma \lesssim \bigintssss_{\mathbb{R}}\frac{\|\tilde{F}_\sigma\|_{L_{x}^2}}{\langle\sigma\rangle^b}\,d\sigma \leq \Big(\bigintssss_{\mathbb{R}}\frac{d\sigma}{\langle\sigma\rangle^{2b}}\Big)^{1/2}\|f\|_{X^{0,b}_S}.
    \end{align*} A similar argument can be used to show that $\|g\|_{L^6_{t,z}}\lesssim \|g\|_{X^{0,b}_S}$ and $\|h\|_{L^6_{t,z}}\lesssim \|h\|_{X^{0,b}_S}$. Hence, when $s=0$, the estimate \eqref{trilineal} can be obtained as a consequence of inequality \eqref{L6}. Proceeding as in the last part of the proof of \eqref{duplafg}, we can obtain the estimate \eqref{trilineal} can be obtained for the general case $s\geq0$. \end{proof}

\begin{corollary}
    Let $s\geq-\frac{1}{2}$, $b>0$ and $a<0$. Then
    \begin{equation}
        \label{cubo} \||u|^2u\|_{X^{s+\frac{1}{2},a}_\text{S}(T)}\lesssim \|u\|_{X^{0,b}_\text{S}(T)}^2\|u\|_{X^{s+\frac{1}{2},b}_\text{S}(T)}. 
    \end{equation}
\end{corollary}
\label{uv}\subsection{The nonlinear interaction $uv$}
Although the bilinear estimate \eqref{be} accounts for the nonlinearity $uv$ in \eqref{Shcr-quaseILW-T}, it breaks down for \eqref{Schr-ILW-T}. To bridge this gap, one of our main technical challenges is proving a refined bilinear estimate. To this end, we first state the following auxiliary result.\\
\begin{lemma} \label{mio}Let $b,c\in\mathbb{R}$ and $\ell_1,\ell_2\geq0$ be such that $\ell_1+\ell_2>\frac{1}{2}$. Then
    \begin{equation}
        \bigintssss_{\mathbb{R}}\frac{dx}{\left\langle x\right\rangle ^{2\ell_1}\left\langle x^2+bx+c\right\rangle ^{\ell_2}}\lesssim 1,
    \end{equation} 
    where the constant is independent from $b$ and $c$. 
\end{lemma}
\begin{proof} [Proof]
    Let $\beta,\gamma\in\mathbb{C}$ be such that $x^2+bx+c=(x-\beta)(x-\gamma)$ for all $x\in\mathbb{R}$. Note that for any $y\in\mathbb{R}$, $|x-iy|=\sqrt{x^2+y^2}\leq|x|$. Then 
    \begin{align*}
        &\bigintssss_{\mathbb{R}}\frac{dx}{\left\langle x\right\rangle ^{2\ell_1}\left\langle x^2+bx+c\right\rangle ^{\ell_2}}\\
        =& \bigintssss_{\mathbb{R}}\frac{dx}{\left\langle x\right\rangle ^{2\ell_1}\left\langle (x-\beta)(x-\gamma)\right\rangle ^{\ell_2}}\\
        = & \Bigg[\bigintssss_{\substack{|x|\leq|x-\beta| \\ |x|\leq|x-\gamma|}}+ \bigintssss_{\substack{|x-\beta|\leq|x| \\ |x-\beta|\leq|x-\gamma|}}+ \bigintssss_{\substack{|x-\gamma|\leq|x| \\ |x-\gamma|\leq|x-\beta|}}\Bigg] \frac{dx}{\left\langle x\right\rangle ^{2\ell_1}\left\langle (x-\beta)(x-\gamma)\right\rangle ^{\ell_2}} \lesssim 1,
    \end{align*}
    which concludes the proof. 
\end{proof}   
We now state our main result. \\
\begin{proposition} 
Let $s\geq0$, $b>\frac{1}{2}$, $-\frac{1}{2}<a<-\frac{3}{8}$, and $1-2|a|<\theta<2|a|-\frac{1}{2}$. Suppose that $|\nu|\not=1$. Then
    \begin{equation}
        \label{estimativa bilineal} \|uv\|_{X^{s+\frac{1}{2},a}_{\text{S}}}\lesssim \|u\|_{X^{s+\frac{1}{2},\, b}_{\text{S}}}\|v\|_{X^{-\theta, \,\theta}_{\text{BO}}} + \|u\|_{X^{\frac{1}{2},\, b}_{\text{S}}}\|v\|_{X^{s-\theta, \,\theta}_{\text{BO}}}. 
    \end{equation}
\end{proposition}

\begin{proof} [Proof] Proceeding as the last part of the proof of \eqref{duplafg}, it suffices to prove \eqref{estimativa bilineal} for $s=0$. Indeed, define 
    \begin{equation*}
        f(\tau,\xi):= \left\langle \tau+\xi^2-\nu \delta^{-1}\xi\right\rangle^b\left\langle\xi\right\rangle^{\frac{1}{2}}\widehat{u}(\tau,\xi), 
    \end{equation*}
    and
    \begin{equation*}
        g(\tau,\xi):= \left\langle \tau-\nu\xi|\xi|\right\rangle^\theta\left\langle\xi\right\rangle^{-\theta}\widehat{v}(\tau,\xi).
    \end{equation*} Note that $\|f\|_{L^2\left(\mathbb{R}^2\right)}=\|u\|_{X^{1/2,b}_{\text{S}}}$ and $\|g\|_{L^2\left(\mathbb{R}^2\right)}=\|v\|_{X^{-\theta, \,\theta}_{\text{BO}}}$. Moreover, using a duality argument, we have
    \begin{equation}
                \label{produtodual}\|uv\|_{X^{1/2,a}_{\text{S}}} =\sup_{\|h\|_{L^2\left(\mathbb{R}^2\right)}=1}\left|\int_{\mathbb{R}^4}\Omega(\tau_1,\xi_1,\tau,\xi)f(\tau_2,\xi_2)g(\tau_1,\xi_1)h(\tau,\xi)\,dV\right|,
    \end{equation}
where $\tau_2:=\tau-\tau_1$, $\xi_2=\xi-\xi_1$, $dV:=d\tau_1\,d\xi_1\,d\tau\,d\xi$, and
\begin{equation}
   \label{Omega} \Omega(\tau_1,\xi_1,\tau,\xi):=\frac{\langle\xi\rangle^{\frac{1}{2}}}{\langle\sigma\rangle^{|a|}\langle\xi_1\rangle^{-\theta}\langle\sigma_1\rangle^{\theta}\langle\xi_2\rangle^{\frac{1}{2}}\langle\sigma_2\rangle^{b}}, 
\end{equation}
with 
   $\sigma:=\tau+\xi^2-\nu\delta^{-1}\xi$, \, $\sigma_1:= \tau_1-\nu\xi_1|\xi_1|,$ and $\sigma_2=\tau_2+\xi^2_2-\nu\delta^{-1}\xi_2.$ We decompose $\mathbb{R}^4$ as the union of the following sets:
\begin{align*}
    &A:= \left\{(\tau_1,\xi_1,\tau,\xi)\in\mathbb{R}^4: |\xi_1|\leq1\right\},\\
&B:= \left\{(\tau_1,\xi_1,\tau,\xi)\in A^c: |\sigma|\geq|\sigma_1|, |\sigma_2|\right\},\\
 &C:= \left\{(\tau_1,\xi_1,\tau,\xi)\in  A^c: |\sigma_1|\geq|\sigma|, |\sigma_2|\right\}, \, \text{and,}\\
&D:= \left\{(\tau_1,\xi_1,\tau,\xi)\in  A^c : |\sigma_2|\geq|\sigma|, |\sigma_1|\right\}.
\end{align*} From equality \eqref{produtodual} we have
\begin{equation}
     \label{est}\|uv\|_{X^{1/2,a}_{\text{S}}} \leq I_A+I_B+I_C+I_D,
\end{equation}
where 
\begin{equation*}
    I_U:= \sup_{\|h\|_{L^2\left(\mathbb{R}^2\right)}=1}\Big|\int_{\mathbb{R}^4}\left(\chi_{U}\Omega\right)(\tau_1,\xi_1,\tau,\xi)f(\tau_2,\xi_2)g(\tau_1,\xi_1)h(\tau,\xi)\,dV\Big|,
\end{equation*}
for $U\in\{A,B,C,D\}.$ We aim to prove that each term on the right-hand side of  \eqref{est} is bounded by $C\|u\|_{X^{\frac{1}{2},a}}\|v\|_{X^{-\theta,\theta}}$ for some constant $C>0$. We need the so-called resonance function: 
\begin{equation}
    \label{resonance function}
    \begin{split}
        R:=-\sigma+\sigma_1+\sigma_2&= \left(1-\nu \sgn(\xi_1) \right)\xi^2_1-\left(2\xi-\nu\delta^{-1}\right)\xi_1 \\&= -\left(1+\nu \sgn(\xi_1) \right)\xi^2_1-\left(2\xi_2-\nu\delta^{-1}\right)\xi_1.
    \end{split} 
\end{equation}  We divide each set $U\in\left\{B, C, D\right\}$ into the following regions:
\begin{align*}
    &U_1:= \Big\{ (\tau_1,\xi_1,\tau,\xi)\in U: |R| \leq \frac{|1-|\nu||}{2}\xi_1^2 \Big\},\, \text{and} \\
    &U_2:= \Big\{ (\tau_1,\xi_1,\tau,\xi)\in U: |R| \geq \frac{|1-|\nu||}{2}\xi_1^2 \Big\}. 
\end{align*}
Using $\langle\xi\rangle\lesssim\langle\xi_1\rangle\langle\xi_2\rangle$, we get 
\begin{equation}
    \label{omega usual} \Omega(\tau_1,\xi_1,\tau,\xi)\lesssim\frac{\langle\xi_1\rangle^{\theta+\frac{1}{2}}}{\langle\sigma\rangle^{|a|}\langle\sigma_1\rangle^{\theta}\langle\sigma_2\rangle^{b}}.
\end{equation}
In $U_1$ a better estimate for $\Omega$ can be obtained: let us see that $\langle\xi\rangle\lesssim\langle\xi_2\rangle$ in $U_1$. In fact, 
\begin{align*}
    2|\xi| &= \left|\left(1-\nu\sgn(\xi_1)\right)\xi_1-\left(2\xi-\nu\delta^{-1}\right)-\left(1-\nu\sgn(\xi_1)\right)\xi_1-\nu\delta^{-1}\right|\\
            &\leq \left|\left(1-\nu\sgn(\xi_1)\right)\xi_1-\left(2\xi-\nu\delta^{-1}\right)\right| + \left(1+|\nu|\right)|\xi_1|+|\nu|\delta^{-1}\\
            &\leq\left(\frac{\left|1-|\nu|\right|}{2}+1+|\nu|\right)|\xi_1|+|\nu|\delta^{-1},
\end{align*}
and
\begin{align*}
    2|\xi_2| &= \left|-\left(1+\nu\sgn(\xi_1)\right)\xi_1-\left(2\xi_2-\nu\delta^{-1}\right)+\left(1+\nu\sgn(\xi_1)\right)\xi_1-\nu\delta^{-1}\right|\\
    &\geq \left|\left(1+\nu\sgn(\xi_1)\right)\xi_1-\nu\delta^{-1}\right|-\left|-\left(1+\nu\sgn(\xi_1)\right)\xi_1-\left(2\xi_2-\nu\delta^{-1}\right)\right|\\
    &\geq \left|1-|\nu|\right||\xi_1|-|\nu|\delta^{-1}-\frac{\left|1-|\nu|\right|}{2}|\xi_1|\\
    &= \frac{\left|1-|\nu|\right|}{2}|\xi_1|-|\nu|\delta^{-1},
\end{align*}
implies $\langle\xi\rangle\lesssim_{|\nu|,\delta^{-1}}\langle\xi_2\rangle$ in $U_1$, as we wanted. As a consequence, 
\begin{equation}
    \label{Omega U1} \left(\chi_{U_1}\Omega\right)(\tau_1,\xi_1, \tau, \xi) \lesssim \frac{\langle\xi_1\rangle^{\theta}}{\langle\sigma\rangle^{|a|}\langle\sigma_1\rangle^{\theta}\langle\sigma_2\rangle^{b}}.
\end{equation} Let us make another observation about $U_1$: in this region, 
\begin{align*}
&\left|1-|\nu|\right||\xi_1|\\\leq&\left|\left(1-\nu\sgn(\xi_1)\right)\xi_1\right|\\
\leq &\left|2\left(1-\nu\sgn(\xi_1)\right)\xi_1-\left(2\xi-\nu\delta^{-1}\right)\right| + \left|\left(2\xi-\nu\delta^{-1}\right)-\left(1-\nu\sgn(\xi_1)\right)\xi_1\right|\\
\leq &\left|2\left(1-\nu\sgn(\xi_1)\right)\xi_1-\left(2\xi-\nu\delta^{-1}\right)\right| + \frac{\left|1-|\nu|\right|}{2}|\xi_1|,
\end{align*}
where in the last inequality, we used the definition of $U_1$. Hence,
\begin{equation}
    \label{sustficti}  \frac{\left|1-|\nu|\right|}{2}|\xi_1| \leq \left|2\left(1-\nu\sgn(\xi_1)\right)\xi_1-\left(2\xi-\nu\delta^{-1}\right)\right| \, \, \, \, \text{in} \, \, \, \, U_1. 
\end{equation} Now we return to the problem of bounding each term on the right hand side of \eqref{est}. Let us start with $I_A$ and $I_B$. In fact, using Cauchy-Schwarz inequality and the Fubini theorem, we obtain  
\begin{align*}
    &\Big|\int_{\mathbb{R}^4}\left(\chi_{A}\Omega\right)(\tau_1,\xi_1,\tau,\xi)f(\tau_2,\xi_2)g(\tau_1,\xi_1)h(\tau,\xi)\,dV\Big|\\
    \leq &\Big\|\int_{\mathbb{R}^2}\left(\chi_{A}\Omega\right)(\tau_1,\xi_1,\tau,\xi)f(\tau_2,\xi_2)g(\tau_1,\xi_1)\,d\tau_1\,d\xi_1\Big\|_{L^2_{\tau,\xi}}\|h\|_{L^2\left(\mathbb{R}^2\right)}\\
    \leq &\bigg\|\Big(\int_{\mathbb{R}^2}|\chi_{A}\Omega|^2\,d\tau_1\,d\xi_1\Big)^{{1/2}}\left(\int_{\mathbb{R}^2}\left|f(\tau_2,\xi_2)\right|^2\left|g(\tau_1,\xi_1)\right|^2\,d\tau_1\,d\xi_1\right)^{{1/2}}\bigg\|_{L^2_{\tau,\xi}}\|h\|_{L^2\left(\mathbb{R}^2\right)}\\
    \leq &\|\chi_{A}\Omega\|_{L^\infty_{\tau,\xi}\,L^2_{\tau_1,\xi_1}}\left\|f\right\|_{L^2(\mathbb{R}^2)}\left\|g\right\|_{L^2(\mathbb{R}^2)}\|h\|_{L^2\left(\mathbb{R}^2\right)},
\end{align*}
which implies that $I_{A}\leq\left\|\chi_{A}\Omega\right\|_{L^\infty_{\tau,\xi}\,L^2_{\tau_1,\xi_1}}\left\|f\right\|_{L^2(\mathbb{R}^2)}\left\|g\right\|_{L^2(\mathbb{R}^2)}$. In a similar way, we can show that $I_{B}\leq\left\|\chi_{ B}\Omega\right\|_{L^\infty_{\tau,\xi}\,L^2_{\tau_1,\xi_1}}\left\|f\right\|_{L^2(\mathbb{R}^2)}\left\|g\right\|_{L^2(\mathbb{R}^2)}$. Then, we aim to prove that 
\begin{equation}
    \label{B1}\left\|\chi_{A}\Omega\right\|_{L^\infty_{\tau,\xi}\,L^2_{\tau_1,\xi_1}}, \left\|\chi_{B}\Omega\right\|_{L^\infty_{\tau,\xi}\,L^2_{\tau_1,\xi_1}}\lesssim1. 
\end{equation} In fact, by \eqref{omega usual}
\begin{align*}
    \iint_{\mathbb{R}^2}\left|\chi_A\Omega\right|^2\,d\tau_1\,d\xi_1
    \lesssim \frac{1}{\langle\sigma\rangle^{2|a|}}\int_{-1}^{1}\int_{\mathbb{R}}\frac{\langle\xi_1\rangle^{2\theta+1}}{\langle\sigma_1\rangle^{2\theta}\langle\sigma_2\rangle^{2b}}\,d\tau_1\,d\xi_1
    \lesssim \frac{1}{\langle\sigma\rangle^{|a|}}\int_{-1}^{1}\frac{\langle\xi_1\rangle^{2\theta+1}}{\left\langle R+\sigma \right\rangle^{2\theta}}\,d\xi_1,
\end{align*}
where the last inequality was obtained due to Lemma \ref{dif} and identity \eqref{resonance function}. Since the integration domain in the last integral has finite measure, 
\begin{equation*}
    \iint_{\mathbb{R}^2}\left(\chi_A\Omega\right)(\tau_1,\xi_1,\tau,\xi)^2\,d\tau_1\,d\xi_1
    \lesssim 1.
\end{equation*}
Taking square roots and the supremum over $\tau$ and $\xi$, we get 
\begin{equation}
    \label{B1.1}\left\|\chi_{A}\Omega\right\|_{L^\infty_{\tau,\xi}\,L^2_{\tau_1,\xi_1}}\lesssim1.
\end{equation} Now, we turn to the set $B$. First, consider $B_1$. Using inequality \eqref{Omega U1}, Lemma \ref{dif} and the fact that $|R|=|-\sigma+\sigma_1+\sigma_2|\leq3|\sigma|$ in $B$, we get
\begin{align*}
    \iint_{\mathbb{R}^2}\left|\left(\chi_{B_1}\Omega\right)\right|^2\,d\tau_1d\xi_1 \lesssim &\frac{1}{\langle\sigma\rangle^{2|a|}}\iint\frac{\langle\xi_1\rangle^{2\theta}}{\langle\sigma_1\rangle^{2\theta}\langle\sigma_2\rangle^{2b}}\,d\tau_1d\xi_1\\
    \lesssim &\int \frac{\langle\xi_1\rangle^{2\theta}}{\langle R\rangle^{2|a|}\langle R+\sigma\rangle^{2\theta}}\,d\xi_1\\
    \leq &\int_{|R|\leq |R+\sigma|} \frac{\langle\xi_1\rangle^{2\theta}}{\langle R\rangle^{2\left(|a|+\theta\right)}}\,d\xi_1 + \int_{|R|\geq |R+\sigma|} \frac{\langle\xi_1\rangle^{2\theta}}{\langle R+\sigma\rangle^{2\left(|a|+\theta\right)}}\,d\xi_1\\
    \lesssim &\bigintssss \frac{\langle\xi_1\rangle^{2\theta}}{\langle R+p(\sigma)\rangle^{2\left(|a|+\theta\right)}}\,d\xi_1,
\end{align*}
where $p(\sigma)\in\{0,\sigma\}$. Making the substitution $\mu= R + p(\sigma)$ in the last integral, and using inequality \eqref{sustficti}, we get 
\begin{align*}
     \bigintssss \frac{\langle\xi_1\rangle^{2\theta}}{\langle R+p(\sigma)\rangle^{2\left(|a|+\theta\right)}}
     =& \bigintssss \frac{\langle\xi_1\rangle^{2\theta}}{\langle \mu\rangle^{2\left(|a|+\theta\right)}}\frac{d\mu}{\left|2\left(1-\nu\sgn(\xi_1)\right)\xi_1-\left(2\xi-\nu\delta^{-1}\right)\right|}\\
     \lesssim & \bigintssss \frac{1}{\langle\mu\rangle^{2\left(|a|+\theta\right)}}\frac{d\mu}{\langle\xi_1\rangle^{1-2\theta}}
     \leq \int_{\mathbb{R}}\frac{d\mu}{\langle\mu\rangle^{2\left(|a|+\theta\right)}}
     \lesssim1,
\end{align*} where for the last inequality we observe that $2\left(|a|+\theta\right)=\left(2|a|+\theta\right)+\theta>1+\theta>1$. \\ Then
\begin{equation}
    \label{chiB1} \left\|\chi_{B_1}\Omega\right\|_{L^\infty_{\tau,\xi}\,L^2_{\tau_1,\xi_1}} \lesssim 1. 
\end{equation}
Now, consider the region $B_2$. Using inequality \eqref{omega usual}, Lemma \ref{dif} and the fact that in $B_2$, $\xi_1^2\lesssim |R| = |-\sigma+\sigma_1+\sigma_2|\leq 3|\sigma|$, 
\begin{align*}
     \iint_{\mathbb{R}^2}\left|\chi_{B_2}\Omega\right|^2\,d\tau_1d\xi_1
     \lesssim  \frac{1}{\langle\sigma\rangle^{2|a|}}\int\frac{\langle\xi_1\rangle^{2\theta+1}}{\langle R + \sigma\rangle^{2\theta}}d\xi_1 \lesssim \int_{\mathbb{R}} \frac{d\xi_1}{\langle\xi_1\rangle^{4|a|-1-2\theta}\langle R + \sigma \rangle^{2\theta}}\lesssim1,
\end{align*}
where the last inequality is due to Lemma \ref{mio}, taking $\ell_1=2|a|-\frac{1}{2}-\theta$ and $\ell_2=2\theta$. Then 
\begin{equation}
    \label{chiB2} \left\|\chi_{B_2}\Omega\right\|_{L^\infty_{\tau,\xi}\,L^2_{\tau_1,\xi_1}} \lesssim 1. 
\end{equation}
Combining inequalities \eqref{B1.1}, \eqref{chiB1} and \eqref{chiB2}, we obtain the estimate \eqref{B1}.
Now, we leave to the region $C$. Note that 
\begin{align*}
    &\Big|\int_{\mathbb{R}^4}\left(\chi_{C}\Omega\right)(\tau,\tau_1,\xi,\xi_1)f(\tau_2,\xi_2)g(\tau_1,\xi_1)h(\tau,\xi)\,dV\Big|\\
        \leq &\bigg\|\Big(\int_{\mathbb{R}^2}\left|\chi_{C}\Omega\right|^2\,d\tau\,d\xi\Big)^{{1/2}}\Big(\int_{\mathbb{R}^2}\left|f(\tau_2,\xi_2)\right|^2\left|h(\tau,\xi)\right|^2\,d\tau\,d\xi\Big)^{{1/2}}\bigg\|_{L^2_{\tau_1,\xi_1}}\|g\|_{L^2\left(\mathbb{R}^2\right)}\\
    \leq &\left\|\chi_{C}\Omega\right\|_{L^\infty_{\tau_1,\xi_1}\,L^2_{\tau,\xi}}\left\|f\right\|_{L^2(\mathbb{R}^2)}\left\|g\right\|_{L^2(\mathbb{R}^2)}\|h\|_{L^2\left(\mathbb{R}^2\right)},
\end{align*}
obtaining $I_C \leq \left\|\chi_{D}\Omega\right\|_{L^\infty_{\tau_1,\xi_1}\,L^2_{\tau,\xi}}\left\|f\right\|_{L^2(\mathbb{R}^2)}\left\|g\right\|_{L^2(\mathbb{R}^2)}$. Therefore, let us prove that
\begin{equation}
    \label{C}\left\|\chi_{C}\Omega\right\|_{L^\infty_{\tau_1,\xi_1}\,L^2_{\tau,\xi}}\lesssim1. 
\end{equation}
First, consider $C_1$. Using inequality \eqref{omega usual} and Lemma \ref{dif},
\begin{align*}
\iint \left|\chi_{C_1}\Omega\right|^2 d\tau\,d\xi
    \lesssim & \frac{\langle\xi_1\rangle^{2\theta}}{\langle\sigma_1\rangle^{2\theta}}\int_{|\sigma_2-\sigma|\leq 2 |\sigma_1|}\frac{ d\xi}{\langle R-\sigma_1\rangle^{2|a|}}\\
    \lesssim &\frac{\langle\xi_1\rangle^{2\theta}}{\langle\sigma_1\rangle^{2\theta}} \int_{|\mu|\leq 2|\sigma_1|} \frac{1}{\langle \mu \rangle^{2|a|}}\frac{d\mu}{2|\xi_1|}\\
    \sim&\frac{1}{\langle\xi_1\rangle^{1-2\theta}\langle\sigma_1\rangle^{2\theta}}\left[\frac{(1+\mu)^{-2|a|+1}}{-2|a|+1}\right]_{0}^{2|\sigma_1|}
    \lesssim \frac{1}{\langle\xi_1\rangle^{1-2\theta}\langle\sigma_1\rangle^{2|a|-1+2\theta}}
    \lesssim 1.
\end{align*} For $C_2$ we make similar considerations, but taking into account that in this region $\xi_1^2\lesssim |R| = |-\sigma+\sigma_1+\sigma_2|\leq 3|\sigma_1|$. Indeed 
\begin{align*}
\iint_{\mathbb{R}^2}\left|\chi_{C_2}\Omega\right|^2\,d\tau_1d\xi_1
     &\lesssim  \frac{\langle\xi_1\rangle^{2\theta+1}}{\langle\sigma_1\rangle^{2\theta}}\int\frac{d\xi}{\langle R - \sigma_1\rangle^{2|a|}} \\&\lesssim \frac{\langle\xi_1\rangle^{2\theta}}{\langle\sigma_1\rangle^{2|a|-1+2\theta}}\lesssim \frac{1}{\langle\xi_1\rangle^{4|a|-2+2\theta}}\lesssim1,
\end{align*} 
It shows the estimate \eqref{C}. Finally, let us consider $I_D$. Indeed, explicitly making the change of variables $\xi_2=\xi-\xi_1$ and $\tau_2=\tau-\tau_1$, 
\begin{align*}
     &I_D=\sup_{\|h\|_{L^2}=1}\Big|\int_{\mathbb{R}^4}\big(\chi_{\tilde{D}}\tilde{\Omega}\big)(\tau_2,\tau_1,\xi_2,\xi_1)f(\tau_2,\xi_2)g(\tau_1,\xi_1)h(\tau,\xi)\,d\tilde{V}\Big|,
\end{align*}
where $d\tilde{V}:=d\tau_1\,d\xi_1\,d\tau_2\,d\xi_2$, $\tilde{\Omega}(\tau_1,\xi_1,\tau_2,\xi_2):=\Omega(\tau_1,\xi_1,\tau_1+\tau_2,\xi_1+\xi_2)$, and $\tilde{D}:=\left\{(\tau_1,\xi_1,\tau_2,\xi_2)\in\mathbb{R}^4: (\tau_1,\xi_1,\tau_1+\tau_2,\xi_1+\xi_2)\in D\right\}$. Define $\tilde{D_1}$ and $\tilde{D_2}$ in a similar way.  \\ \\ Using Cauchy-Schwarz and Fubini's theorem , $$I_D\lesssim\big\|\chi_{\tilde{D}}\tilde{\Omega}\big\|_{L^\infty_{\tau_2,\xi_2}\,L^2_{\tau_1,\xi_1}}\left\|f\right\|_{L^2(\mathbb{R}^2)}\left\|g\right\|_{L^2(\mathbb{R}^2)}.$$ Therefore we have to prove that 
\begin{equation}
    \label{E} \big\|\chi_{\tilde{D}}\tilde{\Omega}\big\|_{L^\infty_{\tau_2,\xi_2}\,L^2_{\tau_1,\xi_1}}\lesssim1.
\end{equation}
Consider $D_1$. Using inequality \eqref{Omega U1}, Lemma \ref{dif} and the fact that $|R|\leq3|\sigma_2|$ in $D$, 
\begin{align*}
   &\iint_{\mathbb{R}^2}|\chi_{\tilde{D_1}}\tilde{\Omega}|^2\,d\tau_1\,d\xi_1\\\lesssim &\frac{1}{\langle\sigma_2\rangle^{2b}}\iint\frac{\langle\xi_1\rangle^{2\theta}}{\langle\sigma\rangle^{2|a|}\langle\sigma_1\rangle^{2\theta}}\,d\tau_1 d\xi_1\\
   \lesssim &\int\frac{\langle\xi_1\rangle^{2\theta}}{\langle R\rangle^{2b}\langle R + \sigma_2\rangle^{2\left(|a|+\theta\right)-1}}\,d\xi_1\\
   \lesssim &\int_{|R|\leq |R+\sigma_2|}\frac{\langle\xi_1\rangle^{2\theta}}{\langle R \rangle^{2\left(|a|+\theta\right)+\left(2b-1\right)}}\,d\xi_1 + \int_{|R|\geq |R+\sigma_2|}\frac{\langle\xi_1\rangle^{2\theta}}{\langle R + \sigma_2\rangle^{2\left(|a|+\theta\right)+\left(2b-1\right)}}\,d\xi_1\\
   \lesssim&\bigintssss\frac{\langle\xi_1\rangle^{2\theta}}{\langle R + p(\sigma_2)\rangle^{2\left(|a|+\theta\right)+\left(2b-1\right)}}\,d\xi_1,
\end{align*}
where $p(\sigma_2)\in\{0,\sigma_2\}$. Making the substitution $\mu= R + \sigma_2$ in the last integral, and using inequality \eqref{sustficti},  
\begin{align*}
    &\bigintssss\frac{\langle\xi_1\rangle^{2\theta}}{\langle R + p(\sigma_2)\rangle^{2\left(|a|+\theta\right)+\left(2b-1\right)}}\,d\xi_1\\
    = & \bigintssss\frac{\langle\xi_1\rangle^{2\theta}}{\langle \mu\rangle^{2\left(|a|+\theta\right)+\left(2b-1\right)}}\frac{d\mu}{\left|2\left(1+\nu\sgn(\xi_1)\right)\xi_1+\left(2\xi_2-\nu\delta^{-1}\right)\right|}\\
    \lesssim & \bigintssss\frac{1}{\langle \mu\rangle^{2\left(|a|+\theta\right)+\left(2b-1\right)}}\frac{d\mu}{\langle\xi_1\rangle^{1-2\theta}}
    \leq \bigintssss\frac{d\mu}{\langle \mu\rangle^{2\left(|a|+\theta\right)+\left(2b-1\right)}}
    \lesssim1. 
\end{align*}
Finally, let us turn to the region $D_2$. Using inequality \eqref{omega usual}, Lemma \ref{dif} and the fact that in $D_2$, $\xi_1^2\lesssim |R| = |-\sigma+\sigma_1+\sigma_2|\leq 3|\sigma_2|$, we get 
\begin{align*}
     \iint_{\mathbb{R}^2}\left|\chi_{\tilde{D_2}}\Omega\right|^2\,d\tau_1d\xi_1
     &\lesssim  \frac{1}{\langle\sigma_2\rangle^{2b}}\bigintssss\frac{\langle\xi_1\rangle^{2\theta+1}}{\langle R + \sigma_2\rangle^{2\left(|a|+\theta\right)-1}}d\xi_1\\ &\lesssim \bigintssss_{\mathbb{R}} \frac{d\xi_1}{\langle\xi_1\rangle^{4b-1-2\theta}\langle R + \sigma_2 \rangle^{2\left(|a|+\theta\right)-1}}\lesssim1,
\end{align*}
where the last inequality is due to Lemma \ref{mio}, taking $\ell_1=2b-\frac{1}{2}-\theta$ and $\ell_2=2\left(|a|+\theta\right)-1$ and noting that 
$$\ell_1+\ell_2 = (2b-1)+(\theta+2|a|)-\frac{1}{2}>0+1-\frac{1}{2}=\frac{1}{2}.$$
It shows the estimate \eqref{E}. 
\end{proof}
	\section{\label{Well-posedness}Well-posedness}
    \subsection{A priori estimates for smooth solutions}
    Let us first derive a priori estimates for smooth solutions to the system \eqref{Schr-ILW-T}. Let $(u,v)$ a smooth solution of \eqref{Schr-ILW-T} with initial data $(u(0),v(0)):=(u_0,v_0)$ and $w:=w(u,v)$ the gauge transformation defined in \eqref{gauge}. Let $T>0$ such that $(u,v)$ is defined. Moreover, let us take $T$ small enough. We define
    \begin{multline}
        \label{N_T} N_T^s\equiv N_T^s(u,v):= \max\Big(\left\|u\right\|_{X^{s+\frac{1}{2},b}_\text{S}(T)}, \left\|v\right\|_{L^\infty_TH^s_x}, \left\|v\right\|_{L^4_TW^{s,4}_x}, \\\left\|w(0)\right\|_{H^s}, \left\|\chi_{[0,T]}(t)\mathscr{N}_\delta(u,v,w)\right\|_{X^{s+\frac{1}{2},b}_\text{S}(T)}\Big),
    \end{multline}
    where $\mathscr{N}_\delta(u,v,w)$ is defined in \eqref{Ndelta}. We first note that the function $T\mapsto N_T^s$ is continuous and non-decreasing. On the other hand, considering the Duhamel formulation of the first equation of \eqref{Schr-ILW-T}, applying the classic linear estimates in Bourgain spaces, and then, applying estimates \eqref{estimativa bilineal} and \eqref{cubo}, we get 
    \begin{multline}
    \label{e1'} \left\|u\right\|_{X^{s+\frac{1}{2},b}_\text{S}(T)}\lesssim \left\|u(0)\right\|_{H^{s+\frac{1}{2}}} + T^{\kappa}\left\|u\right\|_{X^{\frac{1}{2},b}_\text{S}(T)}\left\|v\right\|_{X^{s-\theta,\theta}_\text{BO}(T)}  \\ + T^{\kappa}\left\|u\right\|_{X^{s+\frac{1}{2},b}_\text{S}(T)}\left\|v\right\|_{X^{-\theta,\theta}_\text{BO}(T)} + T^\kappa\left\|u\right\|_{X^{\frac{1}{2},b}_\text{S}(T)}^2\left\|u\right\|_{X^{s+\frac{1}{2},b}_\text{S}(T)}.
\end{multline}
Taking into account that 
\begin{equation}
    \label{dobleu} \|w(0)\|_{H^s}\lesssim \left(1+\|v_0\|_{L^2_x}\right)\|v_0\|_{H^s_x}, 
\end{equation}
and inequalities \eqref{e1'}, \eqref{apriori1} and \eqref{normacuadradou}, taking $T$ small enough, we get 
\begin{equation}
    \label{limite} \lim_{T\searrow0} N_T^s \lesssim \left(1+\|u_0\|_{H^{\frac{1}{2}}}+\|v_0\|_{L^2}\right)\left(\|u_0\|_{H^{s+\frac{1}{2}}}+\|v_0\|_{H^s}\right). 
\end{equation}
Moreover, from inequality \eqref{e1'}, 
\begin{equation}
    \label{Nu} \|u\|_{X^{s+\frac{1}{2},b}_{\text{S}}(T)}\lesssim \|u_0\|_{H^{s+\frac{1}{2}}}+T^\kappa \left(N_T^0\|v\|_{X^{s-\theta,\theta}_{\text{BO}}(T)}+N_T^s\|v\|_{X^{-\theta,\theta}_{\text{BO}}(T)}\right)+T^\kappa\left(N_T^0\right)^2N_T^s.
\end{equation}
By estimate \eqref{apriori2}, 
\begin{equation}
      \label{ast1}  \|v\|_{X^{s-\theta,\theta}_\text{BO}(T)} \lesssim \left(1+N_T^0\right)N_T^s.
\end{equation}
Applying this to the estimate \eqref{Nu}, 
\begin{equation}
    \label{Ast1} 
    \begin{split}
          &\|u\|_{X^{s+\frac{1}{2},b}_{\text{S}}(T)}\\\lesssim  &\|u_0\|_{H^{s+\frac{1}{2}}}+T^\kappa  \left(N_T^0\left(1+N^0_T\right)N^s_T+N_T^s\left(1+N_T^0\right)N_T^0\right)+\left(N_T^0\right)^2N_T^s\\
    \lesssim &\|u_0\|_{H^{s+\frac{1}{2}}}+ N_T^0\left(1+N_T^0\right)N_T^s.   
    \end{split}
\end{equation}

On the other hand, by estimate \eqref{apriori2}, we get 
\begin{align*}
    \|J^s_xv\|_{L^p_TL^q_x}&\lesssim \|v_0\|_{H^s} + T\delta^{-1} N_T^s + \left(1+N_T^0\right)\left(\|w\|_{Y^{s,\frac{1}{2}}(T)}+T^{1+\frac{1}{p}}N_T^0N_T^s\right)+TN_T^0N_T^s\\ 
    &\lesssim \|v_0\|_{H^s} + T\delta^{-1}N_T^s+N_T^0\left(1+N_T^0\right)N_T^s+N_T^0\|w\|_{Y^{s,\frac{1}{2}}(T)}+\|w\|_{Y^{s,\frac{1}{2}}(T)},
\end{align*}
where $(p,q)\in\{(\infty,2),(4,4)\}$. 
Now, taking into account that 
\begin{equation}
    \|w\|_{Y^{s,\frac{1}{2}}(T)}\lesssim \|w(0)\|_{H^s}+\left\|\chi_{[0,T]}\mathscr{N}_\delta (u,v,w)\right\|_{Y^{s,-\frac{1}{2}}}\lesssim N_T^s,
\end{equation}
we get 
\begin{equation}
   \label{ast2} \|J^s_x v\|_{L^q_TL^p_x} \lesssim \|v_0\|_{H^s}+T\delta^{-1}N_T^s+N_T^0\left(1+N_T^0\right)N_T^s+\|w\|_{Y^{s,\frac{1}{2}}(T)}.
\end{equation}
Now, by the bilinear estimates \eqref{normacuadradou} and \eqref{bilcha}, with \eqref{ast1},
\begin{equation}
    \label{Ast2}\left\|\chi_{[0,T]}\mathscr{N}_\delta(u,v,w)\right\|_{Y^{s,-\frac{1}{2}}}\lesssim N_T^0\left(1+N_T^0\right)N_T^s,
\end{equation}
and by \eqref{ast2}, we get 
\begin{equation}
    \label{ast3}
    \begin{split}
        \|J^s_xv\|_{L^p_TL^q_x}&\lesssim \|v_0\|_{H^s} + \left(1+\|v_0\|_{L^2}\right)\|v_0\|_{H^s} + T\delta^{-1} N_T^s + N_T^0\left(1+N_T^0\right)N_T^s \\
    &\lesssim \left(1+\|v_0\|_{L^2}\right)\|v_0\|_{H^s} + T\delta^{-1} N_T^s + N_T^0\left(1+N_T^0\right)N_T^s. 
    \end{split}
\end{equation}
Hence, by the definition of $N_T^s$, with estimates \eqref{ast3}, \eqref{Ast1}, \eqref{dobleu} and \eqref{Ast2}, there are constants $C_1, C_2, C_3>0$, depending only on $s$, such that 
\begin{multline*}
    N_T^s \leq C_1\left(1+\|u_0\|_{H^\frac{1}{2}}+\|v_0\|_{L^2}\right)\left(\|u_0\|_{H^{s+\frac{1}{2}}}+\|v_0\|_{H^s}\right)\\ + C_2T\delta^{-1}N_T^s+C_3N_T^0\left(1+N_T^0\right)N_T^s.
\end{multline*}
Fix any $T_0<\frac{\delta}{2C_2}$. For all $0<T<\text{min}(T^*,T_0)$, we get  
\begin{equation*}
    N_T^s\leq 2C_1\left(1+\|u_0\|_{H^\frac{1}{2}}+\|v_0\|_{L^2}\right)\left(\|u_0\|_{H^{s+\frac{1}{2}}}+\|v_0\|_{H^s}\right)+2C_3N_T^0\left(1+N^0_T\right)N_T^s.
\end{equation*}
Fix any $\varepsilon_0>0$ sufficiently small. Using the limit \eqref{limite} and a continuity argument, we obtain 
\begin{equation*}
    N_T^0 \lesssim \varepsilon
\end{equation*}
providing $\|u_0\|_{H^\frac{1}{2}}+\|v_0\|_{L^2}:=\varepsilon$, with $\varepsilon \ll\varepsilon_0$. Therefore, we obtain the following a priori estimate: 
\begin{equation}
    \label{aprioriNT} N_T^s \lesssim \|u_0\|_{H^{s+\frac{1}{2}}}+\|v_0\|_{H^s}. 
\end{equation}
 \subsection{Difference estimates} 
 Let $(u_1,v_1)$ and $(u_2,v_2)$ smooth solutions of the system \eqref{Schr-ILW-T} such that for $j\in\{1,2\}$, $\|u_j(0)\|_{H^{\frac{1}{2}}}+\|v_j(0)\|_{L^2}\leq \varepsilon \leq \varepsilon_0$, and $F_j$ the spatial primitive of $v_j$ is defined as in \eqref{primitive}. Let us suppose that $P_{\text{LO}}v_1(0)=P_{\text{LO}}v_2(0)$.  Moreover, define the associated gauge transformations by 
 \begin{equation*}
     W_j:= P_{+\text{hi}}\left(e^{-\frac{\nu}{2}iF_j}\right)\, \, \text{and} \, \, w_j:= \partial_x W_j,
 \end{equation*} 
Then, from the discussion of the previous subsection, we have 
 \begin{equation}
     \max_{j\in\{1,2\}}\|w_j\|_{Y^{0,\frac{1}{2}}(T)}\lesssim N_T^0(u_j,v_j,w_j) \lesssim \varepsilon \leq \varepsilon_0,
 \end{equation}
 and, 
  \begin{equation}
     \max_{j\in\{1,2\}}\|w_j\|_{Y^{s,\frac{1}{2}}(T)}\lesssim N_T^s(u_j,v_j) \lesssim \max_{j\in\{1,2\}}\left(\|u_j(0)\|_{H^{s+\frac{1}{2}}}+\|v_j(0)\|_{H^s}\right):=M.
 \end{equation}
 Let $\tilde{u}:= u_1-u_2$, $\tilde{v}:= v_1-v_2$, $\tilde{F}:= F_1-F_2$, $\tilde{W}:= W_1-W_2$ and $\tilde{w}:= w_1-w_2$. Hence,
 \begin{equation}
     \label{eclarga}\begin{split}
        &\partial_t\tilde{w}-\nu\mathscr{H}\partial_x^2\tilde{w}\\
     = &\,\,\nu^2\partial_xP_{+\text{hi}}\left(W_1P_-\partial_x\tilde{v}\right)+\nu^2\partial_xP_{+\text{hi}}\big(\tilde{W}P_-\partial_xv_2\big)\\
     &+\nu^2\partial_xP_{+\text{hi}}\left(P_{\text{lo}}e^{-\frac{\nu}{2}iF_1}P_-\partial_x\tilde{v}\right)+\nu^2\partial_xP_{+\text{hi}}\left(P_{\text{lo}}\left(e^{-\frac{\nu}{2}iF_1}-e^{-\frac{\nu}{2}iF_2}\right)P_-\partial_xv_2\right)\\
     &-\frac{\nu^2}{2}i\partial_xP_{+\text{hi}}\left(e^{-\frac{\nu}{2}iF_1} \mathscr{Q}_\delta \tilde{v}\right)-\frac{\nu^2}{2}i\partial_xP_{+\text{hi}}\left(\left(e^{-\frac{\nu}{2}iF_1} - e^{-\frac{\nu}{2}iF_2}\right)\mathscr{Q}_\delta v_2\right)\\
     &-\frac{\nu^2}{2}i\partial_xP_{+\text{hi}}\left(e^{-\frac{\nu}{2}iF_1} \left(|u_1|^2-|u_2|^2\right)\right)-\frac{\nu^2}{2}i\partial_xP_{+\text{hi}}\left(\left(e^{-\frac{\nu}{2}iF_1} - e^{-\frac{\nu}{2}iF_2}\right)|u_2|^2\right)\\:=& \tilde{\mathscr{N}}_\delta(\tilde{u}, \tilde{v}, \tilde{w}). 
     \end{split}
 \end{equation} Using the Duhamel formulation of \eqref{eclarga}, and estimates \eqref{2-3} and \eqref{2-7}, we get 
 \begin{equation}
     \label{tildew} \left\|\tilde{w}\right\|_{Y^{s,\frac{1}{2}}(T)}:= \left\|\tilde{w}(0)\right\|_{H^s}+\big\|\chi_{[0,T]}(t)\tilde{\mathscr{N}}_\delta(\tilde{u}, \tilde{v}, \tilde{w})\big\|_{Y^{s,-\frac{1}{2}}}. 
 \end{equation}
 First, we bound the first term on the right hand side of \eqref{tildew}. Indeed, 
 \begin{equation}
     \label{tildew2}
     \begin{split}
          \left\|\tilde{w}(0)\right\|_{H^s}
      = &\frac{|\nu|}{2}\left\|P_{+\text{hi}}\left(e^{-\frac{\nu}{2}iF_1(0)}\tilde{v}(0)+\left(e^{-\frac{\nu}{2}iF_1(0)}-e^{-\frac{\nu}{2}iF_2(0)}\right)v_2(0)\right)\right\|_{H^s}\\
      \lesssim &\left(1+\left\|v_1(0)\right\|_{L^2}\right)\left\|\tilde{v}(0)\right\|_{H^s}\\&\textcolor{white}{++}+\left(\left\|\tilde{v}(0)\right\|_{L^2}+\left\|e^{-\frac{\nu}{2}iF_1(0)}-e^{-\frac{\nu}{2}iF_2(0)}\right\|_{L^\infty}\left(1+\|v_1(0)\|_{L^2}\right)\right)\|v_2(0)\|_{H^s}\\
      \lesssim &\left\|\tilde{v}(0)\right\|_{L^2}+\big(\|\tilde{v}(0)\|_{L^2}+\|\tilde{F}(0)\|_{L^2}\big)\varepsilon, 
     \end{split}
 \end{equation}where the last inequality is due to the mean value property. Now, using Bernstein's inequality, we get
\begin{equation}
\begin{split}
    \big\|\tilde{F}(0)\big\|_{L^\infty}&\leq\big\|P_{\text{hi}}\tilde{F}(0)\big\|_{L^\infty}+\big\|P_{\text{lo}}\tilde{F}(0)\big\|_{L^\infty}\\&\lesssim \big\|\partial_xP_{\text{hi}}\tilde{F}(0)\big\|_{L^2}+\big\|P_{\text{lo}}\tilde{F}(0)\big\|_{L^\infty}= \big\|\tilde{v}(0)\big\|_{L^2}+\big\|P_{\text{lo}}\tilde{F}(0)\big\|_{L^\infty}.
\end{split}
          \label{tildeF1}
\end{equation} On the other hand, by construction we observe that since $F_j(x)-F_j(y)=\int_x^yv_j(z)\,dz$, 
 \begin{equation*}
     P_{\text{LO}}\int_x^y v_j\,dz = P_{\text{LO}}F_j(x)-F_j(y) \, \, \text{and} \, \, \int_x^y v_j\,dz = P_{\text{LO}}F_j(x)-P_{\text{LO}}F(y), 
 \end{equation*}
 which implies that 
 \begin{equation}
     P_{\text{lo}}\int_y^xP_{\text{LO}}v_j\,dz= P_{\text{LO}}\int_x^yv_j\,dz. 
 \end{equation}
 For that reason, 
 \begin{align*}
     P_{\text{lo}}\tilde{F}(0)=\int_\mathbb{R}\psi(y)P_{\text{lo}}\int_y^x\tilde{v}(0,z)\,dz\,dy = \int_\mathbb{R}\psi(y)P_{\text{lo}}\int_y^xP_{\text{LO}}\tilde{v}(0,z)\,dz\,dy=0,
 \end{align*}
 where the last equality is because we are assuming that $P_{\text{LO}}v_1(0)=P_{\text{LO}}v_2(0)$. By \eqref{tildeF1},
 \begin{equation}
    \label{tildeF2} \big\|\tilde{F}(0)\big\|_{L^\infty}\lesssim \left\|\tilde{v}(0)\right\|_{L^2},
 \end{equation}
 and applying this in \eqref{tildew2}, we get 
 \begin{equation}
     \label{tildew3}\left\|\tilde{w}(0)\right\|_{H^s}\lesssim \left\|\tilde{v}(0)\right\|_{L^2}+\left\|\tilde{v}(0)\right\|_{H^s}.
 \end{equation} 
 On the other hand, note that 
 \begin{align*}
     &\left\|\chi_{[0,T]}(t)\tilde{\mathscr{N}}_\delta(\tilde{u}, \tilde{v}, \tilde{w})\right\|_{Y^{s,-\frac{1}{2}}} \\ \lesssim  &\left\|\chi_{[0,T]}(t)\partial_xP_{+\text{hi}}\left(W_1P_-\partial_x\tilde{v}\right)\right\|_{Y^{s,-\frac{1}{2}}} + \left\|\chi_{[0,T]}(t)\partial_xP_{+\text{hi}}\left(\tilde{W}P_-\partial_xv_2\right)\right\|_{Y^{s,-\frac{1}{2}}}\\
    &\textcolor{white}{++++}+\left\|\chi_{[0,T]}(t)\partial_xP_{+\text{hi}}\left(P_{\text{lo}}e^{-\frac{\nu}{2}iF_1}P_-\partial_x\tilde{v}\right)\right\|_{Y^{s,-\frac{1}{2}}}\\&\textcolor{white}{++++}+\left\|\chi_{[0,T]}(t)\partial_xP_{+\text{hi}}\left(P_{\text{lo}}\left(e^{-\frac{\nu}{2}iF_1}-e^{-\frac{\nu}{2}iF_2}\right)P_-\partial_xv_2\right)\right\|_{Y^{s,-\frac{1}{2}}}\\
     &\textcolor{white}{++++}+\left\|\chi_{[0,T]}(t)\partial_xP_{+\text{hi}}\left(e^{-\frac{\nu}{2}iF_1} \mathscr{Q}_\delta \tilde{v}\right)\right\|_{Y^{s,-\frac{1}{2}}}\\&\textcolor{white}{++++}+\left\|\chi_{[0,T]}(t)\partial_xP_{+\text{hi}}\left(\left(e^{-\frac{\nu}{2}iF_1} - e^{-\frac{\nu}{2}iF_2}\right)\mathscr{Q}_\delta v_2\right)\right\|_{Y^{s,-\frac{1}{2}}}\\
     &\textcolor{white}{++++}+\left\|\chi_{[0,T]}(t)\partial_xP_{+\text{hi}}\left(e^{-\frac{\nu}{2}iF_1} \left(|u_1|^2-|u_2|^2\right)\right)\right\|_{Y^{s,-\frac{1}{2}}}\\&\textcolor{white}{++++}+\left\|\chi_{[0,T]}(t)\partial_xP_{+\text{hi}}\left(\left(e^{-\frac{\nu}{2}iF_1} - e^{-\frac{\nu}{2}iF_2}\right)|u_2|^2\right)\right\|_{Y^{s,-\frac{1}{2}}} \\:= &I + \cdots + VIII. 
 \end{align*}
By Proposition 3.5 of \cite{molinet-pilod},  
\begin{equation}
    \label{I0}I \lesssim \left\|w_1\right\|_{X^{s,\frac{1}{2}}_\text{BO}(T)}\left(T^\frac{1}{2}\left\|\tilde{v}\right\|_{L^\infty_TL^2_x}+\left\|\tilde{v}\right\|_{L^4_{T,x}}+\left\|\tilde{v}\right\|_{X^{-1,1}_\text{BO}(T)}\right).
\end{equation}
We now bound the last term inside the parentheses in \eqref{I0}. Note that $\tilde{v}$ satisfies 
\begin{equation*}
    \partial_t\tilde{v}-\nu\mathscr{H}\partial_x^2\tilde{v}+\frac{\nu^2}{2}\partial_x\left((v_1+v_2)\tilde{v}\right)= \nu \mathscr{Q}_\delta\partial_x\tilde{v}+\nu\partial_x\left(\tilde{u}\overline{u_1}+u_2\overline{\tilde{u}}\right). 
\end{equation*}
Arguing as in the proof of the estimate \eqref{apriori1}, we obtain 
\begin{multline}
    \label{tildev}\left\|\tilde{v}\right\|_{X^{s-\theta,\theta}_\text{BO}(T)}\lesssim \left\|\tilde{v}\right\|_{L^\infty_TL^2_x}+\left\|\tilde{v}\right\|_{L^\infty_TH^s_x} + \left\|\tilde{v}\right\|_{L^4_{T,x}}+\varepsilon\left\|\tilde{v}\right\|_{L^4_TW^{s,4}_x}\\+\left\|\tilde{u}\right\|_{X^{\frac{1}{2},b}_\text{S}(T)}+ \varepsilon\left\|\tilde{u}\right\|_{X^{s+\frac{1}{2},b}_\text{S}(T)}. 
\end{multline}
We now consider the first factor on the right hand side in \eqref{I0}. By \eqref{bilcha}, 
\begin{equation}
       \label{tildew_1}\left\|w_1\right\|_{X^{s,\frac{1}{2}}_\text{BO}(T)}
  \lesssim  (1+\varepsilon)\|v_1(0)\|_{H^s}+\varepsilon(1+\varepsilon)\|w_1\|_{X^{s,\frac{1}{2}}_\text{BO} (T)}+ \varepsilon^2.
\end{equation}
Applying inequalities \eqref{tildev} and \eqref{tildev} in \eqref{I0}, when $s=0$, we obtain 
\begin{equation}
    \label{I_0} I \lesssim \varepsilon\big(\left\|\tilde{v}\right\|_{L^\infty_TL^2_x}+\left\|\tilde{v}\right\|_{L^4_{T,x}}+\left\|\tilde{u}\right\|_{X^{{1/2},b}_\text{BO}(T)}\big),
\end{equation}
and, in the general case $s\geq0$, 
\begin{equation}
    \label{I} I \lesssim M\Big(\|\tilde{v}\|_{L^\infty_TL^2_x}+\|\tilde{v}\|_{L^4_{T,x}}+\|\tilde{u}\|_{X^{{1/2},b}_\text{BO}(T)}\Big). 
\end{equation}
By Proposition 3.5 in \cite{molinet-pilod}, 
\begin{equation}
    \label{II} II \lesssim \left\|\tilde{w}\right\|_{X^{s,{1/2}}_\text{BO}(T)}\big(T^{1/2}\|v_2\|_{L^\infty L^2_x}+\|v_2\|_{L^4_{T,x}}+\|v_2\|_{X^{-1,1}_\text{BO}(T)}\big)\lesssim \varepsilon\|\tilde{w}\|_{Y^{s,{1/2}}(T)}. 
\end{equation}
We now consider the term III. Indeed, there exists $0<\varepsilon'<\frac{1}{2}$ such that 
\begin{align*}
    III \lesssim \left\|\chi_{[0,T]}\,\partial_xP_{+\text{hi}}\left(P_{\text{lo}}e^{-\frac{\nu}{2}iF_1}P_-\partial_x\tilde{v}\right)\right\|_{X^{s,-\frac{1}{2}+\varepsilon'}_\text{BO}} \leq\left\|\partial_xP_{+\text{hi}}\left(P_{\text{lo}}e^{-\frac{\nu}{2}iF_1}P_-\partial_x\tilde{v}\right)\right\|_{L^2_T H^s_x}. 
\end{align*}
Taking into account that $\text{supp}\left[\partial_xP_{+\text{hi}}\left(P_{\text{lo}}e^{-\frac{\nu}{2}iF_1}P_-\partial_x\tilde{v}\right)\right]^{\wedge_x}\subseteq [1,2]$, we apply Bernstein's inequality, the fractional Leibniz rule \eqref{FLR'}, and Cauchy-Schwarz to obtain 
\begin{equation}
    III \lesssim \left(\int_0^T \left\|\partial_x P_{\text{lo}}e^{-\frac{\nu}{2}iF_1}\right\|_{L^4_x}^2\left\|\tilde{v}\right\|_{L^4_x}^2\,dt\right)^\frac{1}{2}\lesssim \|v_1\|_{L^4_{T,x}}\left\|\tilde{v}\right\|_{L^4_{T,x}} \lesssim \varepsilon \|\tilde{v}\|_{L^4_{T,x}}. \label{III}
\end{equation} Moreover, note that 
\begin{equation}
    \label{4f}
    \begin{split}
            IV &\lesssim \left(\int_0^T\left\|\partial_xP_{\text{lo}}\left(e^{-\frac{\nu}{2}iF_1}-e^{-\frac{\nu}{2}iF_2}\right)\right\|_{L^4_x}^2\|v_2\|_{L^4_x}^2\,dt\right)^\frac{1}{2}\\
    & = \frac{|\nu|}{2}\left(\int_0^T\left\|P_{\text{lo}}\left(e^{-\frac{\nu}{2}iF_1}\tilde{v}+\left(e^{-\frac{\nu}{2}iF_1}-e^{-\frac{\nu}{2}iF_2}\right)v_2\right)\right\|_{L^4_x}^2\|v_2\|_{L^4_x}^2\,dt\right)^\frac{1}{2}\\
    &\lesssim \left\|\tilde{v}\right\|_{L^4_{T,x}}\left\|v_2\right\|_{L^4_{T,x}} + \left\|e^{-\frac{\nu}{2}iF_1}-e^{-\frac{\nu}{2}iF_2}\right\|_{L^\infty_{T,x}}\left\|v_2\right\|_{L^4_{T,x}}\\
    &\lesssim \varepsilon \left\|\tilde{v}\right\|_{L^4_{T,x}} + \varepsilon^2 \big\|\tilde{F}\big\|_{L^\infty_{T,s}}.
    \end{split}
\end{equation}
Consider the second term in \eqref{4f}. First, note that $\tilde{F}$ satisfies 
\begin{equation}
    \label{tildesuperF} \partial_t\tilde{F}-\nu\mathscr{H}\partial_x^2\tilde{F}+\frac{\nu^2}{2}\left(v_1+v_2\right)\tilde{v}= \nu\mathscr{Q}_\delta\tilde{v}+\nu\left(\tilde{u}\overline{u}_1+u_2\overline{\tilde{u}}\right),
\end{equation}
whose integral associated equation is given by 
\begin{multline*}
    \tilde{F}(t)=V(t)\tilde{F}(0)-\frac{\nu^2}{2}\int_0^tV(t-t')(\left(v_1+v_2\right)\tilde{v})(t')dt'\\+\nu\int_0^tV(t-t')\mathscr{Q}_\delta\tilde{v}(t')dt' + \nu \int_0^t V(t-t')\left(\tilde{u}\overline{u}_1+u_2\overline{\tilde{u}}\right)(t')dt'.
\end{multline*}
Applying $P_\text{lo}$ in the last equation, and using Bernstein's and Young's inequalities, we get 
\begin{align*}
    \big\|P_\text{lo}\tilde{F}\big\|_{L^\infty_{T,x}}&\lesssim \int_0^T \left[\big\|\left(v_1+v_2\right)\tilde{v}\right\|_{L^1_x}+ \left\|\mathscr{Q}_\delta P_\text{lo}\tilde{v}\right\|_{L^2_x}+ \left\|P_\text{lo}\left(\tilde{u}\overline{u}_1+u_2\overline{\tilde{u}}\right)\right\|_{L^1_x}\big]dt'\\&:= IV_1+IV_2+IV_3.
\end{align*}
Note that
\begin{align*}
    &IV_1 \leq \int_0^T \left\|v_1+v_2\right\|_{L^2_x}\left\|\tilde{v}\right\|_{L^2_x}\,dt'\leq T\left\|v_1+v_2\right\|_{L^\infty_TL^2_x}\left\|\tilde{v}\right\|_{L^\infty_TL^2_x}\lesssim \varepsilon \left\|\tilde{v}\right\|_{L^\infty_TL^2_x}, \\
    &IV_2 \leq T^\frac{1}{2}\delta^{-1}\left\|\tilde{v}\right\|_{L^\infty_TL^2_x},
\end{align*}
and, 
\begin{align*}
    IV_3
    &\leq \int_0^T \left\|\tilde{u}\right\|_{L^2_x}\left\|u_1\right\|_{L^2_x}\,dt + \int_0^T \left\|u_2\right\|_{L^2_x}\left\|\tilde{u}\right\|_{L^2_x}\,dt\\
    &\leq T \left(\left\|\tilde{u}\right\|_{L^\infty_TL^2_x}\left\|u_1\right\|_{L^\infty_TL^2_x}+ \left\|\tilde{u}\right\|_{L^\infty_TL^2_x}\left\|u_2\right\|_{L^\infty_TL^2_x}\right) \lesssim \varepsilon \left\|\tilde{u}\right\|_{L^\infty_TL^2_x}\lesssim \varepsilon \left\|\tilde{u}\right\|_{X^{0,b}_S(T)}.
\end{align*}
Therefore, 
\begin{equation}
    \label{loF} \big\|P_\text{lo}\tilde{F}\big\|_{L^\infty_{T,x}}\lesssim \varepsilon \left\|\tilde{u}\right\|_{L^\infty_TL^2_x}. 
\end{equation}
On the other hand, by Bernstein inequality, 
\begin{equation}
    \label{hiF} \big\|P_\text{hi}\tilde{F}\big\|_{L^\infty_{T,x}}\lesssim  \big\|\partial_xP_\text{hi}\tilde{F}\big\|_{L^\infty_{T}L^2_x} \leq \left\|\tilde{v}\right\|_{L^\infty_T L^2_x}. 
\end{equation}
Hence, by estimates \eqref{loF} and \eqref{hiF}, we get 
\begin{equation}
   \label{Fimportante} \big\|\tilde{F}\big\|_{L^{\infty}_{T,x}} \lesssim \left\|\tilde{v}\right\|_{L^\infty_TL^2_x}+\varepsilon\left\|\tilde{u}\right\|_{X^{0,b}_\text{S}(T)}.
\end{equation}
Applying this to \eqref{4f}, we obtain 
\begin{equation}
    \label{IV}  IV \lesssim \varepsilon\Big(\left\|\tilde{v}\right\|_{L^\infty_TL^2_x}+\left\|\tilde{u}\right\|_{X^{\frac{1}{2},b}_\text{S}(T)}\Big).
\end{equation}
Now, for $0\leq s \leq \frac{1}{2}$, 
\begin{equation*}
     V\lesssim \left\|\partial_xP_{+\text{hi}}\left(e^{-\frac{\nu}{2}iF_1} \mathscr{Q}_\delta\tilde{v}\right)\right\|_{L^2_T H^s_x} \lesssim \big(1+\left\|v_1\right\|_{L^\infty_TL^2_x}\big)\big(\left\|v_1\mathscr{Q}_\delta\tilde{v}\right\|_{L^2_TH^s_x} + \left\|\mathscr{Q}_\delta\partial_x\tilde{v}\right\|_{L^2_TH^s_x}\big).
\end{equation*}
By the fractional Leibniz rule, Sobolev embedding, and Lemma \eqref{qdelta},
\begin{align*} 
    \left\|v_1\mathscr{Q}_\delta\tilde{v}\right\|_{L^2_TH^s_x} \lesssim T^\frac{1}{2}\delta^{-1}\left(1+\delta^{-1}\right)\left(\|v_1\|_{L^\infty_TH^s_x}+\|v_1\|_{L^\infty_TL^2_x}\right) \left\|\tilde{v}\right\|_{L^\infty_TL^2_x}.
\end{align*}
In the case $s=0$ we get 
\begin{equation*}
    \left\|v_1\mathscr{Q}_\delta\tilde{v}\right\|_{L^2_TL^2_x}\lesssim T^\frac{1}{2}\delta^{-1}\left(1+\delta^{-1}\right)\varepsilon \left\|\tilde{v}\right\|_{L^\infty_TL^2_x},
\end{equation*}
and when $0<s\leq\frac{1}{2}$, we have
\begin{equation*}
    \left\|v_1\mathscr{Q}_\delta\tilde{v}\right\|_{L^2_TH^s_x}\lesssim T^\frac{1}{2}\delta^{-1}\left(1+\delta^{-1}\right)\left\|\tilde{v}\right\|_{L^\infty_TL^2_x}.
\end{equation*}
On the other hand, 
\begin{equation*}
\left\|\mathscr{Q}_\delta\partial_x\tilde{v}\right\|_{L^2_TH^s_x}\leq\left\|\mathscr{Q}_\delta \tilde{v}\right\|_{L^2_TH^{s+1}_x}\lesssim T^\frac{1}{2}\delta^{-1}\left(1+\delta^{-s-1}\right)\left\|\tilde{v}\right\|_{L^\infty_T L^2_x}.
\end{equation*} Therefore, 
\begin{equation}
    \label{V} V\lesssim T^\frac{1}{2}\delta^{-1}\left(1+\delta^{-1}\right)\left\|\tilde{v}\right\|_{L^\infty_T L^2_x}. 
\end{equation} We now turn to the term $VI$. Indeed, 
\begin{align*}
    VI &\lesssim \left\|\partial_xP_{+\text{hi}}\left(\left(e^{-\frac{\nu}{2}iF_1}-e^{-\frac{\nu}{2}iF_2}\right)\mathscr{Q}_\delta v_2\right)\right\|_{L^2_T H^s_x}
    \\&\lesssim\left\|P_{+\text{hi}}\left(e^{-\frac{\nu}{2}iF_1}\tilde{v}\mathscr{Q}_\delta v_2\right)\right\|_{L^2_T H^s_x}+\left\|P_{+\text{hi}}\left(\left(e^{-\frac{\nu}{2}iF_1}-e^{-\frac{\nu}{2}iF_2}\right)v_2\mathscr{Q}_\delta v_2\right)\right\|_{L^2_TH^s_x}\\&\textcolor{white}{===========}+\left\|P_{+\text{hi}}\left(\left(e^{-\frac{\nu}{2}iF_1}-e^{-\frac{\nu}{2}iF_2}\right)\mathscr{Q}_\delta\partial_xv_2\right)\right\|_{L^2_TH^s_x}\\&:= VI_1+VI_2+VI_3. 
\end{align*}
By the fractional Leibniz rule, the Sobolev embedding and Lemma \eqref{qdelta},
\begin{align*}
    VI_1 &\lesssim \left(1+\|v_1\|_{L^\infty_TL^2_x}\right)\left\|\tilde{v}\mathscr{Q}_\delta v_2\right\|_{L^2H^s_x}\\
    &\lesssim T^\frac{1}{2}\delta^{-1}\left(1+\delta^{-1}\right)\left(1+\|v_1\|_{L^\infty_TL^2_x}\right)\left[\left\|\tilde{v}\right\|_{L^\infty_TH^s_x}\|v_2\|_{L^\infty_T L^2_x}+\left\|\tilde{v}\right\|_{L^\infty_TL^2_x}\left\|v_2\right\|_{L^\infty_T L^2_x}\right]\\
    &\lesssim T^\frac{1}{2}\delta^{-1}\left(1+\delta^{-1}\right)\left(\left\|\tilde{v}\right\|_{L^\infty_TH^s_x}+\left\|\tilde{v}\right\|_{L^\infty_TL^2_x}\right). 
\end{align*} On the other hand, using \eqref{fracexp},
\begin{align*}
    VI_2 &\lesssim \left(\left\|\tilde{v}\right\|_{L^\infty_TL^2_x}+\left\|e^{-\frac{\nu}{2}iF_1}-e^{-\frac{\nu}{2}iF_2}\right\|_{L^\infty_{T,x}}\left(1+\|v_1\|_{L^\infty L^2_x}\right)\right)\left\|v_2\mathscr{Q}_\delta v_2\right\|_{H^s_x}\\
    &\lesssim T^\frac{1}{2}\varepsilon\delta^{-1}\left(1+\delta^{-1}\right)\Big(\left\|\tilde{v}\right\|_{L^\infty_TH^s_x}+\left\|\tilde{v}\right\|_{L^\infty_TL^2_x}+\varepsilon\|\tilde{u}\|_{X^{\frac{1}{2},b}_\text{S}}\Big),
\end{align*}
and, 
\begin{align*}
    VI_3 \lesssim  \varepsilon T^\frac{1}{2} \delta^{-1}\left(1+\delta^{-1}\right) \Big(\left\|\tilde{v}\right\|_{L^\infty_TL^2_x}+\varepsilon\left\|\tilde{u}\right\|_{X^{\frac{1}{2},b}_\text{S}(T)}\Big).
\end{align*}
Hence, 
\begin{equation}
    \label{VI} VI \lesssim \varepsilon T^\frac{1}{2}\delta^{-1}\big(1+\delta^{-1}\big)\Big(\left\|\tilde{v}\right\|_{L^\infty_TL^2_x}+\left\|\tilde{v}\right\|_{L^\infty_TH^s_x}+\varepsilon\left\|\tilde{u}\right\|_{X^{\frac{1}{2},b}_\text{S}(T)}\Big).
\end{equation}
Now, 
\begin{align*}
    VII \lesssim \left\|\partial_xP_{+\text{hi}}\left(e^{-\frac{\nu}{2}iF_1}\tilde{u}u_1\right)\right\|_{L^2_TH^s_x}+\left\|\partial_xP_{+\text{hi}}\left(e^{-\frac{\nu}{2}iF_1}\overline{\tilde{u}}u_2\right)\right\|_{L^2_TH^s_x}:= VII_1+VII_2. 
\end{align*}
Observe that
\begin{align}
    VII_1 & \lesssim \left\|P_{+\text{hi}}\left(e^{-\frac{\nu}{2}iF_1}v_1\tilde{u}\overline{u}_1\right)\right\|_{L^2_TH^s_x}+\left\|P_{+\text{hi}}\left(e^{-\frac{\nu}{2}iF_1}\partial_x\left(\tilde{u}\overline{u}_1\right)\right)\right\|_{L^2_TH^s_x}\\
    &\lesssim \left(1+\|v_1\|_{L^\infty_T L^2_x}\right)\left(\left\|v_1\tilde{u}\overline{u}_1\right\|_{L^2_TH^s_x}+\left\|\partial_x\left(\tilde{u}\overline{u}_1\right)\right\|_{L^2_TH^s_x}\right). \label{VII_1}
\end{align}
Proceeding as in the proof of the estimate \eqref{normacuadradou}, we get in the case $s=0$,
\begin{align*}
    VII_1 \lesssim \varepsilon\left\|\tilde{u}\right\|_{X^{\frac{1}{2},b}_\text{S}(T)},
\end{align*}
and for when $0< s \leq \frac{1}{2}$, 
\begin{align*}
    VII_1 \lesssim \varepsilon \left\|\tilde{u}\right\|_{X^{s+\frac{1}{2},b}_\text{S}(T)} + \left\|\tilde{u}\right\|_{X^{\frac{1}{2},b}_\text{S}(T)}. 
\end{align*}
By symmetry of the argument we get the same bound for $VII_2$. Hence, when $s=0$, 
\begin{equation}
    \label{VII_0} VII\lesssim \varepsilon\left\|\tilde{u}\right\|_{X^{\frac{1}{2},b}_\text{S}(T)},
\end{equation}
and when $0< s \leq \frac{1}{2}$, 
\begin{equation}
    \label{VII} VII \lesssim \varepsilon \left\|\tilde{u}\right\|_{X^{s+\frac{1}{2},b}_\text{S}(T)} + \left\|\tilde{u}\right\|_{X^{\frac{1}{2},b}_\text{S}(T)}.
\end{equation} We turn to the term $VIII$. Note that  
\begin{align*}
    VIII&\lesssim \left\|P_{+\text{hi}}\left(\left(e^{-\frac{\nu}{2}iF_1}\tilde{v}+\left(e^{-\frac{\nu}{2}iF_1}-e^{-\frac{\nu}{2}iF_2}\right)v_2\right)|u_2|^2\right)\right\|_{L^2_TH^s_x}\\&\textcolor{white}{=============}+\left\|P_{+\text{hi}}\left(\left(e^{-\frac{\nu}{2}iF_1}-e^{-\frac{\nu}{2}iF_2}\right)\partial_x\left(|u_2|^2\right)\right)\right\|_{L^2_TH^s_x}\\
    &\leq  \left\|P_{+\text{hi}}\left(e^{-\frac{\nu}{2}iF_1}\tilde{v}|u_2|^2\right)\right\|_{L^2_TH^s_x}+\left\|P_{+\text{hi}}\left(\left(e^{-\frac{\nu}{2}iF_1}-e^{-\frac{\nu}{2}iF_2}\right)v_2|u_2|^2\right)\right\|_{L^2_TH^s_x}\\&\textcolor{white}{+++++++++}+\left\|P_{+\text{hi}}\left(\left(e^{-\frac{\nu}{2}iF_1}-e^{-\frac{\nu}{2}iF_2}\right)\partial_x\left(|u_2|^2\right)\right)\right\|_{L^2_TH^s_x}
\end{align*}
Proceeding again as in the proof of the estimate \eqref{normacuadradou}, but using the estimate \eqref{diffracexp}, for all $0<s\leq\frac{1}{2}$, we obtain
\begin{equation}
    \label{VIII} VIII \lesssim \varepsilon\Big(\left\|\tilde{v}\right\|_{L^\infty_T L^2_x}+\left\|\tilde{v}\right\|_{L^4_{T,x}}+\left\|\tilde{v}\right\|_{L^4_TW^{s,4}_x}+\left\|\tilde{u}\right\|_{X^{\frac{1}{2},b}_\text{S}(T)}\Big). 
\end{equation} Applying estimates \eqref{tildew3}, \eqref{I}, \eqref{II}, \eqref{III}, \eqref{IV}, \eqref{V}, \eqref{VI}, \eqref{VII} and \eqref{VIII} in \eqref{tildew}, we obtain 
\begin{multline}
   \label{estimativatildew} \|\tilde{w}\|_{Y^{s,\frac{1}{2}}(T)}\lesssim \left\|\tilde{v}(0)\right\|_{H^s_x}+\Big(\left||\tilde{v}\right\|_{L^\infty_TL^2_x}+\left\|\tilde{v}\right\|_{L^4_{T,x}}+\left\|\tilde{u}\right\|_{X^{\frac{1}{2},b}_\text{S}(T)}\Big)\\+\varepsilon\Big(\left||\tilde{v}\right\|_{L^\infty_TH^s_x}+\left\|\tilde{v}\right\|_{L^4_{T}W^{s,4}_x}+\left\|\tilde{u}\right\|_{X^{s+\frac{1}{2},b}_\text{S}(T)}\Big).
\end{multline}
For case $s=0$, instead of estimate \eqref{VII}, we take into account estimate \eqref{VII_0}, obtaining 
\begin{multline}
   \label{estimativatildew_0} \|\tilde{w}\|_{Y^{0,\frac{1}{2}}(T)}\lesssim \left\|\tilde{v}(0)\right\|_{H^s_x}\\+\varepsilon\Big(\left||\tilde{v}\right\|_{L^\infty_TL^2_x}+\left\|\tilde{v}\right\|_{L^4_{T,x}}+\left\|\tilde{u}\right\|_{X^{\frac{1}{2},b}_\text{S}(T)}\Big)+T^\frac{1}{2}\left\|\tilde{v}\right\|_{L^\infty_T L^2_x}. 
\end{multline} Now, let us move on to the bounding of $\left||\tilde{v}\right\|_{L^\infty_TH^s_x}, \left\|\tilde{v}\right\|_{L^4_{T}W^{s,4}_x}, \, \, \text{and}\, \, \, \left\|\tilde{u}\right\|_{X^{s+\frac{1}{2},b}_\text{S}(T)}$. First, we recall that for $(p,q)\in\left\{(\infty,2),(4,4)\right\}$ we have
\begin{equation}
   \label{Tildev} \left\|J^s_x\tilde{v}\right\|_{L^p_TL^q_x}\lesssim \left\|P_{\text{LO}}\tilde{v}\right\|_{L^p_TL^q_x}+\left\|D^s_x P_{+\text{HI}}\, \tilde{v}\right\|_{L^p_T L^q_x}.
\end{equation}
The following is the equation satisfied by $\tilde{v}:$
\begin{equation}
    \label{partialTildev}\partial_t\tilde{v}-\nu\mathscr{H}\partial_x^2\tilde{v}+\frac{\nu}{2}\partial_x\left(\left(v_1+v_2\right)\tilde{v}\right)= \nu\mathscr{Q}_\delta\partial_x\tilde{v}+\nu\partial_x\left(\tilde{u}\overline{u}_1+u_2\overline{\tilde{u}}\right). 
\end{equation}
Arguing as in the proof of estimate \eqref{PLOob}, we get 
\begin{equation}
\label{LOTi}\left\|P_{\text{LO}}\tilde{v}\right\|_{L^p_TL^q_x} \lesssim \varepsilon \Big(\left\|\tilde{v}\right\|_{L^\infty_TL^2_x}+\left\|\tilde{u}\right\|_{X^{\frac{1}{2},b}_\text{S}(T)}\Big) + T^{1+\frac{1}{p}}\delta^{-1}\left(1+\delta^{-1}\right)\left\|\tilde{v}\right\|_{L^\infty_TL^2_x}. 
\end{equation}
Note that 
\begin{align*}
    P_{+\text{HI}}\tilde{v}&=\frac{2}{\nu}iP_{+\text{HI}}\left(e^{\frac{\nu}{2}iF_1}\tilde{w}\right) + \frac{2}{\nu} iP_{+\text{HI}}\left(\left(e^{\frac{\nu}{2}iF_1}-e^{\frac{\nu}{2}iF_2}\right)w_2\right)\\&+\frac{2}{\nu}iP_{+\text{HI}}\left(P_{+\text{HI}}\left(e^{\frac{\nu}{2}iF_1}\right)\partial_xP_{-\text{hi}}\left(e^{-\frac{\nu}{2}iF_1}-e^{-\frac{\nu}{2}iF_2}\right)\right)\\&
    +\frac{2}{\nu}iP_{+\text{HI}}\left(P_{+\text{HI}}\left(e^{\frac{\nu}{2}iF_1}-e^{\frac{\nu}{2}iF_2}\right)\partial_xP_{-hi}\left(e^{-\frac{\nu}{2}iF_2}\right)\right)\\&+P_{+\text{HI}}\left(e^{\frac{\nu}{2}iF_1}P_{\text{lo}}\left(e^{-\frac{\nu}{2}iF_1}v_1-e^{-\frac{\nu}{2}iF_2}v_2\right)\right)\\&+ P_{+\text{HI}}\left(\left(e^{-\frac{\nu}{2}iF_1}-e^{-\frac{\nu}{2}iF_2}\right)P_\text{lo}\left(e^{-\frac{\nu}{2}iF_2}v_2\right)\right).
\end{align*} Proceeding as in the proof of estimate \eqref{DHI}, we get  
\begin{equation}
\label{DTildev}\left\|D^s_xP_{+\text{HI}}\,\tilde{v}\right\|_{L^p_TL^q_x}\lesssim \left\|\tilde{w}\right\|_{Y^{s,\frac{1}{2}}(T)}+\|w_2\|_{Y^{s,\frac{1}{2}}(T)}\left\|\tilde{v}\right\|_{L^\infty_TL^2_x}+\varepsilon\left\|\tilde{v}\right\|_{L^\infty_TL^2_x}.
\end{equation}
Let us bound the factor $\|w_2\|_{Y^{s,\frac{1}{2}}(T)}$. In fact, by \eqref{3gauge}, we get 
\begin{multline*}
    \|w_2\|_{Y^{s,\frac{1}{2}}(T)}\lesssim \left(1+\varepsilon\right)\|v_2(0)\|_{H^s}+\varepsilon(1+\varepsilon)\|w_2\|_{X^{s,\frac{1}{2}}_\text{BO}(T)}\\+\varepsilon^2 + \varepsilon\left(\|v_1\|_{L^4_TW^{s,4}_x}+(1+\varepsilon)\|u_1\|_{X^{s+\frac{1}{2}}_\text{S}(T)}\right)
\end{multline*}
which implies, for all $\varepsilon>0$ sufficiently small, 
\begin{equation}
    \label{222} \|w_2\|_{Y^{s,\frac{1}{2}}(T)}\lesssim \|v_2(0)\|_{H^s}+\varepsilon
\end{equation}
Applying \eqref{222} to \eqref{DTildev}, we get, in the case $s=0$
\begin{equation}
    \label{DTildev_1_0} \left\|P_{+\text{HI}}\,\tilde{v}\right\|_{L^p_TL^q_x}\lesssim \left\|\tilde{w}\right\|_{Y^{s,\frac{1}{2}}(T)}+\varepsilon\big(\left\|\tilde{v}\right\|_{L^\infty_TL^2_x}+\left\|\tilde{v}\right\|_{L^\infty_TL^2_x}\big),
\end{equation}
and for $0<s\leq\frac{1}{2}$,
\begin{equation}
    \label{DTildev_1} \left\|D^s_xP_{+\text{HI}}\,\tilde{v}\right\|_{L^p_TL^q_x}\lesssim \left\|\tilde{w}\right\|_{Y^{s,\frac{1}{2}}(T)}+M\left\|\tilde{v}\right\|_{L^\infty_TL^2_x}+\varepsilon\left\|\tilde{v}\right\|_{L^\infty_TL^2_x}.
\end{equation}
Applying estimates \eqref{DTildev_1} and \eqref{LOTi} in \eqref{Tildev}, we get 
\begin{multline}
    \label{estimateJs}\left\|J^s_x\tilde{v}\right\|_{L^p_TL^q_x}\lesssim \left\|\tilde{w}\right\|_{Y^{0,\frac{1}{2}}(T)}+\varepsilon\big(\left\|\tilde{v}\right\|_{L^\infty_TL^2_x}+\left\|\tilde{u}\right\|_{X^{{1/2},b}_\text{S}(T)}\big)\\+T^{1+\frac{1}{p}}\delta^{-1}\big(1+\delta^{-1}\big)\left\|\tilde{v}\right\|_{L^\infty_TL^2_x} + M\left\|\tilde{v}\right\|_{L^\infty_TL^2_x}. 
\end{multline}
Additionally, applying estimates \eqref{DTildev_1_0} and \eqref{LOTi} in \eqref{Tildev} with $s=0$, and then the estimate \eqref{estimativatildew_0}, we get
\begin{equation*}
\begin{split}
&\left\|\tilde{v}\right\|_{L^p_TL^q_x}\\\lesssim &\left\|\tilde{w}\right\|_{Y^{0,\frac{1}{2}}(T)}+\varepsilon\big(\left\|\tilde{v}\right\|_{L^\infty_TL^2_x}+\left\|\tilde{u}\right\|_{X^{{1/2},\,b}_\text{S}(T)}\big)+T^{1+\frac{1}{p}}\delta^{-1}\left(1+\delta^{-1}\right)\left\|\tilde{v}\right\|_{L^\infty_TL^2_x}\\
    \lesssim &\left\|\tilde{v}(0)\right\|_{L^2}+\varepsilon\big(\left\|\tilde{v}\right\|_{L^\infty_TL^2_x}+\left\|\tilde{v}\right\|_{L^4_{T,x}}+\left\|\tilde{u}\right\|_{X^{{1/2},\,b}_\text{S}(T)}\big)+T^{1+\frac{1}{p}}\delta^{-1}\left(1+\delta^{-1}\right)\left\|\tilde{v}\right\|_{L^\infty_TL^2_x}.
\end{split}
\end{equation*}
Then, taking $(p,q)=(\infty,2)$, for all $T>0$ sufficiently small, we get  
\begin{equation*}
    \left\|\tilde{v}\right\|_{L^\infty_TL^2_x}\lesssim \left\|\tilde{v}(0)\right\|_{L^2}+\varepsilon\big(\left\|\tilde{v}\right\|_{L^\infty_TL^2_x}+\left\|\tilde{v}\right\|_{L^4_{T,x}}+\left\|\tilde{u}\right\|_{X^{{1/2},\,b}_\text{S}(T)}\big)
\end{equation*}
which implies 
\begin{equation}
  \label{estimateJs1}  \left\|\tilde{v}\right\|_{L^\infty_TL^2_x}\lesssim \left\|\tilde{v}(0)\right\|_{L^2}+\varepsilon\big(\left\|\tilde{v}\right\|_{L^4_{T,x}}+\left\|\tilde{u}\right\|_{X^{{1/2},\,b}_\text{S}(T)}\big).
\end{equation}
Moreover, \eqref{estimateJs1} implies that
\begin{align*}
       \left\|\tilde{v}\right\|_{L^4_{T,x}}&\lesssim \left\|\tilde{v}(0)\right\|_{L^2}+\varepsilon\Big(\left\|\tilde{v}\right\|_{L^\infty_TL^2_x}+\left\|\tilde{u}\right\|_{X^{\frac{1}{2},b}_\text{S}(T)}\Big)+T^{1+\frac{1}{p}}\delta^{-1}\left(1+\delta^{-1}\right)\left\|\tilde{v}\right\|_{L^\infty_TL^2_x}\\
       &\lesssim  \left\|\tilde{v}(0)\right\|_{L^2}+\varepsilon\Big(\left\|\tilde{v}\right\|_{L^4_{T,x}}+\left\|\tilde{u}\right\|_{X^{\frac{1}{2},b}_\text{S}(T)}\Big),
\end{align*} 
which implies that for all $\varepsilon>0$ sufficiently small, 
\begin{equation}
    \label{estimateJs_2} \left\|\tilde{v}\right\|_{L^4_{T,x}} \lesssim \left\|\tilde{v}(0)\right\|_{L^2}+\varepsilon\left\|\tilde{u}\right\|_{X^{\frac{1}{2},b}_\text{S}(T)}.
\end{equation}
Applying \eqref{estimateJs_2} to \eqref{estimateJs1} we get 
\begin{equation}
        \label{estimateJs_1} \left\|\tilde{v}\right\|_{L^\infty_TL^2_x} \lesssim \left\|\tilde{v}(0)\right\|_{L^2}+\varepsilon\left\|\tilde{u}\right\|_{X^{\frac{1}{2},b}_\text{S}(T)}.
\end{equation} On the other hand, observe that $\tilde{u}$ satisfies the following equation:
\begin{equation}
  \label{tildeu}  i\left(\partial_t\tilde{u}-\nu\delta^{-1}\tilde{u}\right)+\partial_x^2\tilde{u}=\tilde{u}v_1+u_2\tilde{v}+\gamma\left(\overline{u}_1\left(u_1+u_2\right)\tilde{u}+u_2^2\overline{\tilde{u}}\right)
\end{equation} Applying the linear estimates in Bourgain spaces to the integral equation associated with \eqref{tildeu}, we get  
\begin{multline}
    \label{estimatetildeu} \left\|\tilde{u}\right\|_{X^{s+\frac{1}{2},b}_\text{S}(T)} \lesssim \left\|\tilde{u}(0)\right\|_{H^{s+\frac{1}{2}}}+\left\|\tilde{u}v_1\right\|_{X^{s+\frac{1}{2},a}_\text{S}(T)}+\left\|u_2\tilde{u}\right\|_{X^{s+\frac{1}{2},a}_\text{S}(T)}\\+\left\|\overline{u}_1\left(u_1+u_2\right)\tilde{u}\right\|_{X^{s+\frac{1}{2},a}_\text{S}(T)}+\left\|u_2^2\overline{\tilde{u}}\right\|_{X^{s+\frac{1}{2},a}_\text{S}(T)}
\end{multline}
Let us estimate each term on the right hand side in \eqref{estimatetildeu} (except the first one). By estimates \eqref{estimativa bilineal} and \eqref{apriori1}, 
\begin{align*}
    \left\|\tilde{u}v_1\right\|_{X^{s+\frac{1}{2},a}_\text{S}(T)} &\lesssim \left\|\tilde{u}\right\|_{X^{\frac{1}{2},b}_\text{S}(T)}\left\|v_1\right\|_{X^{s-\theta,\theta}_\text{BO}(T)}+\left\|\tilde{u}\right\|_{X^{s+\frac{1}{2},b}_\text{S}(T)}\left\|v_1\right\|_{X^{-\theta,\theta}_\text{BO}}\\
    &\lesssim \left\|\tilde{u}\right\|_{X^{\frac{1}{2},b}_\text{S}(T)}\Big(\|v_1\|_{L^\infty_TH^s_x}+\|v_1\|_{L^4_{T,x}}\|v_1\|_{L^4_{T}W^{s,4}_x}\\&+T^\frac{1}{2}\|u_1\|_{X^{\frac{1}{2},b}_\text{S}(T)}\left\|u_1\right\|_{X^{s+\frac{1}{2},b}_\text{S}(T)}\Big)+\left\|\tilde{u}\right\|_{X^{s+\frac{1}{2},b}_\text{S}(T)}\Big(\|v_1\|_{L^\infty_TL^2_x}+\|v_1\|_{L^4_{T,x}}\\&+T^\frac{1}{2}\|u_1\|_{X^{\frac{1}{2},b}_\text{S}(T)}^2\Big)\lesssim M\left\|\tilde{u}\right\|_{X^{\frac{1}{2},b}_\text{S}(T)}+\varepsilon\left\|\tilde{u}\right\|_{X^{s+\frac{1}{2},b}_\text{S}(T)}.
\end{align*}
In the particular case $s=0$, we get 
\begin{equation}
     \label{tildeu1_0}\left\|\tilde{u}v_1\right\|_{X^{\frac{1}{2},a}_\text{S}(T)}\lesssim \varepsilon \left\|\tilde{u}\right\|_{X^{\frac{1}{2},b}_\text{S}(T)}. 
\end{equation}
Using the estimates \eqref{estimativa bilineal} and \eqref{apriori1} again, we get 
\begin{align*}
    \left\|u_2\tilde{v}\right\|_{X^{s+\frac{1}{2},a}_\text{S}(T)}
    \lesssim  \left\|\tilde{v}\right\|_{L^\infty_TH^s_x}+\left\|\tilde{v}\right\|_{L^4_{T,x}}+\varepsilon\left\|\tilde{v}\right\|_{L^4_TW^{s,4}_x}+\left\|\tilde{u}\right\|_{X^{\frac{1}{2},b}_\text{S}(T)}+\varepsilon\left\|\tilde{u}\right\|_{X^{s+\frac{1}{2},b}_\text{S}(T)}.
\end{align*}In the case $s=0$,
\begin{align*}
    \left\|u_2\tilde{v}\right\|_{X^{\frac{1}{2},a}_\text{S}(T)} \lesssim \varepsilon\Big(\left\|\tilde{v}\right\|_{L^\infty_TL^2_x}+\left\|\tilde{v}\right\|_{L^4_{T,x}}+\left\|\tilde{u}\right\|_{X^{\frac{1}{2},b}_\text{S}(T)}\Big).
\end{align*}
By the estimate \eqref{trilineal},
\begin{align*}
\left\|\overline{u_1}\left(u_1+u_2\right)\tilde{u}\right\|_{X^{s+\frac{1}{2},a}_\text{S}(T)}+\left\|u_2^2\overline{\tilde{u}}\right\|_{X^{s+\frac{1}{2},a}_\text{S}(T)}\lesssim \varepsilon\Big(\left\|\tilde{u}\right\|_{X^{\frac{1}{2},b}_\text{S}(T)}+\left\|\tilde{u}\right\|_{X^{s+\frac{1}{2},b}_\text{S}(T)}\Big).
\end{align*}
Therefore, 
\begin{multline}
    \label{Tildeu}\left\|\tilde{u}\right\|_{X^{s+\frac{1}{2},b}_\text{S}(T)} \lesssim \left\|\tilde{u}(0)\right\|_{H^{s+\frac{1}{2}}}+\left\|\tilde{u}\right\|_{X^{\frac{1}{2},b}_\text{S}(T)}+\varepsilon\left\|\tilde{u}\right\|_{X^{s+\frac{1}{2},b}_\text{S}(T)}\\+\left\|\tilde{v}\right\|_{L^\infty_TL^2_x}+ \left\|\tilde{v}\right\|_{L^\infty_TH^s_x}+\left\|\tilde{v}\right\|_{L^4_{T,x}}+ \varepsilon\left\|\tilde{v}\right\|_{L^4_TW^{s,4}_x},
\end{multline}
and in the case $s=0$, 
\begin{align*}
    \left\|\tilde{u}\right\|_{X^{\frac{1}{2},b}_\text{S}(T)} &\lesssim \left\|\tilde{u}(0)\right\|_{H^\frac{1}{2}}+\varepsilon\big(\left\|\tilde{v}\right\|_{L^\infty_TL^2_x}+\left\|\tilde{v}\right\|_{L^4_{T,x}}+\left\|\tilde{u}\right\|_{X^{\frac{1}{2},b}_\text{S}(T)}\big)\\
    &\lesssim  \left\|\tilde{u}(0)\right\|_{H^\frac{1}{2}} + \varepsilon \left\|\tilde{u}\right\|_{X^{\frac{1}{2},b}_\text{S}(T)},
\end{align*}
where the last inequality is due to the estimates \eqref{estimateJs_1} and \eqref{estimateJs_2}. Hence, for all $\varepsilon>0$ sufficiently small, 
\begin{equation}
    \label{tildeu_0} \left\|\tilde{u}\right\|_{X^{\frac{1}{2},b}_\text{S}(T)} \lesssim \left\|\tilde{u}(0)\right\|_{H^\frac{1}{2}}. 
\end{equation}
Applying \eqref{tildeu_0} to \eqref{estimateJs_1} and \eqref{estimateJs_2}, we get 
\begin{equation}
    \label{all_0} \left\|\tilde{v}\right\|_{L^\infty_TL^2_x}, \left\|\tilde{v}\right\|_{L^4_{T,x}}, \left\|\tilde{u}\right\|_{X^{\frac{1}{2},b}_\text{S}(T)} \lesssim \left\|\tilde{u}(0)\right\|_{H^\frac{1}{2}} + \left\|\tilde{v}(0)\right\|_{L^2_x}. 
\end{equation} We now return to the general case $\left\|\tilde{v}\right\|_{L^\infty_TH^s_x}, \, \, \left\|\tilde{v}\right\|_{L^4_{T}W^{s,4}_x}, \, \, \text{and} \, \, \left\|\tilde{u}\right\|_{X^{s+\frac{1}{2},b}_\text{S}(T)}$: applying \eqref{all_0} on \eqref{estimativatildew} we get 
\begin{equation}
\label{mastilew}\left\|\tilde{w}\right\|_{Y^{s,\frac{1}{2}}(T)} \lesssim \left\|\tilde{v}(0)\right\|_{H^s}+\left\|\tilde{u}(0)\right\|_{H^\frac{1}{2}}+\varepsilon\big(\left\|\tilde{v}\right\|_{L^\infty_TH^s_x}+\left\|\tilde{v}\right\|_{L^4_TW^{s,4}_x}+\left\|\tilde{u}\right\|_{X^{s+\frac{1}{2},b}_\text{S}(T)}\big).
\end{equation}
Combining \eqref{all_0} with \eqref{mastilew} in \eqref{estimateJs}, we get 
\begin{multline*}
    \left\|\tilde{v}\right\|_{L^\infty_TH^s_x}+\left\|\tilde{v}\right\|_{L^4_{T}W^{s,4}_x} \lesssim  \left\|\tilde{v}(0)\right\|_{H^s}+\left\|\tilde{u}(0)\right\|_{H^\frac{1}{2}}\\+\varepsilon\big(\left\|\tilde{v}\right\|_{L^\infty_TH^s_x}+\left\|\tilde{v}\right\|_{L^4_TW^{s,4}_x}+\left\|\tilde{u}\right\|_{X^{s+\frac{1}{2},b}_\text{S}(T)}\big),
\end{multline*}
and then, for $\varepsilon>0$ sufficiently small, 
\begin{equation}
     \label{Jsall}\left\|\tilde{v}\right\|_{L^\infty_TH^s_x}+\left\|\tilde{v}\right\|_{L^4_{T}W^{s,4}_x} \lesssim  \left\|\tilde{v}(0)\right\|_{H^s}+\left\|\tilde{u}(0)\right\|_{H^\frac{1}{2}} + \varepsilon\left\|\tilde{u}\right\|_{X^{s+\frac{1}{2},b}_\text{S}(T)}.
\end{equation} Applying \eqref{all_0} and \eqref{Jsall} in \eqref{Tildeu}, we get 
\begin{equation*}
    \left\|\tilde{u}\right\|_{X^{s+\frac{1}{2},b}_\text{S}(T)} \lesssim \left\|\tilde{u}(0)\right\|_{H^{s+\frac{1}{2}}}+ \left\|\tilde{v}(0)\right\|_{H^s} + \varepsilon\left\|\tilde{u}\right\|_{X^{s+\frac{1}{2},b}_\text{S}(T)}.
\end{equation*}
For $\varepsilon>0$ sufficiently small, 
\begin{equation*}
    \left\|\tilde{u}\right\|_{X^{s+\frac{1}{2},b}_\text{S}(T)} \lesssim  \left\|\tilde{u}(0)\right\|_{H^{s+\frac{1}{2}}}+ \left\|\tilde{v}(0)\right\|_{H^s}.
\end{equation*}
Then, using \eqref{Jsall} and \eqref{mastilew}, we get the following estimate: 
\begin{equation}
\label{key diff}\left\|\tilde{v}\right\|_{L^\infty_TH^s_x}+\left\|\tilde{v}\right\|_{L^4_{T}W^{s,4}_x}+  \left\|\tilde{u}\right\|_{X^{s+\frac{1}{2},b}_\text{S}(T)} + \left\|\tilde{w}\right\|_{Y^{s,\frac{1}{2}}(T)}\lesssim  \left\|\tilde{u}(0)\right\|_{H^{s+\frac{1}{2}}}+ \left\|\tilde{v}(0)\right\|_{H^s}. 
\end{equation}
\subsection{Well-Posedness}
Let $\left(u_0,v_0\right)\in B_{\varepsilon_1} \cap \Big(H^{s+\frac{1}{2}}(\mathbb{R})\times H^s(\mathbb{R})\Big)$, with $s\geq0$. Consider the sequence $\left\{\left(u_{0,j},v_{0,j}\right)\right\}\subset H^\infty(\mathbb{R})\times H^\infty(\mathbb{R})$ defined in frequencies by 
\begin{equation}
    \widehat{u_{0,j}}(\xi)=\chi_{[-j,j]}(\xi)\widehat{u_0}(\xi), \, \, \text{and}, \, \, \widehat{v_{0,j}}(\xi)=\chi_{[-j,j]}(\xi)\widehat{v_0}(\xi). 
\end{equation}
Note that $\left(u_{0,j},v_{0,j}\right)\rightarrow(u_0,v_0)$ when $j\rightarrow\infty$. Applying Proposition \ref{Th1} to $\left(u_{0,j},v_{0,j}\right)$, and taking into account that $\left(u_0,v_0\right)\in B_{\varepsilon_1}$, there is any $T>0$ such that for each $j$ there is a unique solution $\left(u^j,v^j\right)\in C\left([0,T];H^\infty(\mathbb{R})\times H^\infty(\mathbb{R})\right)$ of the system \eqref{Schr-ILW-T}, with $\left(u^j(0),v^j(0)\right)=\left(u_{0,j},v_{0,j}\right)$. Moreover, since $P_{LO}v_0^j=P_{LO}v_0$ for all $j$ sufficiently big, by \eqref{key diff} we have 
\begin{multline}
    \left\|v^j-v^\ell\right\|_{L^\infty_TH^s_x}+\left\|v^j-v^\ell\right\|_{L^4_{T}W^{s,4}_x}+  \left\|u^j-u^\ell\right\|_{X^{s+\frac{1}{2},b}_\text{S}(T)}\\ + \left\|w^j-w^\ell\right\|_{Y^{s,\frac{1}{2}}(T)}\lesssim  \left\|u^j_0-u^\ell_0\right\|_{H^{s+\frac{1}{2}}}+  \left\|v^j_0-v^\ell_0\right\|_{H^s}.
\end{multline}
Therefore, the sequence converges strongly in $X^{s+\frac{1}{2}}_\text{S}(T)\times \left(L^\infty_T H^{s}_x \cap L^4_TW^{s,4}_x\right)$ to some function 
\begin{equation}
(u,v)\in C\Big([0,T];H^{s+\frac{1}{2}}(\mathbb{R})\times H^s(\mathbb{R})\Big). 
\end{equation}
Due to these strong convergences, from the estimates derived in the previous subsection it can be concluded that 
$(u,v)$ satisfies the system \eqref{Schr-ILW-T}. Moreover, from the conservation laws \eqref{C2}, \eqref{C3} and \eqref{C4}, when $(u_0,v_0)\in H^1(\mathbb{R})\times H^\frac{1}{2}(\mathbb{R})$ with a small enough norm, the solution $(u,v)$ can be extended for all time $T>0$.  
Uniqueness of solutions and continuity of the solution map for the system \eqref{Schr-ILW-T} follow from \eqref{key diff}. Finally, by undoing the Galilean transformation \eqref{galileana}, we obtain the well-posedness result of the system \eqref{Schr-ILW}. This concludes the proof of Theorem \ref{Th2}. \\ 
\begin{remark} \label{scalingch}
    We recall that the proof of the local and global well-posedness of the system \eqref{Schr-ILW-T} depends on the assumption of a small norm. Well-posedness for arbitrary large data is generally obtained by a scaling argument, but the system \eqref{Schr-ILW-T} does not enjoy this property. However, we can perform the following change of variables  
    \begin{equation*}
        u_\lambda(t,x):= \lambda u(\lambda^2t,\lambda x), \, \, \, v_\lambda(t,x):= \lambda v(\lambda^2t,\lambda x),
    \end{equation*}
    for all $0<\lambda\leq1$. Then $\left(u_\lambda,v_\lambda\right)$ is a solution of the following system:
    \begin{equation}
        \label{Schr-ILW-T-S} 
        \begin{cases}
             i\left(\partial_tu_\lambda-\nu\delta^{-1}\partial_xu_\lambda\right) + \partial_x^2u_\lambda= \lambda\alpha u_\lambda v_\lambda+\gamma\,|u_\lambda|^2u_\lambda \\ 
             \partial_tv_\lambda-\nu\,\mathscr{H}\partial_x^2v_\lambda+\rho\, v_\lambda\partial_xv_\lambda=\nu \mathscr{Q}_{\lambda\delta}\partial_xv_\lambda+\beta\left(|u_\lambda|^2\right)_x
        \end{cases}
        \,\,\,\,\,\,(t,x) \in \mathbb{R}^2.
    \end{equation} Since 
    \begin{equation*}
        \left\|u_\lambda(0,\cdot)\right\|_{H^{s+\frac{1}{2}}}\lesssim \lambda^\frac{1}{2}\Big(1+\lambda^{s+\frac{1}{2}}\Big)\left\|u(0)\right\|_{H^{s+\frac{1}{2}}},
    \end{equation*}
    and 
        \begin{equation*}
        \left\|v_\lambda(0,\cdot)\right\|_{H^{s}}\lesssim \lambda^\frac{1}{2}\left(1+\lambda^{s}\right)\left\|v(0)\right\|_{H^{s}},
    \end{equation*}
for all $\varepsilon_1>0$ we can always choose $\lambda$ small enough such that 
\begin{equation*}
    \left\|\left(u_\lambda(0,\cdot),v_\lambda(0,\cdot)\right)\right\|\leq \varepsilon_1.
\end{equation*}
Moreover, suppose that there exists $\varepsilon_1$ such that for all $0<\lambda\leq1$ there exists a solution $\left(u_\lambda,v_\lambda\right)$ of \eqref{Schr-ILW-T-S} with 
\begin{equation*}
    \left\|\left(u_\lambda(0,\cdot),v_\lambda(0,\cdot)\right)\right\|_{H^{s+\frac{1}{2}}\times H^s}\leq \varepsilon_1,
\end{equation*}
with time of existence $T_\lambda$. Then we obtain, letting 
\begin{equation*}
    \left(u(t,x),v(t,x)\right):=\left(\lambda^{-1}u_\lambda\left(\lambda^{-2}t,\lambda^{-1}x,\lambda^{-1}v\left(\lambda^{-2}t,\lambda^{-1}x\right)\right)\right)
\end{equation*}
a solution of \eqref{Schr-ILW-T} with time of existence $T\gtrsim\lambda^{2}T_\lambda$. Note that the condition of smallness of the solutions appearing in the apriori and nonlinear estimates must be independent of $\lambda\in(0,1)$. This idea was used by Zaiter for the Ostrovsky equation in \cite{zaiter2007remarks}.
\end{remark}
\section{Appendix: Energy Estimates}
\subsection{Energy Estimates for the Solutions of \eqref{Schr-ILW}}
Let $(u,v)\in C\big([0,T]\,;\,H^{s+\frac{1}{2}} (\mathbb{R})\times H^s(\mathbb{R})\big)$ be a solution of the system \eqref{Schr-ILW}, where $s> \frac{1}{2}$ and $T>0$. In this case, the usual energy $E^s_c (t)$ can be defined for all $t\in [0,T]$ by
 \begin{equation*}
    E^s_c(t) := \left\|u(t)\right\|_{L^2_x}^2+\big\|D^{s+\frac{1}{2}}_xu(t)\big\|_{L^2_x}^2+\left\|v(t)\right\|_{L^2_x}^2+\frac{1}{2}\left\|D^s_xv(t)\right\|_{L^2_x}^2.
\end{equation*}  For the solutions of \eqref{Schr-BO ext}, the energy estimates were worked in \cite{linares2024well}. Since $\mathscr{G}_\delta$ and $\mathscr{H}$ are skew-symmetric, in the proof of Proposition 3.1 of \cite{linares2024well}, $\mathscr{G}_\delta$ can take the role of $\mathscr{H}$. Therefore, the inequalities $(3.6)$, $(3.7)$, $(3.8)$,  $(3.9)$, $(3.10)$ and $(3.13)$ of \cite{linares2024well} are still valid for solutions of \eqref{Schr-ILW}. As a consequence, we have the following estimate:
\begin{equation}
    \label{classic energy} \frac{d}{dt}E_{c}^s (t) \lesssim \left\|v_x\right\|_{L^\infty_x}\left\|v\right\|_{H^s_x}^2+\left\|u\right\|_{H^{s+\frac{1}{2}}_x}^2\left\|v\right\|_{H^s_x}+\left\|u\right\|_{H^{s+\frac{1}{2}}_x}^4+ J + K,
 \end{equation}
 where 
 \begin{equation*}
     J:=\beta\int_{\mathbb{R}}D^s_x v \, \partial_x D^s_x\left(|u(t)|^2\right)\,dx \, \, \, \, \text{and} \, \, \, \,   K:= 2\alpha\,Im\int_{\mathbb{R}}u D^{s+1}_x\overline{u}\,D^s_x v\,dx.
 \end{equation*} The terms $J$ and $K$ in inequality \eqref{classic energy} do not seem to be easy to estimate. Instead, as in the proof of Proposition 3.1 of \cite{linares2024well}, we introduce an additional term in the usual energy from which $I$ and $II$ can be controlled. Indeed, applying $D^{s-\frac{1}{2}}_x$ to the second equation of \eqref{Schr-ILW}, multiplying by $D^{s-\frac{1}{2}}_x\left(|u|^2\right)$, and integrating over $\mathbb{R}$, we get 
\begin{multline*}
    \int_{\mathbb{R}}D^{s-\frac{1}{2}}_x\left(|u|^2\right) \partial_tD^{s-\frac{1}{2}}_x v\,dx = \nu\int_{\mathbb{R}}D^{s-\frac{1}{2}}_x\left(|u|^2\right) \mathscr{G}_\delta\partial_x^2D^{s-\frac{1}{2}}_x v\,dx\\-\frac{\rho}{2}\int_{\mathbb{R}}D^{s-\frac{1}{2}}_x\left(|u|^2\right)\partial_x D^{s-\frac{1}{2}}_x\left(v^2\right)\,dx,
\end{multline*}
which implies, 
\begin{equation}
    \label{multilinea}
    \begin{split}
        &\frac{d}{dt}\int_{\mathbb{R}}D^{s-\frac{1}{2}}_x\left(|u|^2\right) D^{s-\frac{1}{2}}_x v\,dx\\ = &\nu\int_{\mathbb{R}}D^{s-\frac{1}{2}}_x\left(|u|^2\right) \mathscr{G}_\delta\partial_x^2D^{s-\frac{1}{2}}_x v\,dx-\frac{\rho}{2}\int_{\mathbb{R}}D^{s-\frac{1}{2}}_x\left(|u|^2\right)\partial_x D^{s-\frac{1}{2}}_x\left(v^2\right)\,dx
    \\+ &\int_{\mathbb{R}}D^{s-\frac{1}{2}}_x\partial_t \left(|u|^2\right)D^{s-\frac{1}{2}}_x v\,dx := I + II + III. 
    \end{split}
\end{equation} With these considerations in mind, we define the modified energy $E^s_m(t)$ by \begin{equation*} E_{m}^s(t):= E^s_c(t)+\int_{\mathbb{R}} D^{s-\frac{1}{2}}_x v(t)\,D^{s-\frac{1}{2}}_x\left(|u(t)|^2\right)\,dx.
\end{equation*}Then, adding the inequalities \eqref{classic energy} and \eqref{multilinea}, we get 
\begin{equation}
    \label{modified energy} \frac{d}{dt}E^s_m(t) \lesssim \left\|v_x\right\|_{L^\infty_x}\left\|v\right\|_{H^s_x}^2+\left\|u\right\|_{H^{s+\frac{1}{2}}_x}^2\left\|v\right\|_{H^s_x}+\left\|u\right\|_{H^{s+\frac{1}{2}}_x}^4+ J + K + I + II + III. 
\end{equation}  Since $H^r(\mathbb{R})$ is a Banach algebra for all $r>{1/2}$, 
\begin{align*}
    II = -\frac{\rho}{2} \int_{\mathbb{R}} D^s_x(|u|^2)\partial_x D^{s-1}_x v\,dx \lesssim\|D^s_x(|u|^2)\|_{L^2_x}\|\partial_x D^{s-1}_x v\|_{L^2_x} \lesssim \|u\|_{H^s_x}^2\left\|v\right\|_{H^s_x}. 
\end{align*} On the other hand, using the first equation of \eqref{Schr-ILW}, 
\begin{align*}
    III &= 2 \text{Re} \int_{\mathbb{R}}D^{s-\frac{1}{2}}_x v\, D^{s-\frac{1}{2}}_x \left(\overline{u}u_t\right)\,dx\\
        &= -2\text{Im}\int_{\mathbb{R}} D^{s-\frac{1}{2}}_x v\, D^{s-\frac{1}{2}}_x \left(\overline{u}u_{xx}\right)\,dx\\
        &=-2\text{Im}\int_{\mathbb{R}} D^{s-\frac{1}{2}}_x v\, D^{s-\frac{1}{2}}_x \left(\overline{u}u_{xx}+\overline{u}_x u_x-|u_x|^2\right)\,dx\\
        &=-2\text{Im}\int_{\mathbb{R}} D^{s-\frac{1}{2}}_x v\, D^{s-\frac{1}{2}}_x\partial_x\left(\overline{u}u_{x}\right)\,dx =2\text{Im}\int_{\mathbb{R}} D^{s-\frac{1}{2}}_x \partial_x v\, D^{s-\frac{1}{2}}_x\left(\overline{u}u_{x}\right)\,dx.
\end{align*}Now, using $\partial_x=-D_x\mathscr{H}$,
\begin{align*}
    III &=-2\text{Im}\int_{\mathbb{R}} D^{s+\frac{1}{2}}_x \mathscr{H}v\, D^{s-\frac{1}{2}}_x\left(\overline{u}u_{x}\right)\,dx\\
        &=2\text{Im}\int_{\mathbb{R}} D^{s}_x v\, \mathscr{H}D^{s}_x\left(\overline{u}u_{x}\right)\,dx\\
        &=2\text{Im}\int_{\mathbb{R}} D^{s}_x v\, \left[\mathscr{H}D^{s}_x;\overline{u}\right]u_x\,dx+2\text{Im} \int_{\mathbb{R}}\left(D^s_x v\right)\overline{u}\mathscr{H}\partial_xD^s_xu\,dx\\
        &=2\text{Im}\int_{\mathbb{R}} D^{s}_x v\, \left[\mathscr{H}D^{s}_x;\overline{u}\right]u_x\,dx+2\text{Im} \int_{\mathbb{R}}\left(D^s_x v\right)\,\overline{u}\,D^{s+1}_xu\,dx := III_1+III_2.
\end{align*}
Making $\alpha=1$, we get $K+III_2=0$. Using Cauchy-Schwarz, Lemma \ref{LiLeibniz} and Sobolev embedding,
\begin{align*}
    III_1&\leq \left\|D^s_x v\right\|_{L^2_x}\left\|\left[\mathscr{H}D^{s}_x;\overline{u}\right]u_x\right\|_{L^2_x}\\
        &\lesssim \left\|D^s_x v\right\|_{L^2_x} \left\|D^{s-1}_x \partial_x\overline{u}\right\|_{L^{\infty-}_x}\left\|\partial_x u\right\|_{L^{2+}_x}\\
        &\lesssim \left\|D^s_x v\right\|_{L^2_x} \left\|D^{\frac{1}{2}+}D^{s-1}_x \partial_x\overline{u}\right\|_{L^{2}_x}\left\|\partial_x u\right\|_{L^{2+}_x}\lesssim \left\|v\right\|_{H^s_x} \left\|u\right\|_{H^{s+\frac{1}{2}}_x}^2.
\end{align*}
Finally, making $\beta=\nu$, we obtain 
\begin{align*}
    J+I &= \nu\int_{\mathbb{R}}D^s_x v \, \partial_x D^s_x\left(|u(t)|^2\right)\,dx+\nu\int_{\mathbb{R}}D^{s-\frac{1}{2}}_x\left(|u|^2\right) \mathscr{G}_\delta\partial_x^2D^{s-\frac{1}{2}}_x v\,dx\\
        &= \nu\int_{\mathbb{R}}D^{s+\frac{1}{2}}_x v \, \partial_x D^{s-\frac{1}{2}}_x\left(|u(t)|^2\right)\,dx-\nu\int_{\mathbb{R}}\partial_xD^{s-\frac{1}{2}}_x\left(|u|^2\right) \mathscr{G}_\delta\partial_x D^{s-\frac{1}{2}}_x v\,dx\\
        &= \nu\int_{\mathbb{R}}\partial_x D^{s-\frac{1}{2}}_x(|u(t)|^2)\big(D^{s+\frac{1}{2}}_x v - \mathscr{G}_\delta D^{s-\frac{1}{2}}_x \partial_xv\big)\, \,dx\\
        &\lesssim \|u\|_{H^{s+\frac{1}{2}}}^2\|F_\delta\|_{L^\infty(\mathbb{R})}\|v\|_{H^{s-\frac{1}{2}}_x},
\end{align*}
where
\begin{equation*}
    F_\delta(\xi):= |\xi|-\xi\coth(\delta\xi)+\frac{1}{\xi}, \, \, \,  \xi\in\mathbb{R}\setminus\{0\}.
\end{equation*} Now, we show that $F_\delta\in L^{\infty}(\mathbb{R})$. In fact, for all $\xi>0$, 
\begin{equation*}
    F_\delta(\xi)= \xi-\xi\coth(\delta\xi)+\frac{1}{\delta}, 
\end{equation*}
and then
\begin{equation*}
    F_\delta^{'}(\xi)= \frac{\left(4\delta\xi-2\right)e^{2\delta\xi}+2}{\left(e^{2\delta\xi}-1\right)^2}\geq0,
\end{equation*}
which implies that $F_\delta \nearrow$ in $(0,\infty)$. Furthermore, since $\left|F_\delta(\xi)\right|\leq|\xi|$, we get the limit $\lim_{\xi\rightarrow0}\,F_\delta(\xi)=0$. Then, $F_\delta\geq0$ in $(0,\infty)$. In fact, $F_\delta\geq0$ in $\mathbb{R}\setminus\{0\}$ because $F_\delta$ is even. Since $\lim_{\xi\rightarrow\infty}\,F_\delta(\xi)=\frac{1}{\delta}$, we conclude that $F_\delta\in L^{\infty}(\mathbb{R})$. Moreover, $\|F_\delta\|_{L^\infty(\mathbb{R})}=\frac{1}{\delta}$. Returning to the inequality \eqref{modified energy}, we obtain 
\begin{equation}
    \label{yacasi1}\frac{d}{dt}E^s_m(t) \lesssim \left\|v_x\right\|_{L^\infty_x}\left\|v\right\|_{H^s_x}^2+\left\|u\right\|_{H^{s+\frac{1}{2}}_x}^2\left\|v\right\|_{H^s_x}+\left\|u\right\|_{H^{s+\frac{1}{2}}_x}^4.
\end{equation} Therefore, we have established the following result: \\

\begin{proposition}
\label{cinco}    Let $s>\frac{1}{2}$ and $T>0$. In system \eqref{Schr-ILW}, suppose that $\alpha=1$ and $\beta=\gamma$. There exist constants $c_s,\kappa_{1,s}>0$ such that for any $(u,v)\in C\big([0,T]\,;H^{s+\,{1/2}}(\mathbb{R})\times H^{s}(\mathbb{R})\big)$ solution of \eqref{Schr-ILW} satisfying 
    \begin{equation}
        \label{peqsol}\left\|(u(t),v(t))\right\|_{H^{s+\frac{1}{2}}_x\times H^{s}_x} \leq {c_s}
    \end{equation}
    the following estimates hold true:
    \begin{equation}
        \label{coercivity} \frac{1}{2}\Big(\left\|u(t)\right\|_{H^{s+\frac{1}{2}}_x}+\left\|v(t)\right\|_{H^{s}_x}\Big)\leq  E_{m}^s(t)\leq \frac{3}{2}\Big(\left\|u(t)\right\|_{H^{s+\frac{1}{2}}_x}+\left\|v(t)\right\|_{H^{s}_x}\Big),
    \end{equation}
    and, 
        \begin{equation}
        \label{energy estimate}  \partial_tE_{m}^s(t)\lesssim \Big(1+\left\|\partial_xv(t)\right\|_{L^{\infty}_x}\Big)E_{m}^s(t).
    \end{equation}
Moreover,
\begin{equation}
    \label{consequence} \sup_{t\in [0,T]}\left\|(u(t),v(t))\right\|_{H^{s+\frac{1}{2}}_x\times H^{s}_x} \leq 2e^{\kappa_{1,s}\big(T+\left\|\partial_x v\right\|_{L^1_T\,L^{\infty}_x}\big)}\left\|\left(u(0),v(0)\right)\right\|_{H^{s+\frac{1}{2}}_x\times H^{s}_x}. 
\end{equation}
\end{proposition}
\subsection{Energy Estimates for the Difference of Two Solutions of \eqref{Schr-ILW}}
Let $(u_1,v_1)$ and $(u_2,v_2)$ two solutions of the system \eqref{Schr-ILW}. We define 
\begin{align*}
    & w:= u_1-u_2 \, \, \, \, \, u:= u_1+u_2 \\
    & z:= v_1-v_2 \, \, \, \, \, \, \, \, v:= v_1+v_2.
\end{align*}
Using $u_1=\frac{1}{2}(u-w)$, $u_2=\frac{1}{2}(u-w)$, $v_1=\frac{1}{2}(v+z)$ and $v_1=\frac{1}{2}(v-z)$, we find that $(w,z)$ is a solution of the following system:  
    \begin{equation}
        \label{Schr-ILW difference} 
        \begin{cases}
             iw_t + w_{xx}= \frac{1}{2} (vw+uz)+\gamma\left(\frac{1}{4}|w|^2w+\frac{1}{2}|u|^2w+\frac{1}{4}u^2\overline{w}\right)  \\ 
             z_t-\nu\,\mathscr{G}_{\delta}z_{xx}+\frac{\rho}{2}(vz)_x\,=\beta\,\partial_x\text{Re}(\overline{u}w)
        \end{cases}
        \,\,\,\,\,\,(t,x) \in \mathbb{R}^2.
    \end{equation}
The aim of this subsection is to derive energy estimates on $(w,z)$. As in the last subsection, let us define the following modified energies $(w,z)$: 
\begin{equation*}
    \tilde{E}_m^0(t):= \|w(t)\|_{L^2_x}^2+\big\|D^{1/2}_xw(t)\big\|_{L^2_x}^2+\frac{1}{2}\|z(t)\|_{L^2_x}^2+\text{Re}\int_\mathbb{R}J^{-1}_xz\left(\overline{u}w\right)(t)\,dx,
\end{equation*}
and 
\begin{multline*}
    \tilde{E}_m^s(t):= \|w(t)\|_{L^2_x}^2+\big\|D^{s+\,{1/2}}_xw(t)\big\|_{L^2_x}^2+\|z(t)\|_{L^2_x}^2+\frac{1}{2}\left\|D^s_xz\right\|_{L^2_x}^2\\+\text{Re}\int_\mathbb{R}D^{s-\,{1/2}}_xz(t)D^{s-\,{1/2}}_x\left(\overline{u}w\right)(t)\,dx,
\end{multline*}
for all $t\geq0$. In subsections \ref{eel0} and \ref{eels} we derive the desired estimates for $\tilde{E}_m^0(t)$ and $\tilde{E}_m^s(t)$, respectively.  

\subsubsection{\label{eel0}Energy Estimates at the level $s=0$} 
In  Proposition 3.2 of \cite{linares2024well}, the energy estimates for the difference of two solutions of \eqref{Schr-BO ext} were obtained. In fact, proceeding as in the proof of inequalities $(3.22)$, $(3.23)$, $(3.24)$,  $(3.25)$, $(3.26)$, $(3.27)$ and $(3.28)$ in \cite{linares2024well}, we obtain the following estimate:
\begin{equation}
    \label{classic tilde energy 0} \frac{d}{dt}\tilde{E}_{c}^0 (t) \lesssim \left\|v_x\right\|_{L^\infty_x}\left\|z\right\|_{H^s_x}^2+\left\|w\right\|_{H^{\frac{1}{2}}_x}^2+\left\|w\right\|_{H^\frac{1}{2}_x}\left\|z\right\|_{L^2_x}+ J + K,
 \end{equation}
 where 
 \begin{equation*}
     J:=\nu\text{Re}\int_\mathbb{R}z\left(\overline{u}w\right)_x\,dx,\, \, \, \text{and}\, \, \, K:=\text{Im}\int_\mathbb{R}\big(D^{1/2}_x\overline{w}\big)u\big(D^{1/2}_xz\big)\,dx.
 \end{equation*}
This terms cannot be estimated directly; instead, they will be compensated for by a correction coming from a modified energy term. We follow \cite{linares2024well}. In fact, note that 
\begin{equation*}
    J = \nu\int_\mathbb{R}\overline{u}w\mathscr{H}D_xz\,dx.
\end{equation*}
On the other hand, from the second equation of \eqref{Schr-ILW difference},
\begin{align*}
    &\partial_tz+\nu\mathscr{G}_\delta D^2_xz+\frac{\rho}{2}\left(vz\right)_x=\nu\partial_x\text{Re}\left(\overline{u}w\right)\\
    \implies &\partial_tJ^{-1}_xz+\nu\mathscr{G}_\delta D^2_xJ^{-1}_xz+\frac{\rho}{2}\partial_xJ^{-1}_x\left(vz\right)=\nu\text{Re}\partial_xJ^{-1}_x\left(\overline{u}w\right)\\
    \implies &\text{Re}\int_\mathbb{R}\overline{u}w\partial_tJ^{-1}_xz\,dx+\nu\text{Re}\int_\mathbb{R}\overline{u}w\mathscr{G}_\delta D^2_xJ^{-1}_xz\,dx+\frac{\rho}{2}\text{Re}\int_\mathbb{R}\overline{u}w\partial_xJ^{-1}_x(vz)\,dx=0.
\end{align*}
Therefore, 
\begin{multline}
    \label{0,v}\frac{d}{dt}\text{Re}\int_\mathbb{R}\overline{u}wJ^{-1}_xz\,dx = -\nu\text{Re}\int_\mathbb{R}\overline{u}w\mathscr{G}_\delta D^2_xJ^{-1}_xz\,dx-\frac{\rho}{2}\text{Re}\int_\mathbb{R}\overline{u}w\partial_xJ^{-1}_x(vz)\,dx\\+\text{Re}\int_\mathbb{R}\partial_t\left(\overline{u}w\right)J^{-1}_xz\,dx := II_1+II_2+II_3.
\end{multline}
Note that
\begin{align*}
J+II_1=\nu\text{Re}\int_\mathbb{R}\overline{u}w\left(\mathscr{H}D_x-\mathscr{G}_\delta D^2_xJ^{-1}_x\right)z\,dx
    &\leq |\nu|\left\|\overline{u}w\right\|_{L^2_x}\left\|D_x\left(\mathscr{H}-D_xJ^{-1}_x\mathscr{G}_\delta\right)z\right\|_{L^2_x}\\
    &\leq |\nu|\|u\|_{L^\infty_x}\|w\|_{L^2_x}\big\|\tilde{F}\big\|_{L^\infty}\|z\|_{L^2_x},
\end{align*}
where 
\begin{align*}
    \tilde{F}(\xi):= |\xi|\bigg|\frac{|\xi|}{\sqrt{1+\xi^2}}\Big(\coth{(\delta\xi)-\frac{1}{\delta\xi}}\Big)-\sgn(\xi)\bigg|, \, \,\,\, \xi\in \mathbb{R}. 
\end{align*}
Since $F$ is a even function, obseving that 
\begin{align*}
\lim_{\xi\rightarrow\infty}\big(\xi\coth{(\delta\xi)-\sqrt{1+\xi^2}}\big)
    = \lim_{\xi\rightarrow\infty}\xi\big(\coth{(\delta\xi)-1}\big)+\lim_{\xi\rightarrow\infty}\big(\xi-\sqrt{1+\xi^2}\big) =0,
\end{align*}
we get $\big\|\tilde{F}\big\|_{L^\infty}<\infty$. Hence, 
\begin{equation*}
    J+II_1\lesssim \|u\|_{L^\infty_x}\|z\|_{L^2_x}\|w\|_{L^2_x}.
\end{equation*}
Now, let us proceed again as in the proof of Proposition 3.2 in \cite{linares2024well}. Indeed, by Cauchy-Schwarz, 
\begin{equation*}
    II_2\lesssim \left\|\overline{u}w\right\|_{L^2_x}\left\|\partial_xJ^{-1}_x(vz)\right\|_{L^2_x}\lesssim \|u\|_{L^\infty_x}\|w\|_{L^2_x}\|v\|_{L^\infty_x}\|z\|_{L^2_x}. 
\end{equation*}
On the other hand, 
\begin{equation}
\label{porra}II_3=\text{Re}\int_\mathbb{R}\overline{u}_twJ^{-1}_xz\,dx+\text{Re}\int_\mathbb{R}\overline{u}w_tJ^{-1}_xz\,dx:= II_{3,1}+II_{3,2}.
\end{equation}
Using the first equation of \eqref{Schr-ILW} in $(u_1,v_1)$ and $(u_2,v_2)$, we get 
\begin{multline*}
II_{3,1}=\text{Im}\int_\mathbb{R}\overline{u}_{xx}wJ^{-1}_xz\,dx-\text{Im}\int_\mathbb{R}\left(\overline{u_1}v_1+\overline{u_2}v_2\right)wJ^{-1}_xz\,dx\\-\gamma\text{Im}\int_\mathbb{R}\left(|u_1|^2\overline{u_1}+|u_2|^2\overline{u_2}\right)wJ^{-1}_x\,dx:= II_{3,1,1}+II_{3,1,2}+II_{3,1,3}.
\end{multline*}
Using  $\partial_x\left(\overline{u}_xw\right)=\overline{u}_{xx}w+\overline{u}_xw_x$, we obtain 
\begin{align*}
    II_{3,1,1}&=\text{Im}\int_\mathbb{R}\partial_x\left(\overline{u}_xw\right)J^{-1}_xz\,dx-\text{Im}\int_\mathbb{R}\overline{u}_xw_xJ^{-1}_xz\,dx\\
    &= -\text{Im}\int_\mathbb{R}\overline{u}_xwJ^{-1}_x\partial_xz\,dx-\text{Im}\int_\mathbb{R}\overline{u}_xw_xJ^{-1}_xz\,dx:= II_{3,1,1,1}+II_{3,1,1,2}.
\end{align*}
Using Cauchy-Schwarz and Lemma \ref{LiLeibniz}, 
\begin{align*}
    II_{3,1,1,1}\leq \|\overline{u}_xw\|_{L^2_x}\|J^{-1}_x\partial_xz\|_{L^2_x}&\leq \big\|\partial_xu\big\|_{L^{2+}_x}\|w\|_{L^{\infty-}_x}\|z\|_{L^2_x}
    \lesssim \|u\|_{H^{s+\,{1/2}}_x}\|w\|_{H^{1/2}_x}\|z\|_{L^2_x}. 
\end{align*}
The contribution $II_{3,1,1,2}$ will be compensated by a term coming from $II_{3,2}$. Now, by the second equation of \eqref{Schr-ILW difference}, we obtain 
\begin{multline}
    \label{porra2} II_{3,2}=-\text{Im}\int_\mathbb{R}\overline{u}w_{xx}J^{-1}_xz\,dx+\frac{1}{2}\text{Im}\int_\mathbb{R}\overline{u}vwJ^{-1}_xz\,dx\\ +\gamma\text{Im}\int_\mathbb{R}\overline{u}\Big(\frac{1}{4}|w|^2w+\frac{1}{2}|u|^2w+\frac{1}{4}u^2\overline{w}\Big)J^{-1}_xz\,dx:= II_{3,2,1}+II_{3,2,2}+II_{3,2,3}.
\end{multline}
We decompose $II_{3,2,1}$ as follows: 
\begin{equation*}
II_{3,2,1}=\text{Im}\int_\mathbb{R}\overline{u}_xw_xJ^{-1}_xz\,dx-\text{Im}\int_\mathbb{R}\partial_x\left(\overline{u}w_x\right)J^{-1}_xz\,dx:= II_{3,2,1,1}+II_{3,2,1,2}.
\end{equation*}
Note that $II_{3,2,1,1}+II_{3,1,1,2}=0$. On the other hand, 
\begin{align*}
    II_{3,1,1,2} &= \text{Im}\int_\mathbb{R}\overline{u}w_xJ^{-1}_x\partial_xz\,dx\\
    &= -\text{Im}\int_\mathbb{R}\overline{u}\mathscr{H}D^{1/2}_xD^{1/2}_xwJ^{-1}_x\partial_xz\,dx\\
    &= \text{Im}\int_\mathbb{R}\big(\big[\mathscr{H}D^{1/2}_x;\overline{u}\big]D^{1/2}_xw \big)J^{-1}_xz\,dx-\text{Im}\int_\mathbb{R}\mathscr{H}D^{1/2}_x\big(\overline{u}D^{1/2}_xw\big)J^{-1}_x\partial_xz\,dx\\
    &= II_{3,1,1,2,1}+II_{3,1,1,2,2}.
\end{align*}
By Cauchy-Schwarz and Lemma \eqref{LiLeibniz},
\begin{align*}
    II_{3,1,1,2,1} \lesssim\|J^{-1}_x\partial_xz\|_{L^2_x}\|D^{-{1/2}}_x\partial_x\overline{u}\|_{L^\infty_x}\|D^{1/2}_xw\|_{L^2_x}
    \lesssim \|u\|_{H^{s+\frac{1}{2}}_x}\|z\|_{L^2_x}\|w\|_{H^\frac{1}{2}_x}.
\end{align*}
Additionally, 
\begin{align*}
    II_{3,1,1,2,2} = \text{Im}\int_\mathbb{R}\mathscr{H}D^\frac{1}{2}_x\Big(\overline{u}D^\frac{1}{2}_xw\Big)\mathscr{H}D_xJ^{-1}_x\,dx &= \text{Im}\int_\mathbb{R}\overline{u}D^\frac{1}{2}_xwD^\frac{3}{2}_xJ^{-1}_x\,dx\\
    &= -\text{Im}\int_\mathbb{R}\left(D^\frac{1}{2}_x\overline{w}\right)uD^\frac{3}{2}_xJ^{-1}_xz\,dx.
\end{align*}
Hence, 
\begin{align*}
K+II_{3,1,1,2,2}&=\text{Im}\int_\mathbb{R}\left(D^\frac{1}{2}_x\overline{w}\right)u\left(D^\frac{1}{2}_x-D^\frac{3}{2}_xJ^{-1}_x\right)z\,dx\\
    &\leq \|u\|_{L^\infty_x}\left\|D^\frac{1}{2}_xw\right\|_{L^2_x}\left\|\left(D^\frac{1}{2}_x-D^\frac{3}{2}_xJ^{-1}_x\right)z\right\|_{L^2_x}\lesssim \|u\|_{H^s_x}\|w\|_{H^\frac{1}{2}_x}\|z\|_{L^2_x},
\end{align*}
where the last inequality holds because the function 
\begin{equation*}
    G(\xi):= |\xi|^\frac{1}{2}-|\xi|^\frac{3}{2}\left(1+\xi^2\right)^{-\frac{1}{2}} = \sqrt{\frac{|\xi|}{1+\xi^2}}\frac{1}{\sqrt{1+\xi^2}+|\xi|} \in L^\infty_\xi.
\end{equation*}
 Finally,
\begin{align*}
    II_{3,2,2}\lesssim \|u\|_{L^\infty_x}\|v\|_{L^\infty_x}\|w\|_{L^2_x}\left\|J^{-1}_xz\right\|_{L^2_x}\lesssim\|u\|_{H^s_x}\|v\|_{H^s_x}\|w\|_{L^2_x}\|z\|_{L^2_x},
\end{align*}
and
\begin{align*}
    II_{3,2,2}\lesssim \left(\|u_1\|_{H^s_x}+\|u_2\|_{H^s_x}\right)^3\|w\|_{L^2_x}\|z\|_{L^2_x}.
\end{align*}
Therefore, we have proved the following result. \\
\begin{proposition}
  \label{meia}  Let $T>0$. There exists a positive constant $\tilde{c}_s$ such that for any $(u_1,v_1),(u_2,v_2)\in C\big([0,T];H^{s+{1/2}}(\mathbb{R})\times H^s(\mathbb{R})\big)$ solutions of \eqref{Schr-ILW} satisfying 
    \begin{equation*}
        \|(u_i(t),v_i(t))\|_{H^{s+{1/2}}_x\times H^s_x} \leq \tilde{c_s}, \, \, i\in\{1,2\},
    \end{equation*}
    the following estimates hold true:
    \begin{equation}
        \frac{1}{2}\big(\|w(t)\|^2_{H^{{1/2}}_x}+\|z(t)\|_{L^2_x}^2\big)\leq\tilde{E}_m^0(t)\leq\frac{3}{2}\big(\|w(t)\|^2_{H^{\frac{1}{2}}_x}+\|z(t)\|_{L^2_x}^2\big),
    \end{equation}
    and 
    \begin{equation}
        \frac{d}{dt}\tilde{E}^0_m(t)\lesssim \Big(1+\left\|\partial_xv_1(t)\right\|_{L^\infty_x}+\left\|\partial_xv_2(t)\right\|_{L^\infty_x}\Big)\tilde{E}^0_m(t).
    \end{equation}
\end{proposition}
\subsubsection{\label{eels}Energy Estimates at the level $s>\frac{1}{2}$}
Arguing as in the proof of inequalities $(3.42)$, $(3.43)$, $(3.44)$,  $(3.45)$, $(3.46)$, $(3.47)$, $(3.48)$, $(3.49)$, $(3.50)$ and $(3.51)$, we obtain 
\begin{equation}
    \label{classic tilde energy s} 
    \begin{split}
\frac{d}{dt}\tilde{E}_{c}^s (t) \lesssim &\left\|v_x\right\|_{L^\infty_x}\left\|z\right\|_{H^s_x}^2+\|z_x\|_{L^\infty_x}\|z\|_{H^s_x}+\|v_x\|_{L^\infty_x}\|w\|_{H^{s+\,{1/2}}_x}^2\\+&\left\|w\right\|_{H^{s+\,{1/2}}_x}^2+\left\|w\right\|_{H^{s+\,{1/2}}_x}\left\|z\right\|_{H^s_x}+ J + K,
    \end{split}
 \end{equation}
 where 
 \begin{equation*}
     J:=\nu\int_{\mathbb{R}}D^s_xz\,\partial_xD^s_x\left(\overline{u}w\right)\,dx,\, \, \, \text{and}\, \, \, K:=\text{Im}\int_{\mathbb{R}}uD^s_x\overline{w}D^s_xz\,dx.
 \end{equation*} The terms $J$ and $K$ cannot be handle directly and will be compensated with other terms coming from a modified term of the energy. Indeed, using that $\partial_x^2=-D_x^2$, applying $D^{s-\frac{1}{2}}_x$ to the second equation of \eqref{Schr-ILW difference}, multiplying by $D^{s-\frac{1}{2}}_x\left(\overline{u}w\right)$, integrating and taking the imaginary part,  
\begin{multline*}
    \text{Re}\int_{\mathbb{R}}D^{s-\frac{1}{2}}_x\left(\overline{u}w\right)\partial_tD^{s-\frac{1}{2}}_xz\,dx+\nu\int_{\mathbb{R}}D^{s-\frac{1}{2}}_x\left(\overline{u}w\right)\mathscr{G}_\delta D^{s+\frac{3}{2}}_xz\,dx\\+\frac{\rho}{2}\text{Re}\int_{\mathbb{R}}D^{s-\frac{1}{2}}\left(\overline{u}w\right)\partial_x D^{s-\frac{1}{2}}_x(vz)\,dx = \nu \text{Re}\int_{\mathbb{R}}D^{s-\frac{1}{2}}_x\left(\overline{u}v\right)\partial_xD^{s-\frac{1}{2}}_x(vz)\,dx, 
\end{multline*}
which implies 
\begin{multline}
    \label{5i} \frac{d}{dt}\text{Re}\int_{\mathbb{R}}D^{s-\frac{1}{2}}_x\left(\overline{u}w\right)D^{s-\frac{1}{2}}_xz\,dx = -\nu\text{Re}\int_\mathbb{R}D^{s}_x\left(\overline{u}w\right)\mathscr{G}_\delta D^{s+1}_x z\,dx\\ -\frac{\rho}{2}\text{Re}\int_\mathbb{R}D^{s-\frac{1}{2}}_x\left(\overline{u}w\right)\partial_x D^{s-\frac{1}{2}}_x(vz)\,dx\\ + \text{Re}\int_\mathbb{R}D^{s-\frac{1}{2}}_x\partial_t\left(\overline{u}w\right)D^{s-\frac{1}{2}}_xz\,dx:= IV_1+IV_2+IV_3. 
\end{multline}
Note that
\begin{align*}
    IV_1+II_3 &= \nu\, \text{Re}\int_\mathbb{R}D_x\left(\mathscr{H}-\mathscr{G}_\delta\right)D^{s}_x z\,D^s_x\left(\overline{u}w\right)dx\lesssim \|u\|_{H^s_x}\|z\|_{H^s_x}\|w\|_{H^s_x}.
\end{align*} Moreover,
\begin{align*}
    IV_2 &=\frac{\rho}{2}\text{Re}\int_\mathbb{R}D^s_x(\overline{u}w)\mathscr{H}D^s_x(vz)\,dx\lesssim \|u\|_{H^s_x}\|v\|_{H^s_x}\|z\|_{H^s_x}\|w\|_{H^s_x}.
\end{align*}
On the other hand, 
\begin{align*}
    IV_3 = \text{Re}\int_\mathbb{R}D^{s-\frac{1}{2}}_x\left(\overline{u}_tw\right)D^{s-\frac{1}{2}}_xz\,dx\,+\,\text{Re}\int_\mathbb{R}D^{s-\frac{1}{2}}_x\left(\overline{u}w_t\right)D^{s-\frac{1}{2}}_x\,dx:= IV_{3,1}+IV_{3,2}. 
\end{align*}
Using the first equation of \eqref{Schr-ILW} for $u_1$ and $u_2$, we get 
\begin{multline*}
    \label{4iii}IV_{3,1}= \text{Im}\int_\mathbb{R}D^{s-\frac{1}{2}}_x\left(\overline{u}_{xx}w\right)D^{s-\frac{1}{2}}_xz\,dx-\text{Im}\int_\mathbb{R}D^{s-\frac{1}{2}}_x\left(\overline{u_1}v_1+u_2v_2\right)D^{s-\frac{1}{2}}_xz\\-\gamma\text{Im}\int_\mathbb{R}D^{s-\frac{1}{2}}_x\left(|u_1|^2\overline{u_1}-|u_2|^2\overline{u_2}\right)D^{s-\frac{1}{2}}_xz\,dx:=IV_{3,1,1}+IV_{3,1,2}+IV_{3,1,3}.
\end{multline*} Taking into account that $\partial_x\left(\overline{u}_xw\right)=\overline{u}_xxw+\overline{u}_xw_x$, and using $\partial_x=-\mathscr{H}D_x$, we get
\begin{align*}
    IV_{3,1,1}&= -\text{Im}\int_\mathbb{R}\Big(\mathscr{H}D^{s}_x\left(\overline{u}_xw\right)D^{s}_xz+D^{s-\frac{1}{2}}_x\big(\overline{u}_xw_x\big)D^{s-\frac{1}{2}}_xz\Big)dx := IV_{3,1,1,1}+IV_{3,1,1,2}.
\end{align*}
Now, observing that 
\begin{align*}
    \overline{u}_xw=-\left(\mathscr{H}D^\frac{1}{2}_xD^\frac{1}{2}_x\overline{u}\right)w
    &=\mathscr{H}D^\frac{1}{2}_x\left(wD^\frac{1}{2}_x\overline{u}\right)-\left(\mathscr{H}D^\frac{1}{2}_xD^\frac{1}{2}_x\overline{u}\right)w-\mathscr{H}D^\frac{1}{2}_x\left(wD^\frac{1}{2}_x\overline{u}\right)\\
    &= \big[\mathscr{H}D^{1/2}_x;w\big]D^{1/2}_x\overline{u}-\mathscr{H}D^{1/2}_x\big(wD^{1/2}_x\overline{u}\big),
\end{align*}
we obtain 
\begin{align*}
&IV_{3,1,1,1}\\ = &-\text{Im}\int_\mathbb{R}D^{s+\,{1/2}}_x\big(wD^{1/2}_x\overline{u}\big)D^s_xz\,dx-\text{Im}\int_\mathbb{R}D^s_xz\,\mathscr{H}D^s_x\big[\mathscr{H}D^{1/2}_x;w\big]D^{1/2}_x\overline{u}\,dx\\
= &-\text{Im}\int_\mathbb{R}\big[D^{s+\,{1/2}}_x;w\big]D^{1/2}_x\overline{u}\,\,D^s_xz\,dx-\text{Im}\int_\mathbb{R}wD^{s+1}_x\overline{u}D^s_xz\,dx\\
&-\text{Im}\int_\mathbb{R}\mathscr{H}\big[\mathscr{H}D^{s+\,{1/2}}_x;w\big]D^{1/2}_x\overline{u}\,D^s_xz\,dx-\text{Im}\int_\mathbb{R}\mathscr{H}\left[D^s_x;w\right]\overline{u}_x\,D^s_xz\,dx\\
&:= IV_{3,1,1,1,1}+\cdots+IV_{3,1,1,1,4}.
\end{align*}
Using Cauchy-Schwarz, Lemma \ref{LiLeibniz} and Sobolev embedding, we get 
\begin{align*}
    IV_{3,1,1,1,1}+IV_{3,1,1,1,3}&\lesssim \left\|D^s_xz\right\|_{L^2_x}\big(\|D^{s+\,{1/2}}_xz\|_{L^2_x}\|D^{1/2}_x\overline{u}\|_{L^\infty_x}+\|w_x\|_{L^{2+}}\|D^s_x\overline{u}\|_{L^{\infty-}_x}\big)\\
    &\lesssim \|u\|_{H^{s+\frac{1}{2}}_x}^2\|w\|_{H^{s+\frac{1}{2}}_x} \left\|z\right\|_{H^s_x},
\end{align*}
\begin{align*}
IV_{3,1,1,1,2}\leq\left\|D^{s+1}_xu\right\|_{L^\infty_x}\|w\|_{H^{s+\frac{1}{2}}_x}\|z\|_{H^s_x},
\end{align*}
and 
\begin{align*}
    IV_{3,1,1,1,4} &\lesssim \left\|D^s_xz\right\|_{L^2_x}\left(\left\|D^s_xw\right\|_{L^{\infty-}_x}\|\overline{u}_x\|_{L^{2+}_x}+\|w_x\|_{L^{2+}_x}\left\|D^{s-1}_x\partial_x\overline{u}\right\|_{L^{\infty-}_x}\right)\\
    &\lesssim \|u\|_{H^{s+\frac{1}{2}}_x}\|w\|_{H^{s+\frac{1}{2}}_x}\|z\|_{H^s_x}.
\end{align*} Therefore, 
\begin{align*}
    IV_{3,1,1,1} \lesssim \left(\|D^{s+1}_xu\|_{L^\infty_x}+\|u\|_{H^{s+{1/2}}_x}\big(1+\|u\|_{H^{s+{1/2}}_x}\big)\right)\|w\|_{H^{s+{1/2}}_x}\|z\|_{H^s_x}.
\end{align*} The term $IV_{3,1,1,2}$ will be canceled with another term coming from $IV_{3,2}$. 
Using that $H^s(\mathbb{R})$ is a Banach algebra for all $s>\frac{1}{2}$, we get
\begin{equation*}
    IV_{3,1,2}\lesssim \|u\|^3_{H^s_x}\|z\|_{H^s_x}. 
\end{equation*}
Now, from the second equation of \eqref{Schr-ILW difference}, we get 
\begin{multline}
    \label{432} IV_{3,2}=-\text{Im}\int_\mathbb{R}D^{s-\,{1/2}}_x\left(\overline{u}w_{xx}\right)D^{s-\,{1/2}}_xz\,dx\\+\frac{1}{2}\text{Im}\int_\mathbb{R}D^{s-\,{1/2}}_x\left(vw+uz\right)D^{s-\,{1/2}}_xz\,dx\\+\gamma\text{Im}\int_\mathbb{R}D^{s-\,{1/2}}_x\Big(\frac{1}{4}|w|^2w+\frac{1}{2}|u|^2w+\frac{1}{4}u^2\overline{w}\Big)D^{s-\,{1/2}}_xz\,dx\\:= IV_{3,2,1}+IV_{3,2,2}+IV_{3,2,3}.
\end{multline}
Note that 
\begin{align*}
    IV_{3,2,1} &= -\text{Im}\int_\mathbb{R}D^{s-\frac{1}{2}}_x\partial_x\left(\overline{u}w_x\right)D^{s-\frac{1}{2}}_xz\,dx+\text{Im}\int_\mathbb{R}D^{s-\frac{1}{2}}_x\left(\overline{u}_xw_x\right)D^{s-\frac{1}{2}}_xz\,dx\\
    &= \text{Im}\int_\mathbb{R}\mathscr{H}D^s_x\left(\overline{u}w_x\right)D^s_xz\,dx+\text{Im}\int_\mathbb{R}D^{s-\frac{1}{2}}_x\left(\overline{u}_xw_x\right)D^{s-\frac{1}{2}}_xz\,dx:= IV_{3,2,1,1}+IV_{3,2,1,2}.
\end{align*}
Note that $IV_{3,2,1,2}+IV_{3,1,1,2}=0$. On the  other hand, 
\begin{align*}
    IV_{3,2,1,1}&=\text{Im}\int_\mathbb{R}\mathscr{H}\big[\mathscr{H}D^{s+\,{1/2}}_x;\overline{u}\big]D^{1/2}_xw\,D^s_xz\,dx+\text{Im}\int_\mathbb{R}\mathscr{H}[D^s_x;\overline{u}]w_x\,D^s_xz\,dx\\&+\text{Im}\int_\mathbb{R}\big[D^{s+\,{1/2}}_x;\overline{u}\big]D^{1/2}_xw\,D^s_xz\,dx+\text{Im}\int_\mathbb{R}\overline{u}D^{s+1}_xuD^s_xz\,dx\\
    &:= IV_{3,2,1,1,1}+\cdots+IV_{3,2,1,1,4}.
\end{align*}
Note that $IV_{3,2,1,1,4}+J=0$. On the other hand, 
\begin{align*}
    IV_{3,2,1,1,1}+IV_{3,2,1,1,3} &\lesssim \left\|D^s_xz\right\|_{L^2_x}\big(\big\|D^{s+{1/2}}_xu\big\|_{L^2_x}\big\|D^{1/2}_xw\big\|_{L^\infty_x}+\left\|D^s_xw\right\|_{L^{\infty-}_x}\|u_x\|_{L^{2+}_x}\big)\\&\leq\|u\|_{H^{s+{1/2}}_x}\|z\|_{H^s_x}\|w\|_{H^{s+{1/2}}_x}.
\end{align*}
Moreover,
\begin{align*}
    IV_{3,2,1,1,2}&\lesssim \left\|D^s_xz\right\|_{L^2_x}\left(\left\|D^s_xu\right\|_{L^{\infty-}_x}\|w_x\|_{L^{2+}_x}+\|u_x\|_{L^{2+}_x}\left\|D^s_xw\right\|_{L^{\infty-}_x}\right)\\&\lesssim \|u\|_{H^{s+{1/2}}_x}\|z\|_{H^s_x}\|w\|_{H^{s+{1/2}}_x}. 
\end{align*}
\begin{proposition}
    \label{sete}Let $s>\frac{1}{2}$ and $T>0$. There exists a positive constant $\tilde{c}_s$ such that for any $(u_1,v_1),(u_2,v_2)\in C\big([0,T];H^{s+\,{1/2}}(\mathbb{R})\times H^s(\mathbb{R})\big)$ solutions of \eqref{Schr-ILW} satisfying 
    \begin{equation*}
        \|(u_i(t),v_i(t))\|_{H^{s+\,{1/2}}_x\times H^s_x} \leq \tilde{c_s}, \, \, i\in\{1,2\},
    \end{equation*}
    the following estimates hold true:
    \begin{equation}
        \frac{1}{2}\big(\|w(t)\|^2_{H^{s+\,{1/2}}_x}+\|z(t)\|_{H^s_x}^2\big)\leq\tilde{E}_m^0(t)\leq\frac{3}{2}\big(\left\|w(t)\right\|^2_{H^{s+\, {1/2}}_x}+\|z(t)\|_{H^s_x}^2\big),
    \end{equation}
    and 
    \begin{equation}
        \frac{d}{dt}\tilde{E}^sm(t)\lesssim (1+\|\partial_xv_1(t)\|_{L^\infty_x}+\|\partial_xv_2(t)\|_{L^\infty_x})\tilde{E}^0_m(t)+f_s(t),
    \end{equation}
    where $f_s=f_s(t)$ is defined by 
    \begin{multline}
        f_s(t)= \|D^s_x\partial_xv\|_{L^\infty}\|D^s_xz\|_{L^2}\|z\|_{L^2}+\|D^s_xv\|_{L^2_x}\|D^s_xz\|_{L^2}\|\partial_xz\|_{L^\infty}\\
        +\left\|D^{s+\,{1/2}}_xv\right\|_{L^\infty}\|w\|_{L^2}\|D^{s+\,{1/2}}_xw\|_{L^2}+\left\|D^{s+1}u\right\|_{L^\infty}\|w\|_{L^2}\left\|D^s_xz\right\|_{L^2}. 
    \end{multline}
\end{proposition}
\begin{remark}
    The terms gathered in $f_s(t)$ cannot be estimated directly, but they always have more derivatives on the functions $u_i,v_i$ than on the terms for $w$ and $z$. These terms will be handled with the Bona-Smith argument as in Proposition 2.18 in \cite{linares2014dispersive}. 
\end{remark}

\subsection{Well-posedness for Smooth Initial Data}

The following result is based on energy estimates derived in the last subsections. \\

\begin{proposition}
        \label{Th1} Let $s>\frac{3}{2}$. In the system \eqref{Schr-ILW}, suppose that $\alpha=1$ and $\beta=\nu$. For any $(u_0,v_0)\in H^{s+\frac{1}{2}}(\mathbb{R})\times H^s(\mathbb{R})$, there exist a positive time $T=T\big(\left\|(u_0,v_0)\right\|_{H^{s+\frac{1}{2}}\times H^s}\big)$ and a unique maximal solution $(u,v)$ of the system \eqref{Schr-ILW} in $C\left([0,T^*],\,H^{s+\,{1/2}}(\mathbb{R})\times H^s(\mathbb{R})\right)$, with $T^*>T$, such that $(u(0),v(0))=(u_0,v_0)$. If the maximal time of existence $T^*$ is finite, then $$\lim_{t\uparrow T^*}\|(u(t),v(t))\|_{H^{s+\frac{1}{2}}\times H^s}=+\infty.$$
        Moreover, for any $0<T'<T$, there exists a neighborhood $\mathscr{U}$ of $(u_0,v_0)$ in $H^{s+\frac{1}{2}}(\mathbb{R})\times H^s(\mathbb{R})$ such that the flow map data to solution 
        $$\mathcal{S}: \mathscr{U}\mapsto C\big([0,T^*],\,H^{s+\,{1/2}}(\mathbb{R})\times H^s(\mathbb{R})\big), (\tilde{u}_0,\tilde{v}_0)\mapsto(\tilde{u},\tilde{v})$$
        is continuous. 
    \end{proposition}
\begin{proof}[Proof]
    The proof of the existence and the uniqueness is a combination of the parabolic regularization method with the energy estimates obtained in Propositions \ref{sete}, \ref{meia} and \ref{cinco}. We refer to the proof of Theorem 1.6 in \cite{paulsen2022long} for the details in a similar setting. The continuous dependence and persistence is obtained by applying the Bona-Smith approximation method. We refer to \cite{bona1975initial}, \cite{iorio1986cauchy}, \cite{linares2015introduction} and \cite{paulsen2022long} for the details. 
\end{proof}


    \bibliographystyle{plainnat}
    \bibliography{references}

\end{document}